\documentclass{article}
\usepackage{amscd}
\usepackage{amsfonts}
\usepackage{amsmath}
\usepackage{mathrsfs}
\usepackage{amssymb}

\usepackage{amsthm}
\usepackage{mathtools}
\usepackage{MnSymbol}
\usepackage{multirow,bigdelim}
\usepackage{eucal}
\usepackage{url}
\usepackage{bm}
\usepackage[dvipdfmx]{hyperref}
\usepackage[dvipdfmx]{graphicx}
\usepackage[all]{xy}
\usepackage{color}
\usepackage{xcolor}
\usepackage{tikz}
\usetikzlibrary{arrows}
\usetikzlibrary{cd}
\usetikzlibrary{patterns}

\DeclareMathOperator{\diag}{diag}
\DeclareMathOperator{\Dyn}{Dyn}

\DeclareMathOperator{\id}{id}
\DeclareMathOperator{\inv}{inv}
\DeclareMathOperator{\Irr}{Irr}

\DeclareMathOperator{\Res}{Res}

\DeclareMathOperator{\Spec}{Spec}
\DeclareMathOperator{\type}{type}
\DeclareMathOperator{\sep}{sep}
\DeclareMathOperator{\CAlg}{CAlg}

\DeclareMathOperator{\Gal}{Gal}
\DeclareMathOperator{\GL}{GL}

\DeclareMathOperator{\SL}{SL}
\DeclareMathOperator{\SO}{SO}
\DeclareMathOperator{\Sp}{Sp}
\DeclareMathOperator{\SU}{SU}
\DeclareMathOperator{\Oo}{O}
\DeclareMathOperator{\Uu}{U}
\newcommand{\Set}{\mathcal{S}\mathrm{et}}
\newcommand{\Grp}{\mathcal{G}\mathrm{rp}}
\newcommand{\bC}{\mathbb{C}}
\newcommand{\bG}{\mathbb{G}}
\newcommand{\bN}{\mathbb{N}}
\newcommand{\bP}{\mathbb{P}}
\newcommand{\bQ}{\mathbb{Q}}
\newcommand{\bR}{\mathbb{R}}
\newcommand{\bZ}{\mathbb{Z}}
\newcommand{\cE}{\mathcal{E}}
\newcommand{\cF}{\mathcal{F}}

\newcommand{\cL}{\mathcal{L}}
\newcommand{\cO}{\mathcal{O}}
\newcommand{\cP}{\mathcal{P}}
\newcommand{\fg}{\mathfrak{g}}
\newcommand{\fh}{\mathfrak{h}}
\newcommand{\fq}{\mathfrak{q}}
\newcommand{\sD}{\mathscr{D}}

\theoremstyle{plain}
\newtheorem{thm}{Theorem}[subsection]
\newtheorem{cor}[thm]{Corollary}
\newtheorem{lem}[thm]{Lemma}
\newtheorem{prop}[thm]{Proposition}
\newtheorem{thmA}{Theorem}
\newtheorem{thmB}{Theorem}
\newtheorem{corC}{Corollary}
\newtheorem{corD}{Corollary}
\newtheorem{thmE}{Theorem}
\newtheorem{thmF}{Theorem}
\newtheorem{corG}{Corollary}

\theoremstyle{definition}

\newtheorem{cons}[thm]{Construction}
\newtheorem{defn}[thm]{Definition}
\newtheorem{ex}[thm]{Example}

\newtheorem{rem}[thm]{Remark}
\begin{document}
	\title{Half-integrality of line bundles on partial flag schemes of classical Lie groups}
	\author{Takuma Hayashi\thanks{Department of Pure and Applied Mathematics, Graduate School of Information Science and Technology, Osaka University, 1-5 Yamadaoka, Suita, Osaka 565-0871, Japan, hayashi-t@ist.osaka-u.ac.jp}}
	
	\date{}
	\maketitle
	\begin{abstract}
		
		In this paper, we develop a theory of Galois descent for equivariant line bundles on partial flag schemes. In particular, we study computational aspects of the classification of descent data of equivariant line bundles attached to characters of parabolic subgroups. As an application, we classify equivariant line bundles on partial flag schemes of the standard $\bZ\left[1/2\right]$-forms of classical Lie groups.
		
	\end{abstract}
	
	\section{Introduction}
	
	In representation theory of real reductive Lie groups and Lie algebras, we can find phenomena of descent. That is, a representation over the field $\bC$ of complex numbers is sometimes defined over the field $\bR$ of real numbers or its smaller subfields. Even we can find $k$-forms for subrings $k\subset\bC$. The aim of this paper is to study its geometric counterpart. More specifically, we work with the descent problem of equivariant line bundles on partial flag schemes. From the perspectives of abstract theory, this is an analog of \cite{MR1324207} Proposition 2.4. We also develop its computational aspects. 
	
	\subsection{Finite dimensional representations}
	The most trivial example of complex representations with real forms is that the trivial representation of a real group or a real Lie algebra. Another typical example is the adjoint representation. When it comes to a result for a specific group, the case of the compact Lie group $\SO(3)$ is easiest. In fact, the complex irreducible representations of $\SO(3)$ are given by the degree $n$ part of the graded commutative algebra $\bC\left[x,y,z\right]/(x^2+y^2+z^2)$ with $\deg x=\deg y=\deg z=1$, where $n$ runs through nonnegative integers. We obtain their real forms just by replacing $\bC$ with $\bR$. The conclusion is that every complex irreducible representation of $\SO(3)$ is defined over $\bR$. This realization even tells us that these representations are defined over $\bZ\left[1/2\right]$ with respect to the standard $\bZ\left[1/2\right]$-form of $\SO(3)$ as a group scheme, where $\bZ$ is the ring of integers. In \cite{E1914}, Cartan sorted complex irreducible finite dimensional representations of real Lie algebras $\fg$ into three types to establish a general classification scheme of real irreducible finite dimensional representations of $\fg$. He also reduced his scheme to the case of complex finite dimensional representations with fundamental highest weight when $\fg$ are semisimple, and wrote down the complete result of his scheme explicitly for each real simple Lie algebra. In \cite{MR102534}, Iwahori reestablished the results of \cite{E1914}. Using maximally split Cartan subalgebras, Fell gave a complete and explicit classification theorem on finite dimensional real irreducible representations of connected (simply connected semisimple) Lie groups in \cite{MR0209401}. Takeuchi translated Fell's work into purely Lie algebra theoretic arguments to classify real irreducible representations of real Lie algebras in \cite{takeuchi} Section 7 -- 13. Note that in Section 11 (resp.\ Section 12), he discussed the compact case, based on Malcev--Dynkin's method \cite{MR1611385} (resp.\ Fell--Dadok's method \cite{MR0209401} and \cite{MR773051}). Subsequently, Onishchik classified finite dimensional real irreducible representations of real semisimple Lie algebras in \cite{MR2041548} Section 8 by using fundamental Cartan subalgebras. For a connected reductive algebraic group $G$ over a field $F$ of characteristic zero and an irreducible representation $V'$ of $G$ over the algebraic closure of $F$, Borel and Tits constructed an obstruction class $\beta_{V'}$ to the rationality (i.e., existence of the datum of Galois descent) of $V'$ in \cite{MR207712} Section 12. We can regard that Fell and Onishchik gave ways to compute $\beta_{V'}$ for $F=\bR$. 
	
	Merkurjev and Tignol developed a geometric analog of the work of Borel and Tits in \cite{MR1324207}. The Borel--Weil--Bott theorem asserts that every irreducible representation of a complex reductive algebraic group appears as a cohomology of an equivariant line bundle on the flag variety. Moreover, this theorem tells us what the cohomology modules of line bundles on partial flag varieties are. Note that for a connected complex reductive algebraic group $G$ and a parabolic subgroup $P$, the $G$-equivariant line bundles on $G/P$ are the associated bundles $\cL_\lambda$ attached to characters $\lambda$ of $P$. However, this picture is not enough to geometrize the descent phenomenon of representations of $\SO(3)$ since $\SO(3)$ does not admit proper parabolic subgroups. They promote the above idea by using partial flag varieties (schemes) without base points which were introduced in \cite{MR0218364} (see also \cite{MR207712} 5.24 for the case that the base ring is a field). For example, we can identify closed points of the complex flag variety of a given complex reductive group with Borel subgroups. More strongly, the flag variety is the moduli space of Borel subgroups which literally makes sense over general bases even if the given reductive group scheme has no Borel subgroups (see \cite{MR0218363} Corollaire 5.8.3). For a general partial flag scheme, we need to discuss whether the corresponding conjugacy class of parabolic subgroups (called a type in \cite{MR0218364}) is defined over the base. In fact, the partial flag schemes are fibers of the morphism $t$ in \cite{MR0218364} Section 3.2. Merkurjev and Tignol constructed a morphism from a partial flag variety to a Severi-Brauer variety over the base field (\cite{MR1324207} Proposition 2.4). Moreover, they obtained line bundles on partial flag varieties over the base field by the pullback of the tautological bundle on the Severi-Brauer variety when the obstruction class is trivial. We also note that they considered the corresponding division algebras rather than the cocycle.
	
	The aim of this paper is to see the converse direction of the work of Merkurjev and Tignol over (Noetherian) commutative rings. Namely, we see that the line bundle on partial flag schemes admits a datum of Galois descent if and only if the corresponding cocycle vanishes in the Galois cohomology. Moreover, we classify forms of line bundles attached to a character of a parabolic subgroup. We thus also obtain representations of reductive groups over commutative rings (for example, $\bZ\left[1/2\right]$). To work with partial flag schemes, we also discuss a computational criterion to the descent problem of parabolic types attached to parabolic subgroups. We adopt a more direct approach to the descent problem in this paper than \cite{MR1324207} to give a numerical formula to the obstruction class.

	\subsection{$A_\fq(\lambda)$-modules}
	
	The descent phenomena of Harish-Chandra modules over fields of characteristic zero were studied in \cite{MR3770183} Section 3, 5, 6, 7. Especially, such arithmetic properties of Harish-Chandra modules are expected to be applicable to rationality and integrality of special $L$-values (\cite{MR3053412}, \cite{10.1093/imrn/rny043}, \cite{MR3770183}, \cite{MR3937337}, \cite{1604.04253}, \cite{MR3970997}).
	
	In geometric representation theory of real reductive Lie groups, most of complex irreducible Harish-Chandra modules are obtained by the localization and operations of twisted D-modules. For example, the $A_\fq(\lambda)$-modules are obtained as the space of global sections of the twisted D-module theoretic direct image of the line bundle on the closed $K$-orbit of the partial flag variety corresponding to the complex conjugate of $\fq$ (see \cite{MR910203} 4.3. Theorem and \cite{MR2945222} Theorem 5.1, Corollary 5.5).
	
	Fabian Januszewski proposed the idea that the smaller rings the geometric objects appearing in these constructions of representations are defined over, the smaller rings the resulting representations are correspondingly defined over. We follow his idea towards a geometric construction of nontrivial models of $A_\fq(\lambda)$-modules in the present paper and \cite{hayashijanuszewski}. That is, we developed the general theory of twisted D-modules over general base schemes in \cite{hayashijanuszewski} Section 1 -- 4. We also study the descent problem of rings of definition of closed $K$-orbits on partial flag schemes attached to models of $\theta$-stable parabolic subgroups in \cite{hayashijanuszewski} Section 5, where $\theta$ is (the model of) the Cartan involution. In this paper, we study nontrivial rings of definition of equivariant line bundles on partial flag schemes. As a result, we obtain tdos on partial flag schemes over smaller rings. Applying an analog of the geometric construction of $A_\fq(\lambda)$-modules over the complex numbers to our forms of tdos and closed $K$-orbits under our general theory of twisted D-modules, we obtain nontrivial models of $A_\fq(\lambda)$-modules (see \cite{hayashijanuszewski} Section 6 for details).
	
	To work with other kinds of representations by localization, the author works on $\bZ\left[1/2\right]$-forms of the $\SO(3,\bC)$-orbit decomposition of the complex flag variety of $\SL_3$ in \cite{hayashikgb}. A further generalization is in progress.
	
	\subsection{Main Results}
	
	In this paper, we develop a general theory of descent of rings of definition of line bundles on partial flag schemes. To make computations of examples easy, we restrict ourselves to the case of Galois descent:
	
	\begin{defn}[\cite{MR4225278} (14.20)]
		We say a finite faithfully flat homomorphism $k\to k'$ of commutative rings, equipped with an action of a finite group $\Gamma$ on $k'$ over $k$ is a Galois extension of Galois group $\Gamma$ if the map $k'\otimes_k k'\to\prod_{\sigma\in\Gamma} k';~a\otimes b\mapsto (a\sigma(b))$ is an isomorphism.
	\end{defn}
	
	When $\Gamma\cong\bZ/2\bZ$, we will denote the nontrivial element of $\Gamma$ by $\bar{\ }$. Our main example for $\Gamma\cong\bZ/2\bZ$ will be $\bZ\left[1/2\right]\subset\bZ\left[1/2,\sqrt{-1}\right]$.
	
	Let $k\to k'$ be a Galois extension of commutative rings or a (possibly infinite) Galois extension of fields of Galois group $\Gamma$. Suppose that we are given a reductive group scheme $G$ over $k$ and a character $\lambda$ of a parabolic subgroup $P'\subset G\otimes_k k'$. We wish to judge the existence of the $k$-forms of $(G\otimes_k k')/P'$ and the line bundle $\cL_\lambda\coloneqq (G\otimes_k k')\times^{P'}\lambda$ on $(G\otimes_k k')/P'$.
	
	To explain our main results, here we assume the following conditions for simplicity:
	
	\begin{enumerate}
		\renewcommand{\labelenumi}{(\roman{enumi})}
		\item $k$ and $k'$ are Noetherian;
		\item $\Spec k'$ is connected;
		\item We are given a maximal torus $H$ of $G$ such that $(G\otimes_k k',H\otimes_k k')$ is split;
		\item $P'$ contains the Borel subgroup $B'\subset G\otimes_k k'$ attached to a simple system $\Pi$ of $(G\otimes_k k',H\otimes_k k')$. 
	\end{enumerate}
	
	Write $\Pi'$ for the corresponding subset of $\Pi$. In the rest, write ${}^\sigma(-)$ for the Galois twist by $\sigma\in\Gamma$. Fix an element $w_\sigma$ of the normalizer subgroup of $H\otimes_k k'$ in $G(k')$ such that ${}^\sigma\Pi=w_\sigma\Pi$.
	
	To classify partial flag schemes over $k$, it will suffice to compute the Dynkin scheme $\Dyn G$ by definitions of $\Dyn G$ and of the scheme $\type G$ of parabolic types of $G$. We give an explicit description of $\Dyn G$ in Theorem \ref{thm:dynkin}. It leads to the result below:
	
	\begin{thmA}[Theorem \ref{thm:dynkin}, Example \ref{ex:settingsection4}]
		The following conditions are equivalent:
		\begin{enumerate}
			\renewcommand{\labelenumi}{(\alph{enumi})}
			\item The parabolic type of $P'$ is defined over $k$.
			\item We have ${}^\sigma P'=w_\sigma P' w^{-1}_\sigma$ for all $\sigma\in\Gamma$.
			\item We have ${}^\sigma \Pi'=w_\sigma \Pi'$ for all $\sigma\in\Gamma$.
		\end{enumerate}
		Moreover, every parabolic type of $G$ over $k$ is obtained as the type of such $P'$.
	\end{thmA}
	
	In the rest, assume the above equivalent conditions. Therefore we obtain the $k$-form of $(G\otimes_k k')/P'$ which we will denote by $\cP_{G,x}$\footnote{The subscript $x$ indicates the type of $G$ over $k$ attached to $P'$.} in this section.
	
	\begin{thmB}[Theorem \ref{thm:descentofllambda}, Theorem \ref{thm:uniqueness}, Remark \ref{rem:associatedbundle}]\label{thmB}
		\begin{enumerate}
			\renewcommand{\labelenumi}{(\arabic{enumi})}
			\item Define
			\[\beta_\lambda\in H^2(\Gamma,(k')^\times)\]
			by 
			\[\beta_{\lambda}(\sigma,\tau)=
			\lambda(w_{\sigma\tau}^{-1}\sigma(w_\tau)w_\sigma).\]
			Then $\cL_\lambda$ admits a datum of Galois descent if and only if the following conditions are satisfied:
			\begin{enumerate}
				\item[(i)] ${}^\sigma\lambda=w_\sigma\lambda$ for $\sigma\in\Gamma$.
				\item[(ii)] $\beta_{\lambda}$ is trivial as a cohomology class.
			\end{enumerate}
			\item The set of isomorphism classes of descent data on $\cL_\lambda$ is a (possibly empty) principal $H^1(\Gamma,(k')^\times)$-set.
			\item If $k'$ is a PID, the equivariant line bundles on $\cP_{G,x}$ are obtained by the Galois descent of the associated bundles $\cL_\lambda$.
		\end{enumerate}
	\end{thmB}
	
	The ``if'' direction of (1) is due to \cite{MR1324207} Proposition 2.4 when $k$ is a field. Let us remark however that we impose our numerical formula of $\beta$ in this paper rather than the corresponding central division algebra by thinking of the Galois cohomology. For the proof of Theorem \ref{thmB}, we see by a minor modification of \cite{MR0209401} Section 3 Corollary 1 that $\cL_\lambda$ is isomorphic to its Galois twists if and only if (i) is satisfied (cf.\ \cite{MR1324207} Proposition 2.2). On this course, we construct ${}^\sigma \cL_\lambda\cong \cL_\lambda$ by using $w_\sigma$ under the condition (i) of (1). We then obtain an obstruction class to the existence of descent data by a similar argument to \cite{MR207712} Section 12. We can compute this class explicitly by our construction of ${}^\sigma \cL_\lambda\cong \cL_\lambda$. The resulting cocycle is $\beta_{\lambda}$ defined in the above.
	
	Notice that the type of $B'$ is defined over $k$. Put $P'=B'$ and take the global sections to obtain a consequence in representation theory:
	
	\begin{corC}[Theorem \ref{thm:descentofrep}]\label{corC}
		Suppose that $k,k'$ are fields of characteristic zero. Let $V'$ be an irreducible representation of $G\otimes_k k'$ with lowest weight $\lambda$. Suppose that ${}^\sigma\lambda=w_\sigma\lambda$ for every $\sigma\in\Gamma$. Then we have $\beta_{V'}=\beta_{\lambda}$. In particular, the following conditions are equivalent:
		\begin{enumerate}
			\renewcommand{\labelenumi}{(\alph{enumi})}
			\item $V'$ admits a $k$-form;
			\item
			\begin{enumerate}
				\item[(b-i)] ${}^\sigma\lambda=w_\sigma\lambda$ for $\sigma\in\Gamma$;.
				\item[(b-ii)] $\beta_{\lambda}$ is trivial as a cohomology class.
			\end{enumerate}
			\item $(G\otimes_k k')\times^{B'} \lambda$ admits a $k$-form.
		\end{enumerate}
		Moreover, if these equivalent conditions are satisfied, a $k$-form of $V'$ is realized as the space of global sections of the $k$-form of $(G\otimes_k k')\times^{B'} \lambda$.
	\end{corC}
	
	From the abstract perspectives, the implications (a) $\iff$ (b) $\Rightarrow$ (c) were proved in \cite{MR0209401} and \cite{MR1324207} respectively. Let us however remark that our numerical description of the obstruction class $\beta_{\lambda}=\beta_{V'}$ in Theorem \ref{thmB} is new even in the case of $(k,k')=(\bR,\bC)$. We can strengthen it in this setting:
	
	\begin{corD}[Example \ref{ex:C/Rcase}]\label{corD}
		Put $k=\bR$ and $k'=\bC$. Then there exists an element $w\in G(\bC)$ with the following properties:
		\begin{enumerate}
			\renewcommand{\labelenumi}{(\roman{enumi})}
			\item $\bar{B}'=wB'w^{-1}$, where $\bar{B}'$ is the complex conjugation to $B'$;
			\item If a character $\lambda$ of $B'$ satisfies $\bar{\lambda}=w\lambda$ then $\lambda(\bar{w}w)\in\{\pm 1\}$, where $\bar{\lambda}$ is the complex conjugation to $\lambda$.
		\end{enumerate}
		Moreover, a complex irreducible representation of $G\otimes_\bR\bC$ with lowest weight $\lambda$ admits a real form if and only if $\bar{\lambda}=w\lambda$ and $\lambda(\bar{w}w)=1$.
	\end{corD}
	It implies that we can obtain the descent data of complex irreducible representations uniformly (if they exist).
	
	For concrete applications of the above results, we introduce the standard $\bZ\left[1/2\right]$-forms $(G,K,H)$ of classical connected Lie groups, their maximal compact subgroups, and fundamental Cartan subgroups by developing a general theory of symmetric pairs and fundamental Cartan subgroups over the $\bZ\left[1/2\right]$-algebras:
	\begin{thmE}[Lemma \ref{lem:fixedpoint}, Theorem \ref{thm:fundamentalcartan}]
		Let $G$ be a reductive group scheme over a commutative $\bZ\left[1/2\right]$-algebra $k$, and $\theta$ be an involution of $G$. Let $K$ be the fixed point subgroup of $G$ by $\theta$. Then $K$ is represented by a smooth closed affine subgroup scheme. Moreover, if we are given a maximal torus $T$ of $K$, its centralizer $H=Z_G(T)$ is a maximal torus of $G$.
	\end{thmE}
	We call our forms standard because the involutions $\theta$ we use to define our forms are standard/familiar in the theory of Lie groups. A remarkable point is that $(G\otimes_{\bZ\left[1/2\right]}\bZ\left[1/2,\sqrt{-1}\right],
	H\otimes_{\bZ\left[1/2\right]}\bZ\left[1/2,\sqrt{-1}\right])$ is split. Moreover, there is no arithmetic obstructions to the existence of descent data on the equivariant line bundles $\cL_\lambda$ in the following sense that an analog of Corollary \ref{corD} holds: 
	\begin{thmF}[Proposition \ref{prop:conj=-1}, Section 4, Example \ref{ex:charofparabolics}, Theorem \ref{thmB}]\label{thmF}
		Let $(G,H)$ be a standard form.
		\begin{enumerate}
			\renewcommand{\labelenumi}{(\arabic{enumi})}
			\item There exist a simple system $\Pi$ of \[(G\otimes_{\bZ\left[1/2\right]}\bZ\left[1/2,\sqrt{-1}\right],
			H\otimes_{\bZ\left[1/2\right]}\bZ\left[1/2,\sqrt{-1}\right])\]
			and an element $w$ of the normalizer of \[H\otimes_{\bZ\left[1/2\right]}\bZ\left[1/2,\sqrt{-1}\right]\]
			in $G(\bZ\left[1/2,\sqrt{-1}\right])$ with the following properties:
			\begin{enumerate}
				\item[(i)] $\bar{\Pi}=w\Pi$;
				\item[(ii)] $(\bar{w}w)^2$ is the unit of
				$G(\bZ\left[1/2,\sqrt{-1}\right])$.
			\end{enumerate}
			\item Let $P'$ be the parabolic subgroup of $G\otimes_{\bZ\left[1/2\right]}\bZ\left[1/2,\sqrt{-1}\right]$ attached to a subset $\Pi'\subset \Pi$ satisfying $\bar{\Pi}'=w\Pi'$. Then a character $\lambda$ of $H\otimes_{\bZ\left[1/2\right]}\bZ\left[1/2,\sqrt{-1}\right]$ extends to that of $P'$ if and only if $\langle\alpha^\vee,\lambda\rangle=0$ for all $\alpha\in\Pi'$, where $\alpha^\vee$ is the coroot attached to $\alpha$, and $\langle-,-\rangle$ is the canonical pairing of cocharacters and characters. Moreover, the extension of $\lambda$ to a character of $P'$ is unique if it exists.
			\item For a character $\lambda$ of $P'$, $\cL_\lambda=(G\otimes_{\bZ\left[1/2\right]}\bZ\left[1/2,\sqrt{-1}\right])\times^{P'}\lambda$ (uniquely) admits a descent datum if and only if $\bar{\lambda}=w\lambda$ and $\lambda(\bar{w}w)=1$.
		\end{enumerate}
	\end{thmF}
	We prove it by finding $w$ for each form $(G,H)$ directly. The following statement is a summary of the above results:
	\begin{corG}\label{corG}
		Let $(G,H)$ be a standard form. Let $\Pi$ and $w$ be as in Theorem \ref{thmF} (1).
		\begin{enumerate}
			\renewcommand{\labelenumi}{(\arabic{enumi})}
			\item The parabolic types of $G$ over $\bZ\left[1/2\right]$ are in one-to-one correspondence with subsets $\Pi'\subset \Pi$ satisfying $\bar{\Pi}'=w\Pi'$.
			\item Let $P'$ and $\Pi'$ be as in Theorem \ref{thmF} (2). Write $\cP_{G,x}$ for the $\bZ\left[1/2\right]$-form of $(G\otimes_{\bZ\left[1/2\right]}
			\bZ\left[1/2,\sqrt{-1}\right])/P'$. Then the set of isomorphism classes of equivariant line bundles on $\cP_{G,x}$ is bijective to that of characters of $H\otimes_{\bZ\left[1/2\right]}\bZ\left[1/2,\sqrt{-1}\right]$ satisfying $\langle \alpha^\vee,\lambda\rangle=0$ for $\alpha\in\Pi'$, $\bar{\lambda}=w\lambda$, and $\lambda(\bar{w}w)=1$.
		\end{enumerate}
	\end{corG}
	We also determine the half-integral parabolic types of $G$, and determine which $\lambda$ satisfies $\langle \alpha^\vee,\lambda\rangle=0$, $\alpha\in\Pi'$, $\bar{\lambda}=w\lambda$, and $\lambda(\bar{w}w)=1$ for our standard models $G$ to deduce the complete and explicit classification of equivariant line bundles on partial flag schemes of $G$ over $\bZ\left[1/2\right]$. The author believes that this is a paradigm to the descent problems for other forms of classical connected Lie groups.
	
	We remark that there are more choices of maximal tori in representation theory of real reductive Lie groups. We can start with any of them for the complete classification in the standard setting by our general formalism Theorem \ref{thmB}. The difficulty of the problem depends on the choice of a maximal torus. For example, suppose that $G$ is split. The classification problem is obvious if we choose the split maximal torus. However, the problem is difficult if we start with the fundamental Cartan sugbroups as we see in this paper. Despite this issue, we still stick to using fundamental Cartan subgroups at least for the concrete application to the standard forms because we are aiming to obtain $\bZ\left[1/2\right]$-forms of $A_{\fq}(\lambda)$-modules by localization in \cite{hayashijanuszewski} (Section 1.2). Recall that in the complex setting, $A_{\fq}(\lambda)$-modules correspond to closed $K$-orbits on partial flag varieties attached to $\theta$-stable parabolic subgroups. For this reason, rings of definition of closed $K$-orbits on partial flag schemes attached to models of $\theta$-stable parabolic subgroups are discussed in \cite{hayashijanuszewski}. We stick to fundamental Cartan subgroups since they suit this context most.
	
	\subsection{Organization of this paper}
	In Section 2, we develop computational techniques to describe the Dynkin scheme and to construct equivariant line bundles on partial flag schemes by Galois descent. The uniqueness problem of descent data is also discussed. We give its application to the rationality problem of representations of connected reductive algebraic groups over fields of characteristic zero at the end of Section 2.
	
	In Section 3, we construct the standard $\bZ\left[1/2\right]$-forms of classical connected Lie groups, their maximal compact subgroups, and fundamental Cartan subgroups. For this, we confirm the notions of symmetric and fundamental Cartan subgroups over commutative $\bZ\left[1/2\right]$-algebras.
	
	In Section 4, we complete our program of Section 2 for our standard $\bZ\left[1/2\right]$-forms of classical connected Lie groups in Section 3. Namely, we compute the Dynkin scheme, and classify equivariant line bundles on partial flag schemes over $\bZ\left[1/2\right]$ for these groups.
	\subsection{Notation}
	Let $\bN$ denote the set of nonnegative integers.
	
	For a prime number $p$, let $\mathbb{Z}_p$ (resp.~$\mathbb{Q}_p$) denote the ring of $p$-adic integers (resp.~the field of $p$-adic numbers).
	
	For a given group, we will denote its unit by $e$.
	
	For a commutative ring $k$, let $k^\times$ denote the group of units of $k$.
	
	For an object $X$ of a category, we denote the identity morphism of $X$ by $\id_X$.
	
	For a quasi-coherent sheaf $\cF$ on a scheme $X$, let $\Gamma(X,\cF)$ denote the module of global sections of $\cF$.
	
	Let $\Gamma$ be a group. For a $\Gamma$-set $S$, write $S^\Gamma=\{s\in S:~\sigma (s)=s\mathrm{\ for\ all\ }
	\sigma\in\Gamma\}$. For $\sigma\in\Gamma$ and sets $S_\tau$ indexed by $\tau\in\Gamma$, write $s_\sigma$ for
	\[\prod_{\tau\in\Gamma} S_\tau\to \prod_{\tau\in\Gamma} S_{\sigma\tau};~(x_\tau)\mapsto (x_{\sigma\tau}).\]
	We also define $s_\sigma:\prod_{\tau\in\Gamma} X_\tau\to \prod_{\tau\in\Gamma} X_{\sigma\tau}$ for copresheaves $X_\tau$ ($\tau\in\Gamma$) on a fixed category in a similar way.
	
	In this section, fix a commutative ring $k$. Write $\Set$ (resp.\ $\Grp$) for the category of sets (resp.\ groups). Let $\CAlg_k$ be the category of commutative $k$-algebras, and $\Set^{\CAlg_k}$ (resp.\ $\Grp^{\CAlg_k}$) be the category of copresheaves on $\CAlg_k$ valued in $\Set$ (resp.\ $\Grp$). The objects of $\Set^{\CAlg_k}$ (resp.\ $\Grp^{\CAlg_k}$) are called $k$-spaces (resp.\ group $k$-spaces). The category of $k$-schemes (resp.\ group $k$-schemes) will be regarded as a full subcategory of $\Set^{\CAlg_k}$ (resp.\ $\Grp^{\CAlg_k}$) via the restricted co-Yoneda functor.
	
	Suppose that we are given a homomorphism $i:k\to k'$ of commutative rings. Then the precomposition with $i$ and the base change define an adjunction
	\[-\otimes_k k':\CAlg_k\rightleftarrows \CAlg_{k'}.\]
	They induce adjunctions
	\[-\otimes_k k':
	\Set^{\CAlg_k}\rightleftarrows\Set^{\CAlg_{k'}}:\Res_{k'/k}\]
	\[-\otimes_k k':
	\Grp^{\CAlg_k}\rightleftarrows\Grp^{\CAlg_{k'}}:\Res_{k'/k}\]
	by the precomposition. We remark that the above base change functors $-\otimes_k k'$ are compatible with those from $k$-schemes and\ group $k$-schemes to $k'$-schemes and group $k'$-schemes respectively. The right adjoint functor $\Res_{k'/k}$ is called the Weil restriction. The adjoint functor theorem (see \cite{MR1294136} 1.66) implies that $\Res_{k'/k}$ restricts to the adjunction of affine schemes if $k'$ is finitely generated and projective as a $k$-module. Hence under this condition, for an affine $k'$-group scheme $G'$, $\Res_{k'/k} G'$ is represented by an affine group $k$-scheme. In this paper, we sometimes omit $\otimes_k k'$ to write $G$ for $G\otimes_k k'$ if there is no risk of confusion on the base in order to save space. Following this convention, we will denote the multiplicative (resp.\ additive) group scheme over $k$ by $\bG_m$ (resp.\ $\bG_a$).
	
	For an automorphism $\sigma$ of $k$, write ${}^\sigma(-)$ for the base change functors of $\CAlg_k$, $\Set^{\CAlg_k}$, and $\Grp^{\CAlg_k}$ by $\sigma$. We will identify ${}^\sigma(-):\CAlg_k\to\CAlg_k$ with the functor defined by the precomposition of the structure homomorphism with $\sigma^{-1}$.
	
	Let $G$ be a group scheme over $k$. Let $X^\ast(G)$ be the character group of $G$. For a subgroup scheme $H\subset G$, $\lambda\in X^\ast(H)$, and $g\in G(k)$, define $g\lambda\in X^\ast(gHg^{-1})$ by $(g\lambda)(t) = \lambda(g^{-1}tg)$. Dually, for a cocharacter $\mu$ of $G$ and $g\in G(k)$, define a new cocharacter $g\mu$ of $G$ by $(g\mu)(a)=g\mu(a)g^{-1}$. For a subgroup scheme $H$, let $N_G(H)$ (resp.~$Z_G(H)$) be the normalizer (resp.~centralizer) subgroup if it is represented by a group $k$-scheme (see \cite{MR3362641} Proposition 2.1.6 and Lemma 2.2.4). Write $Z_G$ for the center $Z_G(G)$ of $G$.
	
	Let $H$ be a commutative group scheme over $k$. We will denote the cocharacter group of $H$ by $X_\ast(H)$. If $\Spec k$ is connected, the canonical pairing $X_\ast(H)\otimes_\bZ X^\ast(H) \to X^\ast(\bG_m)\cong \bZ$ will be referred to as $\langle-,-\rangle$ (see \cite{MR0212024} Corollaire 1.6 for the isomorphism $X^\ast(\bG_m)\cong\bZ$).
	
	Let $G$ be a reductive group scheme over $k$ in the sense of \cite{MR0228502}. Let $\sD(G)$ denote the derived subgroup of $G$ (\cite{MR0218363} Section 6.1, 6.2). Let $\Dyn G$ and $\type G$ be the Dynkin scheme and the scheme of parabolic types of $G$ respectively (\cite{MR0228503}
	Section 3, \cite{MR0218364} D\'efinition 3.4). Let $\cP_G$ denote the moduli scheme of parabolic subgroups of $G$, and $t:\cP_G\to\type G$ be the canonical morphism (\cite{MR0218364} Section 3.2). For a parabolic type $x$ of $G$, let $\cP_{G,x}$ denote the fiber of $t$ at $x$. For a maximal torus $H$ of $G$, let $W(G,H)$ be the Weyl group scheme $N_G(H)/H$ (\cite{MR3362641} Proposition 3.2.8). Let us also recall the notations on the dynamic method (\cite{MR3362817} Section 2.1). For a cocharacter $\mu$ of $G$, define group $k$-spaces $P_G(\mu)$ and $Z_G(\mu)$ by
	\[P_G(\mu)(A)=\{g\in G(A):\ \lim_{a\to 0}\mu(a)g\mu(a)^{-1}\ {\rm exists}\}\]
	\[Z_G(\mu)(A)=\{g\in G(A):\ \mu(a)g\mu(a)^{-1}=g\ 
	{\rm for\ all\ }A'\in\CAlg_A{\rm \ and\ }a\in A'\}.\]
	Then $P_G(\mu)$ is a parabolic subgroup of $G$ by \cite{MR3362641} Theorem 4.1.7 4 and \cite{MR3362817} Proposition 2.2.9. The group $k$-space $Z_G(\mu)$ is a reductive subgroup scheme of $G$ (\cite{MR3362817} Lemma 2.1.5, Proposition 2.1.8 (3), (4), \cite{MR3362641} Example 4.1.9).
	
	Suppose that $k$ is nonzero. Let $(G,H)$ be a split reductive group scheme over $k$ (see \cite{MR0218363} D\'efinition 1.13). Write $\Delta(G,H)$ for the set of roots with respect to $H$. For $\alpha\in\Delta(G,H)$, let $\alpha^\vee$ denote the coroot of $\alpha$. Suppose that $\Spec k$ is connected. Then $W(G,H)(k)$ is canonically isomorphic to the Weyl group attached to the root datum of $G$ (\cite{MR3362641} Corollary 5.1.11).
	
	Let $n$ be a nonnegative integer. For square matrices $A_1,A_2,\ldots,A_n$, let us denote the blockwise diagonal matrix
	\[\left(\begin{array}{cccc}
		A_1 &  &  &  \\
		& A_2 &  &  \\
		&  & \ddots &  \\
		&  &  & A_n
	\end{array}
	\right)\]
	by $\diag(A_1,A_2,\ldots,A_n)$. Let $I_n$ be the unit matrix of size $n$. Set
	\[J_n=\left(\begin{array}{cc}
		0 & I_n \\
		-I_n&0 
	\end{array}
	\right)\]
	\[S_2=\left(\begin{array}{cc}
		0 & 1  \\
		1&0 
	\end{array}
	\right)\]
	\[S_{2n+1}=\diag(\overbrace{S_2,S_2,\ldots,S_2}^{n},1)\]
	\[S'_{2n+1}=\diag(1,\overbrace{S_2,S_2,\ldots,S_2}^{n})\]
	\[S_{2n}=\diag(\overbrace{S_2,S_2,\ldots,S_2}^{n})\]
	\[S'_{2n}=\left(\begin{array}{cc}
		0 & I_n  \\
		I_n&0 
	\end{array}
	\right)\]
	\[w_2=\diag(1,-1).\]
	Define an $n\times n$ square matrix $K_n$ by
	\[K_n=\left(\begin{array}{ccc}
		&  & 1 \\
		& \reflectbox{$\ddots$} &  \\
		1 &  & 
	\end{array}\right).\] 
	For nonnegative integers $p,q$, set
	\[I_{p,q}=\diag(I_p,-I_q)\]
	\[I_{p,q,p,q}=\diag(I_{p,q},I_{p,q}).\]
	For a matrix $A$, write $A^T$ for its transpose.
	
	Let $n$ be a positive integer. We define affine group schemes $\GL_n$, $\SL_n$, $\Oo(n)$, $\SO(n)$, $\SO_{2n-1}$, $\Sp_n$, $\SO_{2n}$, $\SO_{2n}'$ over $\bZ$\footnote{In the most part of this paper, we will regard them as group schemes over $\bZ\left[1/2\right]$ or $\bZ\left[1/2,\sqrt{-1}\right]$ by the base change.} as group $\bZ$-spaces by
	\[\GL_n(A)=\{{\rm invertible\ }n\times n\ {\rm matrices\ whose\ entries\ belong\ to\ }A\}\]
	\[\SL_n(A)=\{g\in\GL_n(A):\ \det g=1\}\]
	\[\Oo(n,A)=\{g\in\GL_n(A):\ g^T=g^{-1}\}\]
	\[\SO(n,A)=\{g\in\SL_n(A):\ g^T=g^{-1}\}\]
	\[\SO_{2n-1}(A)=\{g\in\SL_{2n-1}(A):\ g^TS_{2n-1}g=S_{2n-1}\}\]
	\[\SO'_{2n-1}(A)=\{g\in\SL_{2n-1}(A):\ g^TS'_{2n-1}g=S'_{2n-1}\}\]
	\[\Sp_n(A)=\{g\in\GL_{2n}(A):\ g^TJ_ng=J_n\}\]
	\[\SO_{2n}(A)=\{g\in\SL_{2n}(A):\ g^TS_{2n}g=S_{2n}\}\]
	\[\SO_{2n}'(A)=\{g\in\SL_{2n}(A):\ g^TS_{2n}'g=S_{2n}'\}.\]
	We remark that $\Sp_n$ is a subgroup scheme of $\SL_{2n}$ by \cite{MR780184} (22) and Theorem 6.4.
	
	For a $\bZ\left[1/2\right]$-algebra $A$, we denote an element
	\[(a_{ij}\otimes 1+b_{ij}\otimes\sqrt{-1})\in(\Res_{\bZ\left[1/2,\sqrt{-1}\right]/\bZ\left[1/2\right]}\GL_n)(A)\]
	by $a\otimes 1+b\otimes\sqrt{-1}$, where $a=(a_{ij})$ and $b=(b_{ij})$ are matrices whose entries belong to $A$. Define an anti-involution $(-)^\ast$ on $\Res_{\bZ\left[1/2,\sqrt{-1}\right]/\bZ\left[1/2\right]}\GL_n$ by
	\[g=a\otimes 1+b\otimes\sqrt{-1}\mapsto a^T\otimes 1-b^T\otimes\sqrt{-1}=g^\ast.\]
	Set
	\[g_2=\left(\begin{array}{cc}
		1&-\sqrt{-1}\\
		-\sqrt{-1}&1\\
	\end{array}\right)\in\GL_2(\bZ\left[1/2,\sqrt{-1}\right])\]
	\[g_{2n-1}=\diag(g_2,g_2,\ldots,g_2,1)\in\GL_{2n-1}(\bZ\left[1/2,\sqrt{-1}\right])\]
	\[g_{2n}=\diag(g_2,g_2,\ldots,g_2)\in\GL_{2n}(\bZ\left[1/2,\sqrt{-1}\right])\]
	\[g'_{2n}=\left(\begin{array}{cc}
		I_n & \sqrt{-1}I_n \\
		\frac{1}{2}I_n & -\frac{\sqrt{-1}}{2}I_n
	\end{array}
	\right)\in\GL_{2n}(\bZ\left[1/2,\sqrt{-1}\right]).\]
	
	Suppose that we are given a finite group $\Gamma$ and $\Gamma$-modules $M_\tau$ indexed by $\tau\in\Gamma$. For $\sigma\in\Gamma$ and $m\in M_\sigma$, we denote the image of $m$ under the canonical map $M_\sigma\to\oplus_{\tau\in\Gamma} M_\tau=\prod_{\tau\in\Gamma} M_\tau$ by $m\tau$.
	
	\section{Descent theory for partial flag schemes}
	This section is devoted to a general theory to classify partial flag schemes and equivariant line bundles on them by Galois descent and its application to representation theory.
	\subsection{Compute the Dynkin scheme}
	Recall that for a reductive group scheme $G$ over a commutative ring $k$, $\Dyn G$ is the constant $k$-scheme attached to the simple system of $G$ if $G$ is a split reductive group scheme with a pinning; For general $G$, we glue it \'etale locally (\cite{MR0228503} Section 3). Let $\type G$ be the moduli scheme of open and closed subschemes of $\Dyn G$ (see \cite{MR0218364} Section 3). Since the partial flag schemes are defined as the fibers of the projection morphism $t:\cP_G\to\type(G)$, we can classify partial flag schemes from $\Dyn(G)$. For this reason, we discuss how to compute $\Dyn(G)$ in this section.
	
	For practical computations of $\Dyn(G)$, let us simplify the situation. Let $k \to k'$ be a Galois extension of Galois group $\Gamma$ such that $\Spec k'$ is connected. Suppose that there is a maximal torus $H$ of $G$ such that $(G\otimes_k k',H\otimes_k k')$ is split. Fix a simple system $\Pi$ of $(G\otimes_k k',H\otimes_k k')$. For each $\sigma\in\Gamma$, set ${}^\sigma\Pi = \{{}^\sigma\alpha\in \Delta(G\otimes_k k',H\otimes_k k'):~\alpha\in\Pi\}$. Then there exists a unique element $w_\sigma\in W(G,H)(k')$ such that $w_\sigma(\Pi) = {}^\sigma\Pi$. We remark that in latter sections, we will fix a representative of $w_\sigma$ in $N_{G(k')}(H\otimes_k k')$ which will be denoted by the same symbol $w_\sigma$. See \cite{MR3362641} Corollary 5.1.11 for the existence of $w_\sigma\in N_{G(k')}(H\otimes_k k')$. It follows by definition of $w_\sigma$ that these elements satisfy the cocycle condition $w_{\sigma\tau}=\sigma(w_\tau)w_\sigma$ in $W(G,H)(k')$ for $\sigma,\tau\in\Gamma$. In particular, we get $w_\sigma^{-1}=\sigma(w_{\sigma^{-1}})$.
	
	\begin{defn}[cf.~\cite{MR106967}]
		The map $\Gamma\times X^\ast(H\otimes_k k')\to X^\ast(H\otimes_k k');~(\sigma,\lambda)\mapsto {}^{\sigma}(w_{\sigma^{-1}}\lambda)$ determines a left action of $\Gamma$ on $X^\ast(H\otimes_k k')$. We call it the $\ast$-action. When $\Gamma=\bZ/2\bZ$, we call the action by the unique nontrivial element of $\Gamma$ the $\ast$-involution.
	\end{defn}
	
	It is clear that the $\ast$-action restricts to an action on $\Pi$. Unwinding the definitions, we obtain the following description:
	
	\begin{thm}\label{thm:dynkin}
		\begin{enumerate}
			\renewcommand{\labelenumi}{(\arabic{enumi})}
			\item The Galois group acts on the coordinate ring $\prod_{\alpha\in\Pi} k'$ of $\Dyn G\otimes_k k'$ by
			\[\sigma\cdot (c_\alpha)_{\alpha\in\Pi}
			=(\sigma
			(c_{{}^{\sigma^{-1}}(w_{\sigma}\alpha)}))_{\alpha\in\Pi}.\]
			Moreover, $\Dyn G$ is an affine $k$-scheme whose coordinate ring is $(\prod_{\alpha\in\Pi} k')^\Gamma$.
			\item Regard $\Pi$ as a $\Gamma$-set for the $\ast$-action. Then $(\type G)(k)$ is bijective to the power set of $\Gamma\backslash\Pi$.
		\end{enumerate}
	\end{thm}
	\begin{ex}\label{ex:dynkin}
		Suppose $\Gamma=\bZ/2\bZ$. Regard $\Pi$ as a $\Gamma$-set by the $\ast$-action. Then we have $\Dyn G\cong
		\coprod_{\overset{\cO\in\Gamma\backslash\Pi}{|\cO|=1}}
		\Spec k\coprod
		\coprod_{\overset{\cO\in\Gamma\backslash\Pi}{|\cO|=2}}
		\Spec k'$.
	\end{ex}
	\subsection{Galois descent of associated bundles on partial flag schemes}
	In this section, we discuss existence and uniqueness of data of Galois descent of equivariant line bundles on partial flag schemes. Let $k\to k'$ be a Galois extension of commutative rings or an infinite Galois extension of fields of Galois group $\Gamma$, $G$ be a reductive group scheme over $k$, and $P'$ be a parabolic subgroup of $G\otimes_k k'$. Suppose that for every element $\sigma\in\Gamma$, there exists an element $w_\sigma\in G(k')$ such that ${}^\sigma P' = w_\sigma P'w_\sigma^{-1}$. Then $t(P')$ descends to a parabolic type $x\in (\type G)(k)$. In particular, $(G\otimes_k k')/P'$ is equipped with the structure of a projective $k$-scheme. In fact, the set of the isomorphisms $(G\otimes_k k')/{}^\sigma P'\cong (G\otimes_k k')/P';~g{}^\sigma P' \mapsto gw_{\sigma} P'$ determines a datum of Galois descent of $\cP_{G,x}$.
	
	Let $\lambda$ be a character of $P'$. Then we can define a line bundle $\cL_\lambda$ on $(G\otimes_k k')/P'$ by $(G\otimes_k k')\times^{P'} \lambda$ (see \cite{MR2015057} I.5.8, I.5.9\footnote{Precisely speaking, Jantzen assumed that the canonical map $G\otimes_k k'\to (G\otimes_k k')/P'$ is a locally trivial fibration in the Zariski topology for $(G\otimes_k k')\times^{P'} \lambda$ being a line bundle. However, the map $G\otimes_k k' \to (G\otimes_k k')/P'$ is locally trivial only in the \'etale topology. One can still prove that $(G\otimes_k k')\times^{P'} \lambda$ is a line bundle by applying a similar argument to Jantzen's in the \'etale topology and Hilbert's theorem 90.}).
	\begin{rem}\label{rem:associatedbundle}
		If $k'$ is a PID, every $G\otimes_k k'$-equivariant line bundle on $(G\otimes_k k')/P'$ is of the form $\cL_\lambda$. More generally, let $G$ be a reductive group scheme over a PID $k$, and $P\subset G$ be a parabolic subgroup. Then every $G$-equivariant line bundle $\cL$ on $G/P$ is of the form $G\times^P \lambda$, where $\lambda\in X^\ast(P)$. In fact, restrict $\cL$ to the base point of $G/P$, i.e., take the pullback along the map $\Spec k\to G/P$ corresponding to $P\in (G/P)(k)$, to get a $P$-equivariant line bundle on $\Spec k$. Since $k$ is a PID, we can identify it with a character $\lambda\in X^\ast(P)$. Extend it $G$-linearly to get $G\times \lambda\to \cL$. It follows by definition of $\lambda$ that it descends to a $G$-equivariant homomorphism $G\times^P \lambda\to \cL$ of $G$-equivariant line bundles. This is clearly an isomorphism. From general perspectives, one can deduce this correspondence from twice use of \cite{MR4225278} Theorem 14.87 (see \cite{MR2838836} Section 5.2.16 for example if necessary). In fact, note that $G$ is equipped with commuting left and right actions of $G$ and $P$.
	\end{rem}
	We wish to classify data of Galois descent on the quasi-coherent sheaf corresponding to $\cL_\lambda$ (\cite{MR4225278} (14.20), (14.21)). For this, in virtue of the definition of $\cL_\lambda$, we may and do study data of Galois descent on the bundle $\cL_\lambda$ itself rather than the corresponding sheaf in this paper.
	\begin{lem}
		For $\sigma\in\Gamma$, $(g,v) \mapsto (gw_\sigma,v)$ determines a well-defined $(G\otimes_k k')$-equivariant isomorphism $\Psi_\sigma: (G\otimes_k k')\times^{{}^\sigma P'} w_\sigma\lambda\cong (G\otimes_k k')\times^{P'} \lambda$ of line bundles on $(G\otimes_k k')/P'$ under the identification
		$(G\otimes_k k')/{}^\sigma P' \cong (G\otimes_k k')/P'$.
	\end{lem}
	\begin{proof}
		The map $\Psi_\sigma$ is well-defined from
		\[\begin{split}
			(gw_\sigma pw_\sigma^{-1},
			w_\sigma p^{-1} w_\sigma^{-1}v)
			&=(g w_\sigma pw_\sigma^{-1},(w\lambda)(w_\sigma p^{-1} w_\sigma^{-1})v)\\
			&=(g w_\sigma pw_\sigma^{-1},\lambda(p)^{-1}v)\\
			&\mapsto (gw_\sigma p,\lambda(p)^{-1}v)\\
			&=(gw_\sigma,v),
		\end{split}\]
		where $A'$ is a $k'$-algebra, $g\in G(A')$, $p\in P'(A')$, and $v\in A'$. It is clear by construction that $\varphi_\sigma$ is a $(G\otimes_k k')$-equivariant isomorphism of line bundles.
	\end{proof}
	This shows that the associated line bundle $\cL_\lambda$ descends to a $G$-equivariant line bundle on $\cP_{G,x}$ only if ${}^\sigma\lambda=w_\sigma\lambda$.
	
	Suppose that the equality ${}^\sigma\lambda=w_\sigma\lambda$ holds for every element $\sigma\in\Gamma$. To see whether $(\Psi_\sigma)_{\sigma\in\Gamma}$ determines a datum of Galois descent of $\cL_\lambda$, we shall compute $\Psi_\sigma\circ {}^\sigma\Psi_\tau\circ \Psi_{\sigma\tau}^{-1}$ for a pair $\sigma,\tau$ of elements of $\Gamma$:
	\[\begin{split}
		(\Psi_\sigma\circ {}^\sigma\Psi_\tau\circ \Psi_{\sigma\tau}^{-1})(g,v)
		&=(\Psi_\sigma\circ {}^\sigma\Psi_\tau)(gw_{\sigma\tau}^{-1},v)
		\\
		&=\Psi_{\sigma}(gw_{\sigma\tau}^{-1}\sigma(w_\tau),v)\\
		&=(gw_{\sigma\tau}^{-1}\sigma(w_\tau)w_\sigma,v).
	\end{split}\]
	\begin{lem}
		We have $w_{\sigma\tau}^{-1}\sigma(w_\tau)w_\sigma
		\in P' (k')$.
	\end{lem}
	\begin{proof}
		Observe that
		\[\begin{split}
			(w_{\sigma\tau}^{-1}\sigma(w_\tau)w_\sigma) P'
			(w_{\sigma\tau}^{-1}\sigma(w_\tau)w_\sigma)^{-1}
			&=(w_{\sigma\tau}^{-1}\sigma(w_\tau)) {}^\sigma P'
			(w_{\sigma\tau}^{-1}\sigma(w_\tau))^{-1}\\
			&=w_{\sigma\tau}^{-1} {}^\sigma (w_\tau P' w_\tau^{-1}) w_{\sigma\tau}\\
			&=w_{\sigma\tau}^{-1} {}^\sigma ({}^\tau P') w_{\sigma\tau}\\
			&=w_{\sigma\tau}^{-1} {}^{\sigma\tau} P' w_{\sigma\tau}\\
			&=P'.
		\end{split}\]
		The assertion now follows from \cite{MR0218364} Proposition 1.2.
	\end{proof}
	By this observation, we get
	\[\Psi_\sigma\circ {}^\sigma\Psi_\tau\circ \Psi_{\sigma\tau}^{-1}
	=\lambda(w_{\sigma\tau}^{-1}\sigma(w_\tau)w_\sigma)
	\id_{\cL_{\lambda}}.\]
	\begin{lem}
		\begin{enumerate}
			\renewcommand{\labelenumi}{(\arabic{enumi})}
			\item For $\sigma,\tau\in\Gamma$, set $\beta_\lambda(\sigma,\tau)=\lambda(w_{\sigma\tau}^{-1}\sigma(w_\tau)w_\sigma)$. Then $\beta_\lambda$ satisfies the cocycle condition.
			\item The cocycle $\beta_\lambda$ is independent of the choice of $w_\sigma$ as an element of $H^2(\Gamma,(k')^\times)$.
		\end{enumerate}
	\end{lem}
	\begin{proof}
		Part (1) follows from the equality
		\[\beta_{V'}(\sigma,\tau\rho)
		\beta_{V'}(\sigma\tau,\rho)^{-1}
		\beta_{V'}(\sigma,\tau)^{-1}\id_{\cL_\lambda}=
		\sigma(\beta_{V'}(\tau,\rho))^{-1}\id_{\cL_\lambda}.\]
		In fact, this follows from
		\begin{flalign*}
			&\beta_{V'}(\sigma,\tau\rho)
			\beta_{V'}(\sigma\tau,\rho)^{-1}
			\beta_{V'}(\sigma,\tau)^{-1}\id_{\cL_\lambda}\\
			&=
			\circ \varphi_\sigma\circ{}^\sigma\varphi_{\tau\rho}\circ \varphi_{\sigma\tau\rho}^{-1}
			\circ \varphi_{\sigma\tau\rho} \circ
			{}^{\sigma\tau}\varphi_\rho^{-1}\circ \varphi_{\sigma\tau}^{-1}
			\circ \varphi_{\sigma\tau} \circ{}^\sigma\varphi_\tau^{-1}\circ \varphi_\sigma^{-1}\\
			&= \varphi_\sigma\circ{}^\sigma\varphi_{\tau\rho}\circ
			{}^{\sigma\tau}\varphi_\rho^{-1}\circ \circ{}^\sigma\varphi_\tau^{-1}\circ \varphi_\sigma^{-1}\\
			&=\varphi_\sigma\circ{}^\sigma(\varphi_{\tau\rho}\circ
			{}^{\tau}\varphi_\rho^{-1}\circ \varphi_\tau^{-1})\circ \varphi_\sigma^{-1}\\
			&=\varphi_\sigma\circ
			{}^\sigma(\beta_{V'}(\tau,\rho)^{-1}\id_{\cL_\lambda})
			\circ\varphi_\sigma^{-1}\\
			&=\varphi_\sigma\circ
			\sigma(\beta_{V'}(\tau,\rho))^{-1}\id_{{}^\sigma\cL_\lambda}\circ
			\varphi_\sigma^{-1}\\
			&=\sigma(\beta_{V'}(\tau,\rho))^{-1} \id_{\cL_\lambda}.
		\end{flalign*}
		
		We next prove (2). Fix $w_\sigma$ for each $\sigma\in\Gamma$. Then the other choices of $w_\sigma$ are of the form $w_\sigma p_\sigma$, where $p_\sigma\in P'(k')$. In this proof, let $\beta_\lambda'$ denote the cocycle attached to $(w_\sigma p_\sigma)$. Set $c_\sigma= \lambda(p_\sigma)$. We complete the proof by showing $\beta'_\lambda(\sigma,\tau)=\beta_\lambda(\sigma,\tau) \sigma(c_\tau) c_\sigma c_{\sigma\tau}^{-1}$ for $\sigma,\tau\in\Gamma$.
		
		Observe that we have
		\[\begin{split}
			(w_{\sigma\tau}\lambda)
			(\sigma(w_\tau p_\tau)w_\sigma w_{\sigma\tau}^{-1})
			&=({}^{\sigma\tau}\lambda)
			(\sigma(w_\tau p_\tau)w_\sigma w_{\sigma\tau}^{-1})\\
			&=({}^{\sigma}({}^\tau\lambda))
			(\sigma(w_\tau p_\tau)w_\sigma w_{\sigma\tau}^{-1})\\
			&=\sigma({}^\tau\lambda(w_\tau p_\tau 
			\sigma^{-1}(w_\sigma w_{\sigma\tau}^{-1})))\\
			&=\sigma((w_\tau\lambda)(w_\tau p_\tau 
			\sigma^{-1}(w_\sigma w_{\sigma\tau}^{-1})))\\
			&=\sigma(\lambda(p_\tau 
			\sigma^{-1}(w_\sigma w_{\sigma\tau}^{-1})w_\tau))\\
			&=\sigma(\lambda(p_\tau))
			\sigma(\lambda(\sigma^{-1}(w_\sigma w_{\sigma\tau}^{-1})w_\tau))\\
			&=\sigma(c_\tau){}^\sigma\lambda
			(w_\sigma w_{\sigma\tau}^{-1}\sigma(w_\tau))\\
			&=\sigma(c_\tau)(w_\sigma\lambda)
			(w_\sigma w_{\sigma\tau}^{-1}\sigma(w_\tau))\\
			&=\sigma(c_\tau)\lambda
			(w_{\sigma\tau}^{-1}\sigma(w_\tau)w_\sigma)\\
			&=\sigma(c_\tau)\beta_\lambda(\sigma,\tau).
		\end{split}\]
		The assertion now follows from
		\[\begin{split}
			\beta'_\lambda(\sigma,\tau)
			&=\lambda(p_{\sigma\tau}^{-1} w^{-1}_{\sigma\tau}
			\sigma(w_\tau p_\tau)w_\sigma p_\sigma)\\
			&=\lambda(p_{\sigma\tau}^{-1}) \lambda(w^{-1}_{\sigma\tau}
			\sigma(w_\tau p_\tau)w_\sigma)\lambda(p_\sigma)\\
			&=(w_{\sigma\tau}\lambda)
			(\sigma(w_\tau p_\tau)w_\sigma w_{\sigma\tau}^{-1})c_\sigma c_{\sigma\tau}^{-1} \\
			&=\beta_\lambda(\sigma,\tau) \sigma(c_\tau) 
			c_\sigma c_{\sigma\tau}^{-1}.
		\end{split}\]
	\end{proof}
	\begin{thm}\label{thm:descentofllambda}
		The equivariant line bundle $\cL_\lambda$ admits a datum of Galois descent to that on $\cP_{G,x}$ if the following conditions are satisfied:
		\begin{enumerate}
			\renewcommand{\labelenumi}{(\roman{enumi})}
			\item The equality ${}^\sigma\lambda=w_\sigma\lambda$ holds for $\sigma\in \Gamma$. 
			\item The cocycle $\beta_\lambda$ is the trivial cohomology class.
		\end{enumerate}
		The converse follows if $k'$ is Noetherian.
	\end{thm}
	
	\begin{proof}
		Suppose that (i) and (ii) are satisfied. Then there exists a set $(c_\sigma)$ of units of $k'$ indexed by $\sigma\in\Gamma$ such that $\beta_{\lambda}(\sigma,\tau)
		=c_\sigma \sigma(c_\tau) c_{\sigma\tau}^{-1}$ for $\sigma,\tau\in\Gamma$. The isomorphisms $(c_\sigma^{-1}\Psi_\sigma)$ form a datum of Galois descent of $\cL_{\lambda}$ by
		\[\begin{split}
			(c_\sigma^{-1}\Psi_\sigma)\circ
			(\sigma(c_\tau^{-1}){}^\sigma\Psi_\tau)
			\circ (c_{\sigma\tau}^{-1}\Psi_{\sigma\tau})^{-1}
			&=c_\sigma^{-1} \sigma(c_\tau^{-1}) c_{\sigma\tau}
			\Psi_\sigma\circ {}^\sigma\Psi_\tau\circ \Psi_{\sigma\tau}^{-1}\\
			&=c_\sigma^{-1} \sigma(c_\tau^{-1}) c_{\sigma\tau}
			\beta_\lambda(\sigma,\tau) \id_{\cL_\lambda}\\
			&=\id_{\cL_\lambda}.
		\end{split}\]
		
		Conversely, suppose that there exists a datum $(\Psi'_\sigma:{}^\sigma\cL_{\lambda}\cong \cL_{\lambda})$ of Galois descent. In particular, (i) follows by the former argument. Assume that $k'$ is Noetherian. Recall that the structure morphism $(G\otimes_k k')/P'\to \Spec k'$ is smooth projective with connected geometric fibers by \cite{MR0218364} Corollaire 3.6. Hence in view of \cite{MR0463157} Theorem 12.11, there exists a unit $c_\sigma$ of $k'$ such that $\Psi_\sigma=c_\sigma \Psi'_\sigma$
		for each $\sigma\in\Gamma$. Since $\Psi'_\sigma$ enjoys the cocycle condition, we have
		\[\begin{split}
			c_\sigma \sigma(c_\tau) c_{\sigma\tau}^{-1}
			\id_{\cL_\lambda}
			&=(c_\sigma\Psi_\sigma')\circ
			(\sigma(c_\tau){}^\sigma\Psi_\tau')
			\circ (c_{\sigma\tau}\Psi_{\sigma\tau}')^{-1}\\
			&=\Psi_\sigma\circ {}^\sigma\Psi_\tau\circ \Psi_{\sigma\tau}^{-1}\\
			&=\beta_\lambda(\sigma,\tau) \id_{\cL_\lambda}.
		\end{split}\]
		This shows that $\beta_{\lambda}$ is trivial in the Galois cohomology $H^2(\Gamma,(k')^\times)$.
	\end{proof}
	\begin{ex}\label{ex:defoverk}
		Suppose that $P'$ is defined over $k$, i.e., $P'=P\otimes_k k'$ for some parabolic subgroup $P\subset G$. Then we may put $w_\sigma=e$ for all $\sigma\in\Gamma$. Therefore $\lambda$ satisfies (i) in Theorem \ref{thm:descentofllambda} if and only if $\lambda$ is defined over $k$. If so, the cocycle $\beta_\lambda$ is trivial.
	\end{ex}
	\begin{ex}[\cite{MR3770183} Remark 6.4]
		Let $k\to k'$ be a Galois extension of nonzero commutative rings of Galois group $\Gamma$, $G$ be a reductive group scheme over $k$ with adjoint derived group (\cite{MR0218363} D\'efinition 4.3.3, D\'efinition 2.6.1, D\'efinition 2.7, Proposition 4.3.5), and $H\subset G$ be a maximal torus. Suppose that $(G\otimes_k k',H\otimes_k k')$ is split. Then $(\sD(G)\otimes_k k',(H\cap \sD(G))\otimes_k k')$ is split with
		\[\Delta(\sD(G)\otimes_k k',(H\cap \sD(G))\otimes_k k')
		\cong\Delta(G\otimes_k k',H\otimes_k k')\]
		by \cite{MR0218363} Proposition 6.2.7. Fix a pinning $\cE$ of $(\sD(G)\otimes_k k',(H\cap \sD(G))\otimes_k k')$ (\cite{MR0229653} D\'efinition 1.1). In view of \cite{MR0228503} Lemme 1.5, for $\sigma\in\Gamma$, there exists a unique element $w_\sigma\in N_{\sD(G)(k')}((H\cap \sD(G))\otimes_k k')$ such that $w_\sigma(-)w_\sigma^{-1}:(\sD(G)\otimes_k k',\cE)\cong (\sD(G)\otimes_k k',{}^\sigma\cE)$, where ${}^\sigma\cE$ is the base change of $\cE$ by $\sigma$. By the uniqueness, we obtain $w_{\sigma\tau}^{-1}\sigma(w_\tau)w_\sigma=e$ in $G(k')$ for all pairs $\sigma,\tau\in\Gamma$. In virtue of \cite{MR0218363} Proposition 6.2.8 (i), $w_\sigma$ belongs to $N_{G(k')}(H\otimes_k k')$. Let $B'\subset G\otimes_k k'$ be the Borel subgroup attached to $\cE$ (see \cite{MR0218363} Proposition 6.2.8 (ii) if necessary). It follows by definition that ${}^\sigma B'=w_\sigma B' w_\sigma^{-1}$ for $\sigma\in\Gamma$. Let $P'\subset G\otimes_k k'$ be a parabolic subgroup containing $B'$ such that ${}^\sigma P'=w_\sigma P'w_\sigma^{-1}$\footnote{We will see in Proposition \ref{prop:wsigmaautomatic} that this is equivalent to the conditions that ${}^\sigma P'$ is $G(k')$-conjugate to $P'$.}. Then for $\lambda\in X^\ast(P')$, $\cL_\lambda$ is defined over $k$ if and only if ${}^\sigma\lambda=w_\sigma\lambda$ for all $\sigma\in\Gamma$.
	\end{ex}
	\begin{ex}\label{ex:perfcase}
		We can follow a similar choice to the above when $k=F$ is a perfect field and $k'$ is its algebraic closure $\bar{F}$. In fact, let $H$ be a maximal torus of $G$, and $B'$ be a Borel subgroup of $G\otimes_F \bar{F}$ containing $H\otimes_F \bar{F}$. Then $H\cap \sD(G)$ and $B'\cap(\sD(G)\otimes_F \bar{F})$ are a maximal torus and a Borel subgroup of $\sD(G)$ respectively (\cite{MR0218363} Proposition 6.2.8). Fix a pinning $\cE$ of $(\sD(G)\otimes_F \bar{F},(H\cap \sD(G))\otimes_F \bar{F},B'\cap(\sD(G)\otimes_F \bar{F}))$. For each $\sigma\in\Gamma$, there exists an element $w_\sigma\in N_{\sD(G)(\bar{F})}((H\cap \sD(G))\otimes_F \bar{F})$ such that $w_\sigma(-)w^{-1}_\sigma:(\sD(G)\otimes_F \bar{F},\cE)\cong(\sD(G)\otimes_F \bar{F},{}^\sigma\cE)$ (see \cite{MR0228503} Lemme 1.5, 1.5.2). Moreover, $w_{\sigma\tau}^{-1}\sigma(w_\tau)w_\sigma$ belongs to the center of $\sD(G)$ by the uniqueness in \cite{MR0228503} Lemme 1.5. Since $Z_{\sD(G)}$ is finite, there exists a positive integer such that $(w_{\sigma\tau}^{-1}\sigma(w_\tau)w_\sigma)^n=e$. \cite{MR0218363} Proposition 6.2.8 implies $w_\sigma\in N_{G(\bar{F})}(H\otimes_F \bar{F})$ and
		${}^\sigma B'=w_\sigma B' w^{-1}_\sigma$. If we write $\Pi$ for the simple system attached to $(B',H\otimes_F \bar{F})$, we have ${}^\sigma\Pi=w_\sigma \Pi$. Let $\lambda$ be a character of $H\otimes_F \bar{F}$ satisfying ${}^\sigma \lambda=w_\sigma\lambda$ for every $\sigma\in\Gamma$. Then $\beta_\lambda$ is valued in the set of roots of unity of $\bar{F}$ by $(w_{\sigma\tau}^{-1}\sigma(w_\tau)w_\sigma)^n=e$.
	\end{ex}
	
	\begin{ex}[Class field theory]\label{ex:CFT}
		Let $k=F$ be a global field, and $k'=F^{\sep}$ be the separable closure of $F$. To study whether $\beta_\lambda$ is trivial, we may replace $F$ by a local field by \cite{MR4174395} Theorem 14.11 and the compatibility of the restriction map of the second Galois cohomology groups with the construction of $\beta_\lambda$. Suppose that $F$ is a non-archimedian local field. Write $\Gal(F^{\sep}/F)$ for the absolute Galois group of $F$ here. Assume that there is a finite Galois extension $F'/F$ of degree $n$ such that $P'$ admits a compatible $F'$-form $P'_{F'}\subset G\otimes_F F'$ , and that for each $\sigma\in \Gal(F'/F)$, there exists an element $w_\sigma'\in G(F')$ such that $w_\sigma' P'_{F'} (w_\sigma')^{-1}$, where $\Gal(F'/F)$ is the Galois group of $F'/F$. Then \cite{MR4174395} Corollary 8.10 implies that $\inv_F(\beta_\lambda)$ lies in $\frac{1}{n}\bZ/\bZ$ by (see \cite{MR4174395} Part II Section 8.2 for the definition of $\inv_F$). Assume also that the characteristic of $F$ is zero. Write $m=|Z_{\sD(G)}(\bar{F})|$. Then by Example \ref{ex:perfcase}, $\inv_F(\beta_\lambda)\in \frac{1}{(n,m)}\bZ/\bZ$, where $(n,m)$ is the greatest common divisor of $m$ and $n$. See Example \ref{ex:C/Rcase} for the case $F=\bR$.
	\end{ex}
	
	We next discuss the uniqueness of descent data up to isomorphisms.
	\begin{thm}\label{thm:uniqueness}
		Assume $k'$ Noetherian. Let $\lambda$ be any of a character of $P'$. Then the set of isomorphism classes of descent data on $\cL_\lambda$ is a principal $H^1(\Gamma,(k')^\times)$-set. In particular, the descent data on $\cL_\lambda$ are at most unique up to isomorphisms if $H^1(\Gamma,(k')^\times)$ vanishes.
	\end{thm}
	\begin{proof}	
		Let $S$ be the set of descent data on $\cL_\lambda$, and $Z^1$ be the abelian group of 1-cocyles $\Gamma\to (k')^\times$. Then $S$ is a principal $Z^1$-group for $(c_\sigma)\cdot (\cL_\lambda,\Psi_\sigma')\coloneqq (\cL_\lambda,c_\sigma \Psi_\sigma')$ by the cocycle condition for descent data and by the proof of the converse direction of Theorem \ref{thm:descentofllambda}.
		
		We prove that it descends to the simply transitive action of $H^1(\Gamma,(k')^\times)$ on the set of isomorphism classes of descent data on $\cL_\lambda$. The proof is completed by showing the following assertions:
		\begin{enumerate}
			\renewcommand{\labelenumi}{(\roman{enumi})}
			\item $Z^1$ respects isomorphisms of descent data. I.e., for $(c_\sigma)\in Z^1$ and \[(\cL_\lambda,\Psi_\sigma'),(\cL_\lambda,\Psi_\sigma'')\in S,\]
			we have $(\cL_\lambda,c_\sigma\Psi_\sigma')\cong(\cL_\lambda,c_\sigma\Psi_\sigma'')$ if $(\cL_\lambda,\Psi_\sigma')\cong(\cL_\lambda,\Psi_\sigma'')$.
			\item For $(\cL_\lambda,\Psi_\sigma'),(\cL_\lambda,\Psi_\sigma'')\in S$, they are isomorphic to each other if and only if there exists a unit $t\in (k')^\times$ such that $\Psi_\sigma'=\sigma(t)t^{-1}\Psi_\sigma''$ for all $\sigma\in\Gamma$.
		\end{enumerate}
		For (i), suppose that we are given an isomorphism
		$(\cL_\lambda,\Psi_\sigma') \cong (\cL_\lambda,\Psi_\sigma'')$. 
		The underlying scalar map on $\cL_\lambda$ clearly determines an isomorphism $(\cL_\lambda,c_\sigma\Psi_\sigma')
		\cong (\cL_\lambda,c_\sigma\Psi_\sigma'')$ for 1-cocyles $(c_\sigma)$.
		
		We next prove (ii). Let $(\cL_\lambda,\Psi_\sigma')\in S$. For $t\in (k')^\times$, the scalar map $t\id_{\cL_\lambda}$ determines an isomorphism $(\cL_\lambda,\sigma(t)t^{-1} \Psi_\sigma')\cong (\cL_\lambda,\Psi_\sigma')$. This shows the ``if'' direction. Conversely, suppose that we are given an isomorphism $(\cL_\lambda,\Psi_\sigma')\cong(\cL_\lambda,\Psi_\sigma'')$ of descent data. Since $\cL_{\lambda}$ is a line bundle, this isomorphism is given by $t\id_{\cL_\lambda}$ for a unit $t\in (k')^\times$. In particular, we obtain $\Psi_\sigma'=\sigma(t)t^{-1}\Psi_\sigma''$ by definition of morphisms of descent data. This completes the proof.
	\end{proof}
	\begin{ex}[Hilbert's theorem 90]\label{ex:vanishing}
		The cohomology $H^1(\Gamma,(k')^\times)$ vanishes if $k$ is a PID. In fact, if we are given a 1-cocycle $c$, it attaches a descent datum on $k'$ as a $k'$-module along the Galois extension $k\to k'$. Since $k$ is a PID, this datum is isomorphic to the trivial descent datum on $k'$. If we write $t$ for the image of $1\in k'$ under this isomorphism, we obtain $c=(\sigma(t)t^{-1})$. In particular, $c$ is a coboundary.
	\end{ex}
	In the rest, assume $\Gamma=\bZ/2\bZ$ to work more on the existence of data of Galois descent. Choose an element $w\in G(k')$ such that $\bar{P}'=w P' w^{-1}$.
	\begin{cor}\label{cor:quadcase}
		The equivariant line bundle $\cL_\lambda$ admits a datum of Galois descent to that on $\cP_{G,x}$ if the following conditions are satisfied:
		\begin{enumerate}
			\renewcommand{\labelenumi}{(\roman{enumi})}
			\item $\bar{\lambda}=w\lambda$.
			\item There exists a nonzero element $c\in F'$ such that $\lambda(\bar{w}w)=\bar{c}c$.
		\end{enumerate}
		The converse holds if $k'$ is Noetherian.
	\end{cor}
	\begin{ex}\label{ex:C/Rcase}
		Put $k=\bR$. Let us consider the setting of Example \ref{ex:perfcase}. Choose $w\in N_{G(\bC)} (H\otimes_\bR \bC)$ such that $\bar{\Pi}=w\Pi$ as in Example \ref{ex:perfcase}. 
		Let $\lambda$ be a character of $H\otimes_\bR \bC$ such that $\bar{\lambda}=w\lambda$. Then $\lambda(\bar{w}w)\in\bR$ by
		\[\overline{\lambda(\bar{w}w)}
		=\bar{\lambda}(w\bar{w})
		=(w\lambda)(w\bar{w})
		=\lambda(\bar{w}w).\]
		Since $\lambda(\bar{w}w)$ is a root of unity, we deduce $\lambda(\bar{w}w)\in\{\pm 1\}$. In particular, the cocycle $\beta_\lambda$ is trivial in the Galois cohomology if and only if $\lambda(\bar{w}w)=1$.
	\end{ex}
	Despite the $\bC/\bR$ case, it is impossible in general to normalize $w$ so that $\lambda(\bar{w}w)\in\{\pm 1\}$ by arithmetic obstructions:
	\begin{ex}\label{ex:counterexample}
		Set $q\in\{\pm 1\}$ and $I_q=\diag(3,q)$. Define a reductive group $G_q$ over $\bQ$ as a group $\bQ$-space by
		\[G_q(A)=\{g\in\SL_2(A\otimes_\bQ \bQ(\sqrt{-1})):~
		g^\ast I_q g=I_q\}\]
		(see also Example \ref{ex:symmpair} and Example \ref{ex:modelofupq}). We remark that we have an isomorphism
		\[G_q\otimes_\bQ \bQ(\sqrt{-1})\cong \SL_2;\]
		\[a\otimes 1+b\otimes\sqrt{-1}\mapsto a+b\sqrt{-1}\]
		\[\frac{g+I_q^{-1}(g^T)^{-1}I_q}{2}\otimes 1+\frac{g-I_q^{-1}(g^T)^{-1}I_q}{2\sqrt{-1}}\otimes \sqrt{-1}\leftmapsto g.\]
		We endow the conjugate action on $\SL_2$ by this isomorphism. Let $B'\subset \SL_2$ be the Borel subgroup of upper triangular matrices. Set
		\[w=\left(\begin{array}{cc}
			0 & 1\\
			-1& 0
		\end{array}
		\right)\in\SL_2(\bQ(\sqrt{-1})).\]
		Then $w$ satisfies $\bar{B}'=wB'w^{-1}$. We also have
		\[\bar{w}w=\left(\begin{array}{cc}
			-\frac{q}{3} &0\\
			0 & -3q
		\end{array}
		\right).\]
		Define a character $\lambda:B'\to\bG_m$ by $\lambda\left(\left(\begin{array}{cc}
			x & y \\
			0 & x^{-1}
		\end{array}
		\right)\right)=x$.
		It satisfies $w\lambda=\bar{\lambda}$. Note that $\lambda(\bar{w}w)=-\frac{q}{3}$. In particular, no $c\in\bQ(\sqrt{-1})^\times$ can satisfy  $\frac{\lambda(\bar{w}w)}{\bar{c}c}=-q$.
		
		More generally, it is evident that $\beta_{n\lambda}$ is trivial if and only if $n$ is even. One can explain it from the perspectives of Example \ref{ex:CFT}. We study (the images of) $\beta_{n\lambda}$ in the second Galois cohomology groups of $\bQ_p$ and $\bR$. Let $p>3$ be a prime number. Notice that $\bQ_p(\sqrt{-1})$ is a possibly equal unramified extension of $\bQ_p$. Since the $p$-adic order of 3 is zero, $\inv_{\bQ_p}(\beta_{n\lambda})$ is trivial (see the construction of $\inv_{\bQ_p}$).
		
		Notice that $\bQ_3(\sqrt{-1})$ is an unramified quadratic extension of $\bQ_3$. It is easy to show that $\inv_{\bQ_3}(\beta_{n\lambda})=\frac{n}{2}\in\bQ/\bZ$. Therefore $\beta_{n\lambda}$ is trivial in the second Galois cohomology group of $\bQ_3$ if and only if $n$ is even. In particular, $\beta_{n\lambda}$ is trivial only if $n$ is even. The converse follows since $\beta_{n\lambda}=n\beta_\lambda$ is at most order 2 in the second Galois cohomology groups of $\bQ_2$ and $\bR$ (Example \ref{ex:CFT}, \cite{MR4174395} Example 6.13 (b)).
	\end{ex}
	
	\subsection{Reduction techniques}
	
	In the former section, we explained how we could descend equivariant line bundles on partial flag schemes. In this section, we see how to check the descent conditions which appeared in Section 2.1.
	
	\subsubsection{Existence of $w_\sigma$}
	In this section, we discuss how we can determine whether $w_\sigma$ exists and find $w_\sigma$. Let $k\to k'$ be a Galois extension of commutative rings or an infinite Galois extension of fields of Galois group $\Gamma$, $G$ be a reductive group scheme over $k$, and $P'$ be a parabolic subgroup of $G\otimes_k k'$.
	
	\begin{prop}\label{prop:wsigmaautomatic}
		Let $\sigma\in\Gamma$. Suppose that $P'$ contains a Borel subgroup $B'$. Moreover, assume that there exists an element $w_\sigma\in G(k')$ such that ${}^\sigma B'=w_\sigma B' w_\sigma^{-1}$. Then ${}^\sigma P'$ is $G(k')$-conjugate to $P'$, i.e., there exists an element $w'_\sigma\in G(k')$ such that ${}^\sigma P' = w'_\sigma P' (w'_\sigma)^{-1}$ if and only if ${}^\sigma P'=w_\sigma P' w_\sigma^{-1}$.
	\end{prop}
	
	\begin{proof}
		The ``if'' direction is clear. Suppose that the above $w'_\sigma$ exists. Since ${}^\sigma B'=w_\sigma B' w_\sigma^{-1}$, both ${}^\sigma P'=w'_\sigma P' (w'_\sigma)^{-1}$ and $w_\sigma P' w_\sigma^{-1}$ contain the Borel subgroup ${}^\sigma B'$. The subgroups ${}^\sigma P'$ and $w_\sigma P' w_\sigma^{-1}$ now coincide from \cite{MR0218364} Proposition 1.17.
	\end{proof}
	
	For this reason, we may choose such $w_\sigma$ and check ${}^\sigma P'=w_\sigma P' w_\sigma^{-1}$ if we are given the above $B'$ and $w_\sigma$.
	
	\begin{ex}\label{ex:quasisplit}
		Suppose that $P'$ contains a Borel subgroup which is defined over $k$. Then we may put $w_\sigma=e$. In particular, ${}^\sigma P'$ is $G(k')$-conjugate to $P'$ if and only if $P'$ is defined over $k$.
		In this case, $\beta_\lambda$ is trivial by definitions.
	\end{ex}
	
	\begin{ex}[\cite{MR546588} 1.10]
		Let $F$ be a nonarchimedian local field, and $\cO_F$ be the ring of its integers. Then every reductive group scheme $G$ over $\cO_F$ is quasi-split. In fact, let $\kappa_F$ be the residue field of $\cO_F$. Lang's theorem implies that its reduction to the finite field $\kappa_F$ is quasi-split. Apply Hensel's lemma (\cite{MR0238860} Th\'eor\`eme 18.5.17) to the flag scheme of $G$ to deduce the assertion.
	\end{ex}
	
	\begin{ex}
		Let $G_q$ and $\lambda$ be as in Example \ref{ex:counterexample}. Let $p>3$ be a prime number. Since $3\in\bZ_p^\times$, $G_q$ naturally admits the structure of a reductive group scheme over $\bZ_p$ (see Example \ref{ex:symmpair} if necessary). This gives another proof of the fact that $\beta_\lambda$ is trivial in the second Galois cohomology group of $\bQ_p$ with $p>3$.
	\end{ex}
	
	\begin{ex}\label{ex:settingsection4}
		Consider the setting of Section 2.1. Namely, assume that $\Spec k'$ is connected, and that there is a maximal torus $H$ of $G$ such that $(G\otimes_k k',H\otimes_k k')$ is split. Fix a simple system $\Pi$ of $(G\otimes_k k',H\otimes_k k')$. For each $\sigma\in\Gamma$, choose $w_\sigma\in N_{G(k')}(H\otimes_k k')$ such that ${}^\sigma\Pi=w_\sigma \Pi$.
		
		Let $B'$ be the Borel subgroup of $G\otimes_k k'$ attached to $(H\otimes_k k',\Pi)$, and $P'\subset G\otimes_k k'$ be a parabolic subgroup containing $B'$. Then $P'$ is attached to a subset $\Pi'\subset \Pi$ since $\Spec k'$ is connected. Fix $\sigma\in\Gamma$. Then the following conditions are equivalent:
		\begin{enumerate}
			\renewcommand{\labelenumi}{(\alph{enumi})}
			\item ${}^\sigma P'$ is $G(k')$-conjugate to $P'$;
			\item ${}^\sigma P'=w_\sigma P' w_\sigma^{-1}$;
			\item ${}^\sigma \Pi'=w_\sigma \Pi'$.
		\end{enumerate}
		In particular, finding parabolic subgroups $P'$ containing $B'$ such that ${}^\sigma P'$ is $G(k')$-conjugate to $P'$ for all $\sigma\in\Gamma$ is equivalent to finding elements of $(\type G)(k)$ in this setting. We also remark that we can obtain all parabolic types and partial flag schemes over $k$ from such $P'$.
	\end{ex}
	
	\subsubsection{Construct $\lambda\in X^\ast(P')$ with ${}^\sigma \lambda=w_\sigma\lambda$}
	In this section, we remark how we can find $\lambda\in X^\ast(P')$ and how we can check ${}^\sigma\lambda=w_\sigma\lambda$.
	
	\begin{lem}\label{lem:charofaddgrp}
		Let $k$ be a reduced commutative ring. Then a character of $\bG_a$ over $k$ is trivial.
	\end{lem}
	\begin{proof}
		Suppose that we are given a character $\chi:\bG_a\to \bG_m$ over $k$. Let $f=\sum_i a_i x^i\in k\left[t\right]$ be the corresponding polynomial. Since $\chi$ is a group homomorphism, we have
		\[a_0=1\]
		\[\begin{array}{ll}
			a_ia_j=\binom{i+j}{i}a_{i+j}&(i,j\geq 0).
		\end{array}\]
		
		The assertion follows if $a_n=0$ for $n\geq 1$; Otherwise, choose a positive integer $n$ such that $a_n\neq 0$. One can easily show by induction that $\frac{(mn)!}{(n!)^m}a_{mn}=a_n^m$ for $m\geq 1$. Since $k$ is reduced, $a_{mn}$ are nonzero. On the other hand, $a_{mn}=0$ for sufficiently large $m$ since $f$ is a polynomial and $n\geq 1$, a contradiction.
	\end{proof}
	\begin{prop}
		Let $G$ be a reductive group scheme over a scheme $S$, $H\subset G$ be a maximal torus, $Q\subset G$ be a parabolic subgroup containing $H$, and $L$ be the Levi subgroup of $Q$ with respect to $H$ (\cite{MR0218363} Lemme 5.11.3, Proposition 5.11.4, \cite{MR0218364} 1.7).
		\begin{enumerate}
			\renewcommand{\labelenumi}{(\arabic{enumi})}
			\item The restriction map $X^\ast(Q)\to X^\ast(L)$ is an isomorphism if $S$ is reduced.
			\item The restriction map $X^\ast(L)\to X^\ast(H)$ is injective.
		\end{enumerate}
		In particular, a character of $Q$ is determined by its restriction to $H$ if $S$ is reduced. 
	\end{prop}
	\begin{proof}
		Observe that the assertions are local in the \'etale topology of $S$. In view of \cite{MR0218364} Lemme 1.14, we may and do assume that the following conditions are satisfied:
		\begin{enumerate}
			\renewcommand{\labelenumi}{(\roman{enumi})}
			\item $S$ is nonempty affine;
			\item $(G,H)$ is split;
			\item $Q$ is attached to a parabolic subset $\Phi$ of $\Delta(G,H)$.
		\end{enumerate}
		Part (1) then follows from \cite{MR0218364} Section 2.1.1 and Lemma \ref{lem:charofaddgrp}. For (2), notice that $(L,H)$ is split reductive by definitions. Part (2) now follows from \cite{MR0218363} Proposition 6.1.8.
	\end{proof}
	\begin{ex}\label{ex:charofparabolics}
		Let $(G,H)$ be a split reductive group scheme over a connected scheme, and $Q$ be the parabolic subgroup of $G$ attached to a parabolic subset $\Phi$ of $\Delta(G,H)$. Then the restriction to $H$ determines an isomorphism
		\[X^\ast(Q)\cong
		\{\lambda\in X^\ast(H):~\langle\alpha^\vee,\lambda\rangle=0\ 
		{\rm for\ all\ }\alpha\in\Phi\cap (-\Phi)\}\]
		by \cite{MR0218363} Proposition 6.1.8, 6.1.1, and \cite{MR0212024} Th\'eor\`eme 3.1. Moreover, suppose that $\Phi$ contains a simple system $\Pi$ of $\Delta(G,H)$. Then for a character $\lambda\in X^\ast(H)$, $\langle\alpha^\vee,\lambda\rangle=0$ for all $\alpha\in\Phi\cap (-\Phi)$ if and only if $\langle\alpha^\vee,\lambda\rangle=0$ for all $\alpha\in\Phi\cap (-\Phi)\cap\Pi$.
	\end{ex}
	\begin{cor}
		Let $k\to k'$ be a Galois extension of reduced commutative rings or an infinite Galois extension of fields with Galois group $\Gamma$, $G$ be a reductive group scheme over $k$, $H'$ be a maximal torus of $G\otimes_k k'$, $P'$ be a parabolic subgroup of $G\otimes_k k'$ containing $H'$, and $\sigma\in\Gamma$. Suppose that we are given an element $w\in G(k')$ such that ${}^\sigma P'=w P' w^{-1}$ and ${}^\sigma H'=w H' w^{-1}$. Consider a character $\lambda\in X^\ast(H')$ lying in the image of the restriction map $X^\ast(P')\hookrightarrow X^\ast(H')$. Write $\lambda^{P'}$ be the preimage of $\lambda$ in $X^\ast(P')$. Then ${}^\sigma \lambda^{P'}=w\lambda^{P'}$ if and only if ${}^\sigma\lambda=w\lambda$.
	\end{cor}
	\begin{proof}
		This is immediate from the equalities
		\[\begin{array}{cc}
			(w\lambda^{P'})|_{wH'w^{-1}}=w\lambda,&
			({}^\sigma \lambda^{P'})|_{{}^\sigma H'}={}^\sigma \lambda.
		\end{array}\]
	\end{proof}
	\subsection{Application to irreducible representations of reductive algebraic groups}
	Let $G$ be a connected reductive algebraic group over a field $F$ of characteristic zero, and $H$ be its maximal torus. Choose a (possibly infinite) Galois extension $F'/F$ such that $H\otimes_F F'$ is split. In this section, we discuss existence and a geometric realization of $F$-forms of (absolutely) irreducible representations of $G\otimes_F F'$. We also study the classification of irreducible representations of $G$ when $F'$ is a quadratic extension.
	
	Fix a simple system $\Pi$ for $(G\otimes_F F',H\otimes_F F')$. Write $B'\subset G\otimes_F F'$ for the Borel subgroup attached to $(H\otimes_F F',\Pi)$. Let $V'$ be an (absolutely) irreducible representation of $G\otimes_F F'$ with lowest weight $\lambda$. For $\sigma\in\Gamma$, choose $w_\sigma\in N_{G(F')}(H\otimes_F F')$ such that ${}^\sigma\Pi=w_\sigma\Pi$. If $V'$ admits a descent datum, ${}^\sigma \lambda=w_\sigma \lambda$ holds for all $\sigma\in\Gamma$ by comparing the lowest weights of $V'$ and ${}^\sigma V'$. Let us assume this equality. Fix an isomorphism $\psi_\sigma:{}^\sigma V'\cong V'$. Then Borel-Tits' obstruction class $\beta_{V'}\in H^2(\Gamma,(F')^\times)$ for $F$-rationality of $V'$ is characterized by
	\[\psi_\sigma \circ {}^\sigma \psi_\tau
	\circ \psi_{\sigma\tau}^{-1}
	=\beta_{V'}(\sigma,\tau)\id_{V'}\]
	(see \cite{MR207712} Section 12 for details). Identify $V'$ with $\Gamma((G\otimes_F F')/B',\cL_\lambda)$. Put $\psi_\sigma=\Gamma((G\otimes_F F')/B',\Psi_\sigma)$ to get $\beta_{V'}=\beta_\lambda$.
	\begin{thm}\label{thm:descentofrep}
		An irreducible representation $V'$ of $G\otimes_F F'$ with lowest weight $\lambda$ is defined over $F$ if and only if the following conditions are satisfied:
		\begin{enumerate}
			\renewcommand{\labelenumi}{(\roman{enumi})}
			\item ${}^\sigma \lambda=w_\sigma \lambda$ for all $\sigma\in\Gamma$.
			\item The cocycle $\beta_\lambda$ is trivial in the Galois cohomology $H^2(\Gamma,(F')^\times)$.
		\end{enumerate}
		Moreover, if these equivalent conditions are satisfied, the $F$-form of $V'$ is realized as the space of global sections of the $F$-form of $\cL_{\lambda}$. 
	\end{thm}
	\begin{cor}[\cite{MR0209401} Section 3 Theorem 5, \cite{MR3770183} Remark 6.4]
		Suppose that $G$ admits a Borel subgroup $B$ over $F$. Choose a maximal torus $H\subset B$ and a Galois extension $F'/F$ such that $H\otimes_F F'$ is split. Let $V'$ be an irreducible representation $V'$ of $G\otimes_F F'$ with lowest weight $\lambda$.
		\begin{enumerate}
			\renewcommand{\labelenumi}{(\arabic{enumi})}
			\item The following conditions are equivalent:
			\begin{enumerate}
				\item[(a)] $V'$ admits an $F$-form;
				\item[(b)] ${}^\sigma V'\cong V'$ for all $\sigma\in\Gamma$;
				\item[(c)] $\lambda$ is defined over $F$;
				\item[(d)] ${}^\sigma \lambda=\lambda$ for all $\sigma\in\Gamma$;
			\end{enumerate}
			\item Suppose that the equivalent conditions of (1) are satisfied. Let $\lambda_F$ be the $F$-form of $\lambda$. Then $\Gamma(G/B,G\times^B \lambda_F)$ exhibits an $F$-form of $V'$.
		\end{enumerate}
	\end{cor}
	\begin{proof}
		Part (2) follows from \cite{MR217085} Proposition (1.4.15). We prove (1). It is clear that (a) implies (b). The implication (b) $\Rightarrow$ (d) follows by comparing the lowest weights. The conditions (b) and (c) are equivalent by Galois descent. To prove (d) $\Rightarrow$ (a), we may and do put $w_\sigma=e$ for all $\sigma\in\Gamma$ since $B$ is defined over $F$. In view of our choice of $w_\sigma$, $V'$ satisfies the conditions (i) and (ii) in Theorem \ref{thm:descentofrep}. In fact, $\beta_{\lambda}$ is the trivial cocycle. The assertion now follows from Theorem \ref{thm:descentofrep}.
	\end{proof}
	
	For further applications, let us concentrate on the case when $F'$ is a quadratic extension of $F$. We discuss the classification problem of irreducible representations of $G$ in the spirit of Cartan in \cite{E1914}: Let $\Irr G$ (resp.\ $\Irr (G\otimes_F F')$) denote the set of isomorphism classes of irreducible representations of $G$ (resp.\ $G\otimes_F F'$). Set
	\[\Irr_a G=\{V\in\Irr G:~V\otimes_F F'\ \mathrm{is\ irreducible}\}.\]
	Define subsets $\Irr_i(G\otimes_F F')\subset \Irr (G\otimes_F F')$ for $i\in\{0,1,-1\}$ by
	\[\Irr_1(G\otimes_F F')
	=\{V'\in \Irr (G\otimes_F F'):~V'\ \textrm{admits\ an}\ F\textrm{-form}\}\]
	\[\Irr_0(G\otimes_F F')
	=\{V'\in \Irr (G\otimes_F F'):~\bar{V}'\not\cong V'\}\]
	\[\Irr_{-1}(G\otimes_F F')
	=\{V'\in \Irr (G\otimes_F F'):~\Res_{F'/F}V'\ \textrm{is\ irreducible},~\bar{V}'\cong V'\},\]
	where $\Res_{F'/F} V'$ denotes the restriction of scalar (see \cite{MR3853058} Proposition 3.1.1). Define an equivalence relation on $\Irr_0(G\otimes_F F')$ by $\bar{V}\sim V$ for $V\in \Irr_0(G\otimes_F F')$. Then we have bijections
	\[\Irr_a G\cong\Irr_1(G\otimes_F F')\]
	\[\Irr G\cong\Irr_1(G\otimes_F F')\coprod \Irr_0(G\otimes_F F')/\sim\coprod \Irr_{-1}(G\otimes_F F').\]
	In fact, observe that every irreducible representation of $G$ is embedded to an irreducible representation of $G\otimes_F F'$ by induction. Since $G\otimes_F F'$ is split, $\Res_{F'/F} V'$ is reducible if and only $V'$ is defined over $F$ (see \cite{MR3770183} Proposition 5.1). Moreover, the $F$-form is unique up to isomorphisms (\cite{MR3770183} Corollary 3.4, Proposition 3.5). The assertion now follows from \cite{MR3770183} Proposition 5.1 and the proof of Proposition 5.2.
	
	For a numerical description of the three subsets $\Irr_i(G\otimes_F F')$ ($i\in\{0,\pm 1\}$), write $X^\ast(H\otimes_F F')_-$ for the set of anti-dominant characters of $H\otimes_F F'$. Let $w$ be an element of $N_{G(F')}(H\otimes_F F')$ such that $\bar{\Pi}=w\Pi$.
	\begin{cor}
		\begin{enumerate}
			\renewcommand{\labelenumi}{(\arabic{enumi})}
			\item The correspondence $\Irr (G\otimes_F F')\cong X^\ast(H\otimes_F F')_-$ obtained by taking the lowest weights restricts to the following bijections:
			\[\Irr_1 (G\otimes_F F')\cong\{\lambda\in X^\ast(H\otimes_F F')_-:~\bar{\lambda}= w\lambda,~\exists c\in F'\mathrm{\ such\ that\ } \lambda(\bar{w}w)=\bar{c}c\}\]
			\[\Irr_0 (G\otimes_F F')\cong\{\lambda\in X^\ast(H\otimes_F F')_-:~\bar{\lambda}\neq w\lambda\}\]
			\[\Irr_{-1} (G\otimes_F F')\cong\{\lambda\in X^\ast(H\otimes_F F')_-:~\bar{\lambda}= w\lambda,~\forall c\in F,~\lambda(\bar{w}w)\neq\bar{c}c\}.\]
			\item The equivalence relation on the set $\{\lambda\in X^\ast(H\otimes_F F')_-:~\bar{\lambda}\neq w\lambda\}$ induced from that on $\Irr_0 (G\otimes_F F')$ via the bijection of (1) is generated by $\lambda\sim \overline{w\lambda}$.
			\item If we are given a character $\lambda\in X^\ast(H\otimes_F F')_-$ and $c\in F'$ satisfying $\bar{\lambda}= w\lambda$ and $\lambda(\bar{w}w)=\bar{c}c$, the corresponding absolutely irreducible representation of $G$ is realized as the space of global sections of the $F$-form of $\cL_\lambda$.
		\end{enumerate}
	\end{cor}
	\begin{ex}
		Let $G_q$ and $B'$ be as in Example \ref{ex:counterexample}. Let $H\subset G_q$ be the subgroup consisting of diagonal matrices. Then $H$ is a maximal torus of $G_q$ such that $H\otimes_\bQ \bQ(\sqrt{-1})$ is split and $H\otimes_\bQ \bQ(\sqrt{-1}) \subset B'$. For an integer $n$, define $\lambda_n\in X^\ast(H\otimes_F F')$ by
		\[\lambda_n\left(\left(\begin{array}{cc}
			a & 0 \\
			0 & a^{-1}
		\end{array}
		\right)\right)=a^n.\]
		Then we have the following bijections:
		\[\Irr_1 (G_q\otimes_\bQ \bQ(\sqrt{-1}))\cong
		\{\lambda_{-2n}\in X^\ast(H\otimes_F F'):~n\in\bN\}\]
		\[\Irr_0 (G_q\otimes_\bQ \bQ(\sqrt{-1}))=\emptyset\]
		\[\Irr_{-1} (G_q\otimes_\bQ \bQ(\sqrt{-1}))\cong
		\{\lambda_{-2n-1}\in X^\ast(H\otimes_F F'):~n\in\bN\}.\]
	\end{ex}
	\begin{ex}
		One can show that the real flag variety of $\SO(3)$ (see \cite{MR0218363} Corollaire 5.8.3) is isomorphic to the subvariety of the real projective space $\bP^2$ defined by the 
		homogeneous polynomial $x^2+y^2+z^2$. Indeed, we will prove a similar result over $\bZ\left[1/2\right]$ in \cite{hayashikgb}. This leads to the model of irreducible representations of $\SO(3)$ in the beginning of Section 1.
	\end{ex}
	\section{The standard forms}
	In this section, we construct standard $\bZ\left[1/2\right]$-forms of classical connected Lie groups, their maximal compact subgroups, and their fundamental Cartan subgroups by using involutions as preliminaries to Section 4 of concrete applications of the abstract theory in Section 2.
	\subsection{Symmetric and fundamental Cartan subgroups}
	In Section 2.2 we explained how to descend equivariant line bundles on partial flag schemes. We also saw in Section 2.3 how we could check the descent conditions by using maximal tori. In this section, we introduce symmetric pairs and a special class of maximal tori to construct reductive groups schemes and their maximal tori. Throughout this section, let $k$ be a commutative $\bZ\left[1/2\right]$-algebra.
	
	Firstly, we supply a way to construct pairs $(G,K)$ of standard $\bZ\left[1/2\right]$-forms of classical Lie groups and their maximal compact subgroups. The following observation is helpful:
	\begin{lem}\label{lem:fixedpoint}
		Let $X$ be a smooth separated scheme over $k$, equipped with an involution $\theta$. Then the $k$-space
		\[X^{\theta}:R\mapsto \{x\in X(R):~\theta(x)=x\}\]
		is represented by a smooth closed subscheme of $X$ over $k$.
	\end{lem}
	\begin{proof}
		Since $X^\theta$ is isomorphic to the intersection of $\{(x,x)\in X\times X:~x\in X\}\cong X$ and $\{(x,\theta(x))\in X\times X:~x\in X\}\cong X$ in $X\times X$, $X^\theta$ is represented by a closed subscheme of $X$ locally of finite presentation over $k$.
		
		We prove that $X^\theta$ is smooth over $k$. Choose an affine open covering $X=\cup_i U_i$. Then the $\theta$-stable open subscheme $\cup_i (U_i\cap \theta(U_i))\subset X$ contains $X^\theta$. For each index $i$, $U_i\cap \theta(U_i)$ is affine since $X$ is separated. Therefore we may replace $X$ by $\cup_i (U_i\cap \theta(U_i))$ to assume that $U_i$ are $\theta$-stable. Since $U_i^\theta\cong U_i\times_X X^\theta$, we may replace $X$ by $U_i$ to assume that $X$ is affine.
		
		Write $A$ for the coordinate ring of $X$. We denote the involution on $A$ corresponding to $\theta$ by the same symbol $\theta$.
		Set
		\[A^0=\{a\in A:~\theta(a)=a\}\]
		\[A^1=\{a\in A:~\theta(a)=-a\}.\]
		Since $2$ is a unit of $A$, we have a decomposition $A=A^0\oplus A^1$ as a $k$-module. Let $I$ be the ideal of $A$ generated by $A^1$. It is easy to show that for a $k$-algebra $B$ and a $k$-algebra homomorphism $f:A\to B$, $f$ belongs to $X^\theta(B)$ if and only if $f(A^1)=0$. Hence $X^\theta$ is represented by $\Spec A/I$.
		
		The proof is completed by showing that $X^\theta$ is formally smooth over $k$. Suppose that we are given a diagram of commutative rings
		\[\begin{tikzcd}
			k\ar[r, "f"]\ar[d]&R\ar[dd]\\
			A\ar[d]\ar[ru, dashed, "h"]&\\
			A/I\ar[r, "g"]&R/J.
		\end{tikzcd}\]
		Here $J$ is a square-zero ideal of $R$, the map $k\to A$ is the structure homomorphism, and the other unlabeled arrows $A\to A/I$ and $R\to R/J$ are the quotient maps. The dashed arrow $h$ exists since $X$ is formally smooth. We denote the projection map $A\to A^0$ by $p$. Then $h|_{A^0}\circ p$ descends to a ring homomorphism $A/I\to R$. This is a solution of the lifting problem.
	\end{proof}
	\begin{ex}[symmetric pair]\label{ex:symmpair}
		Let $G$ be a reductive group scheme over $k$, equipped with an involution $\theta$. Write $(G^\theta)^\circ$ for the unit component of $G^\theta$ (\cite{MR0217086} 15.6.5). Then $(G^\theta)^\circ$ is a closed reductive subgroup scheme by \cite{MR3362641} Proposition 3.1.3. Note that $(G^\theta)^\circ=G^\theta$ if every geometric fiber of $G$ is simply connected (\cite{MR0230728} Theorem 8.1). We call an open and closed subgroup scheme $H\subset G^\theta$ a symmetric subgroup of $G$, and $(G,H)$ a symmetric pair.
	\end{ex}
	We will give examples in Section 3.3. Let us next introduce a special class of maximal tori for symmetric pairs.
	\begin{thm}\label{thm:fundamentalcartan}
		Let $G$ be a reductive group scheme over $k$, $K\subset G$ be a symmetric subgroup with connected geometric fibers, and $T$ be a maximal torus of $K$. Then $Z_G(T)$ is a maximal torus of $G$. We call $Z_G(T)$ a fundamental Cartan subgroup of $G$.
	\end{thm}
	We analyze $Z_G(T)$ by using the dynamic method.
	\begin{defn}\label{defn:regcochar}
		We say a cocharacter $\mu:\bG_m\to G$ is regular if $Z_G(\mu)$ is a torus.
	\end{defn}
	\begin{lem}\label{lem:toruscriterion}
		For a reductive group scheme $H$ over a scheme $S$, the following conditions are equivalent:
		\begin{enumerate}
			\renewcommand{\labelenumi}{(\alph{enumi})}
			\item $H$ is a torus.
			\item $H$ is a commutative group scheme.
			\item We have $Z_H=H$.
			\item For every geometric point $\bar{s}$ of $S$, $H_{\bar{s}}$ is a torus.
			\item For every geometric point $\bar{s}$ of $S$, $H_{\bar{s}}$ is a commutative algebraic group.
			\item For every geometric point $\bar{s}$ of $S$, $Z_{H_{\bar{s}}}=H_{\bar{s}}$.
		\end{enumerate}
	\end{lem}
	\begin{proof}
		The implications (a)$\Rightarrow$ (b) $\Leftrightarrow$ (c) $\Rightarrow$ (f) $\Leftrightarrow$ (e) $\Leftarrow$ (d) are obvious. \cite{MR3362641} Theorem 5.3.1 implies (b) $\Rightarrow$ (a) and (d) $\Rightarrow$ (e).
		
		Finally, suppose that $G$ satisfies (d). Since $H$ is reductive, $Z_H$ is of multiplicative type in the sense of \cite{MR3362641} Definition B.1.1 (\cite{MR3362641} Theorem 3.3.4). In particular, $Z_H$ is flat of finite presentation over $S$ (see below \cite{MR3362641} Definition B.1.1). The condition (c) now follows from \cite{MR3362641} Lemma B.3.1.
	\end{proof}
	\begin{lem}\label{lem:regcochar}
		For a cocharacter $\mu$ of $G$, the following conditions are equivalent:
		\begin{enumerate}
			\renewcommand{\labelenumi}{(\alph{enumi})}
			\item $\mu$ is regular.
			\item $Z_G(\mu)$ is a maximal torus.
			\item For every geometric point $\bar{s}$ of $\Spec k$, $\mu_{\bar{s}}$ factors through a maximal torus $H$ such that for every root $\alpha\in\Delta(G_{\bar{s}}, H)$, we have $\langle \mu_{\bar{s}},\alpha\rangle\neq 0$.
			\item $P_G(\mu)$ is a Borel subgroup.
		\end{enumerate}
		In particular, $\mu$ factors through a unique maximal torus if $\mu$ is regular.
	\end{lem}
	\begin{proof}
		In view of Lemma \ref{lem:toruscriterion}, we may pass to geometric fibers to assume that $k=F$ is an algebraically closed field. Consider the following condition:
		\begin{enumerate}
			\item[(c)'] $Z_G(\mu)$ is a maximal torus. Moreover, for every root $\alpha\in\Delta(G, Z_G(\mu))$, $\langle \mu,\alpha\rangle\neq 0$.
		\end{enumerate}
		We plan to prove (a) $\Rightarrow$ (b) $\Rightarrow$ (c)' $\Rightarrow$ (c) $\Rightarrow$ (d) $\Rightarrow$ (a).
		
		Suppose that (a) is satisfied. It follows by definition that $\mu$ maps to $Z_G(\mu)$. Let $H\subset G$ be a torus containing $Z_G(\mu)$. Then $H$ is contained in $Z_G(\mu)$ by definition of $Z_G(\mu)$. This shows (b).
		
		Suppose that $\mu$ satisfies (b). Put $H=Z_G(\mu)$, and $\fh$ be its Lie algebra. For $\alpha\in\Delta(G,Z_G(\mu))\cup\{0\}$, write $\mathfrak{g}_\alpha$ for the $\alpha$-weight space. It follows by definition that $\bG_m$ acts on $\mathfrak{g}_\alpha$ through $\mu$ by $\langle\mu,\alpha\rangle$. \cite{MR3362817} Proposition 2.1.8 implies $\fh=\fh\oplus\oplus_{\overset{\alpha\in\Delta(G,H)}{
				\langle \mu,\alpha\rangle=0	}}\mathfrak{g}_\alpha$. This shows (c)', i.e., no root $\alpha$ can satisfy $\langle \mu,\alpha\rangle=0$.
		
		Since (c)' is stable under extension of algebraically closed field, (c)' implies (c).
		
		The implication (c) $\Rightarrow$ (d) follows from the proof of \cite{MR3362817} Proposition 2.2.9 (if necessary, see also \cite{MR1102012} 13.8).
		
		Set $\Delta^+(G,H)\coloneqq\{\alpha\in\Delta(G,H):~\langle\mu,\alpha\rangle>0\}$. Then the Lie algebra of $P_G(\mu)$ is $\fh\oplus\oplus_{\alpha\in\Delta^+(G,H)}\mathfrak{g}_\alpha$ by \cite{MR3362641} Proposition 2.1.8, where $\fh$ is the Lie algebra of $H$. Suppose that $\mu$ satisfies (c). Then $\Delta^+(G,H)$ is a positive system of $\Delta(G,H)$. In view of the proof of \cite{MR3362641} Proposition 1.4.4, we can attach a Borel subgroup $B$ of $G$ whose Lie algebra is $\fh\oplus\oplus_{\alpha\in\Delta^+(G,H)}\mathfrak{g}_\alpha$. We can conclude that $P_G(\mu)=B$ from \cite{MR3362817} Proposition 2.18 and \cite{MR3729270} Proposition 10.15. In particular, (d) holds.
		
		Finally, (d) implies (a) from \cite{MR3362641} Example 4.1.9 and \cite{MR1102012} 10.6 Theorem (2).
	\end{proof}
	\begin{cor}\label{cor:regcochar}
		Let $\mu$ be a cocharacter of $K$. If $\mu$ is regular in $G$, $\mu$ is so in $K$.
	\end{cor}
	\begin{proof}
		Recall that $Z_K(\mu)$ is reductive. The assertion is immediate from $Z_K(\mu)\subset Z_G(\mu)$ and Lemma \ref{lem:toruscriterion} (b).
	\end{proof}
	\begin{lem}\label{lem:funcar/acfield}
		Theorem \ref{thm:fundamentalcartan} holds if $k=F$ is an algebraically closed field.
	\end{lem}
	\begin{proof}
		Fix an involution $\theta$ of $G$ such that $K=(G^\theta)^\circ$.
		
		Choose a $\theta$-stable Borel subgroup of $G$ and a $\theta$-stable maximal torus
		$H\subset B$. Note that the existence follows from \cite{MR0230728} Theorem 7.5. In fact, we can easily see that an involution $\theta'$ of an algebraic group $G'$ over an algebraically closed field of characteristic $\neq 2$ is semisimple in the sense of \cite{MR0230728} Section 7 by embedding $G'$ into the semidirect product $G'\rtimes \bZ/2\bZ$ attached to $G'$ and $\theta'$. There exists a cocharacter $\mu\in X_\ast(H)$ such that $B = P_G(\mu)$ (see the proof of \cite{MR3362817} Proposition 2.2.9). In particular, the set $\Delta(B,H)$ of roots of $B$ with respect to $H$ is given by
		\[\Delta(B,H)=\{\alpha\in \Delta(G,H):~
		\langle\mu,\alpha\rangle>0\}.\]
		For $\alpha\in\Delta(B,H)$, we have
		\[\begin{split}
			\langle\mu+\theta\mu,\alpha\rangle
			&=\langle\mu,\alpha\rangle
			+\langle\theta\mu,\alpha\rangle\\
			&=\langle\mu,\alpha\rangle
			+\langle\mu,\theta\alpha\rangle\\
			&>0.
		\end{split}\]
		The last inequality follows since $\Delta(B,H)$ is a $\theta$-stable subset of $\Delta(G,H)$. Hence we may replace $\mu$ by $\mu+\theta\mu$ to assume that $\mu$ is $\theta$-invariant. Since $\bG_m$ is connected, $\mu$ factors through $K$. Moreover, $\mu$ is regular as a cocharacter of $K$ by Lemma \ref{lem:regcochar} and Corollary \ref{cor:regcochar}. In particular, $Z_K(\mu)$ is a maximal torus of $K$ by Lemma \ref{lem:regcochar}.
		
		Choose $g\in K(F)$ such that $gZ_K(\mu)g^{-1}=T$. It is clear that we have
		\[Z_G(g\mu)=gZ_G(\mu)g^{-1}=gHg^{-1}.\]
		Since $g\mu$ is a cocharacter to $gZ_K(\mu)g^{-1}=T$, we have $Z_G(T)\subset Z_G(g\mu)$. Since $T=gZ_K(\mu)g^{-1}\subset gZ_G(\mu)g^{-1}=gHg^{-1}$, $Z_G(gHg^{-1})\subset Z_G(T)$ holds. We combine these relations of containment with $Z_G(gHg^{-1})=gHg^{-1}$ to deduce $gHg^{-1}=Z_G(T)$. This shows that $Z_G(T)$ is a maximal torus of $G$.
	\end{proof}
	\begin{proof}[Proof of Theorem \ref{thm:fundamentalcartan}]
		The group scheme $Z_G(T)$ is smooth over $k$, and it is a closed subgroup of $G$ by \cite{MR3362641} Lemma 2.2.4. In particular, $Z_G(T)$ is a smooth affine group scheme. Moreover, $Z_G(T)$ is a reductive group scheme by \cite{MR3362641} Theorem 1.1.19 3.
		
		To prove the assertion, we may pass to geometric fibers by Lemma \ref{lem:toruscriterion} and the definition of maximal tori. We then win by Lemma \ref{lem:funcar/acfield}.
	\end{proof}
	For concrete applications in Section 4, let us show an analog of a basic property of fundamental Cartan subalgebras in representation theory of real reductive Lie groups.
	\begin{lem}\label{lem:basicoffuncar}
		Let $G$ be a reductive group scheme over $k$, $K$ be a symmetric subgroup with connected geometric fibers, and $T$ be a maximal torus of $K$. Suppose that $(G,Z_G(T))$ is split. Then for every root $\alpha$ of $G$ with respect to $Z_G(T)$, $\alpha|_T$ is nonzero.
	\end{lem}
	\begin{proof}
		Write $H=Z_G(T)$. Recall that $\alpha$ determines an $H$-invariant unipotent subgroup $U_\alpha\subset G$. Moreover, $H$ acts on $U_\alpha\cong\bG_a$ by $\alpha$ (\cite{MR3362641} Theorem 4.1.4, Definition 4.2.3, Definition 5.1.1 2).
		
		Suppose that $\alpha|_T = 0$. Then we have $U_\alpha\subset Z_G(T)=H$, which contradicts to the fact that $U_\alpha\cap H$ is trivial (\cite{MR3362641} Theorem 4.1.4).
	\end{proof}
	Let $k\to k'$ be a Galois extension of commutative $\bZ\left[1/2\right]$-algebras of Galois group $\Gamma=\bZ/2\bZ$ such that $\Spec k'$ is connected. Let $G$ be a reductive group scheme over $k$. Set $K = (G^\theta)^\circ$. Let $T$ be a maximal torus $T$ of $K$. Set $H = Z_G(T)$. Suppose that $(G\otimes_k k',H\otimes_k k')$ and $(K\otimes_k k',T\otimes_k k')$ are split.
	\begin{cons}[\cite{MR1920389} Section II.5]\label{cons:simplesys}
		Fix an ordered free basis $\{\epsilon_1,\epsilon_2,\ldots,\epsilon_r\}$ of $X_\ast(T\otimes_k k')$. Then we define a positive system of $(G\otimes_k k',H\otimes_k k')$ by the lexicographic way. That is, a root $\alpha$ is positive if and only if there exists an index $i\in\{1,2,\ldots,r\}$ such that $\langle \epsilon_i,\alpha\rangle>0$ and $\langle \epsilon_l,\alpha\rangle=0$ for $1\leq l\leq i-1$. Note that
		this determines a positive system in fact since $\langle \epsilon_i,\alpha\rangle$ is nonzero for at least one index $i\in\{1,2,\ldots,r\}$ by Lemma \ref{lem:basicoffuncar}.
	\end{cons}
	\begin{prop}\label{prop:conj=-1}
		Suppose that the Galois action on $X^\ast(T\otimes_k k')$ is $-1$.
		\begin{enumerate}
			\renewcommand{\labelenumi}{(\arabic{enumi})}
			\item We have $\bar{\Pi}=-\Pi$.
			\item The element $w$ of the Weyl group attached to the nontrivial element of $\Gamma$ is the longest element with respect to the above positive system.
			\item We have $\bar{w}=w$ in the Weyl group.
			\item If we take a representative of $w$ in $N_{G(k')}(T\otimes_k k')$ which we denote by the
			same symbol $w$, then $\bar{w}w$ belongs to $T(k')$.
		\end{enumerate}
	\end{prop}
	\begin{proof}
		Part (1) follows by the definition of our positive system. Part (2) immediately follows from (1). Since $\bar{w}(\Pi)=-\Pi$ by (1), (3) holds. Part (4) follows by (3) and the definition of the Weyl group scheme.
	\end{proof}
	The above observation is quite helpful when we study the descent property of line bundles on partial flag schemes of standard models of classical Lie groups (see Section 4). We will see examples of $T$ satisfying the condition in Proposition \ref{prop:conj=-1} in Section 3.2 and 3.4.
	
	A further assumption sometimes makes it easy to find the lift $w$:
	\begin{prop}
		Consider the setting of Proposition \ref{prop:conj=-1}. Suppose that the Weyl group $W(K,T)(k')$ contains $-1$. Let $w\in K(k')$ be its lift. Then $w$ represents the longest element of $W(G,H)(k')$.
	\end{prop}
	\begin{proof}
		It will suffice to show that $w$ sends $\Delta^+(G,H)$ to $-\Delta^+(G,H)$. For $\alpha\in\Delta^+(G,H)$ and $i\in\{1,2,\ldots,r\}$, we have
		\[\langle\epsilon_i,w\alpha\rangle
		=\langle w^{-1}\epsilon_i,\alpha\rangle
		=-\langle\epsilon_i,\alpha\rangle.\]
		Hence $w\alpha$ is a negative root by the definition of our positive system. This completes the proof.
	\end{proof}
	\subsection{Examples of nonsplit tori}
	The aim of this section is four-fold:
	\begin{enumerate}
		\item Supply examples $T$ of non-split tori over commutative rings $k$.
		\item Find Galois extensions $k'$ of $k$ such that $T\otimes_k k'$ are split. 
		\item Compute the Galois action on $X^\ast(T\otimes_k k')$.
		\item ``Express'' the Galois actions in terms of matrices for some of $T$.
	\end{enumerate}
	The aim 3 is for showcases of the computation of the Galois action of $X^\ast(H\otimes_k k')$ for the maximal tori $H$ of models of classical Lie groups which we will list in the next section. We work with 4 in order to give key computations to find the lift $w$ of Proposition \ref{prop:conj=-1} (2).
	
	The first example is the $\bZ\left[1/2\right]$-form of the compact torus of rank $1$ with three different realizations. One is $\SO(2)$. In fact, we have an isomorphism $\mu_2:\bG_m\cong\SO(2)$ over $\bZ\left[1/2,\sqrt{-1}\right]$ which is defined by
	\[\mu_2(a)=g_2 \left(\begin{array}{cc}
		a&0\\
		0&a^{-1}
	\end{array}\right)g_2^{-1}=\left(\begin{array}{cc}
		\frac{a+a^{-1}}{2}&-\frac{a-a^{-1}}{2\sqrt{-1}}\\
		\frac{a-a^{-1}}{2\sqrt{-1}}&\frac{a+a^{-1}}{2}
	\end{array}\right)\]
	for every $\bZ\left[1/2,\sqrt{-1}\right]$-algebra $A'$ and unit $a$ of $A'$. It is clear that $\bar{\mu}_2(a)=\mu_2(a)^{-1}$. This shows that the conjugation action on \[X^\ast(\SO(2)\otimes_{\bZ\left[1/2\right]} \bZ\left[1/2,\sqrt{-1}\right])\cong\bZ\]
	is equal to $-1$. We also have $w_2\mu_2=-\mu_2$. Note that $w_2\in\Oo(2,\bZ)\setminus\SO(2,\bZ)$.
	
	The second realization is $\Uu(1)$. In fact,
	\[a\otimes 1+b\otimes\sqrt{-1}\mapsto\left(\begin{array}{cc}
		a & -b \\ 
		b & a
	\end{array} \right)\]
	determines an isomorphism $\Uu(1)\cong\SO(2)$ over $\bZ\left[1/2\right]$. We denote the composition $\Uu(1)\otimes_{\bZ\left[1/2\right]} \bZ\left[1/2,\sqrt{-1}\right]\to\bG_m$ of the base change of the inverse of this isomorphism to $\bZ\left[1/2,\sqrt{-1}\right]$ with $\mu_2$ by the same symbol $\mu_2$. That is, we have
	\[\mu_2(a)=\frac{a+a^{-1}}{2}\otimes 1+\frac{a-a^{-1}}{2\sqrt{-1}}\otimes\sqrt{-1}.\]
	
	The third realization is simply given by
	\[A\mapsto \left\{\left(\begin{array}{cc}
		a & 0 \\ 
		0 & a^{-1}
	\end{array} \right)\in \Res_{\bZ\left[1/2,\sqrt{-1}\right]/\bZ\left[1/2\right]}\SL_2:~a\in \Uu(1,A)\right\}\cong \Uu(1,A).\]
	Note that we have
	\[\left(\begin{array}{cc}
		0 & 1 \\ 
		1 & 0
	\end{array} \right)\left(\begin{array}{cc}
		a & 0 \\ 
		0 & a^{-1}
	\end{array} \right)\left(\begin{array}{cc}
		0 & 1 \\ 
		1 & 0
	\end{array} \right)=\left(\begin{array}{cc}
		a^{-1} & 0 \\ 
		0 & a
	\end{array} \right).\]
	
	The second example is the torus
	\begin{flalign*}
		T=&\left\{\left(\begin{array}{cccccc}
			a_1 &  &  & -b_1 &  &  \\
			& \ddots &  &  & \ddots &  \\
			&  & a_n &  &  & -b_n \\
			b_1 &  &  & a_1 &  &  \\
			& \ddots &  &  & \ddots &  \\
			&  & b_n &  &  & a_n
		\end{array}
		\right)\in\GL_{2n}:~a_i^2+b_i^2=1{\rm \ for\ }1\leq i\leq n
		\right\}\\
		&\cong\Uu(1)^n
	\end{flalign*}
	over $\bZ\left[1/2\right]$ for a positive integer $n$. In fact, the conjugation by $g_{2n}'$ determines an isomorphism from $T\otimes_{\bZ\left[1/2\right]} \bZ\left[1/2,\sqrt{-1}\right]$ to the split maximal torus of diagonal matrices in $\Sp_n$ or $\SO_{2n}'$. The above computation shows that the conjugation action on \[X^\ast(T\otimes_{\bZ\left[1/2\right]} \bZ\left[1/2,\sqrt{-1}\right])\cong\bZ^n\]
	is equal to $-1$.
	
	The final example is obtained by the Weil restriction. Fix a Galois extension $k\to k'$ of Galois group $\Gamma$ and a split torus $H'$ over $k'$.
	\begin{lem}\label{lem:galtwistofreschar}
		Let $\sigma\in\Gamma$, $A'$ be a commutative $k'$-algebra, and $X'$ be a $k'$-scheme.
		\begin{enumerate}
			\renewcommand{\labelenumi}{(\arabic{enumi})}
			\item Define a map $\delta_{A'}:A'\otimes_k k'\to \prod_{\tau\in\Gamma} {}^{\tau^{-1}}A'$ by $a\otimes c\mapsto (a\tau(c))$. Then $\delta_{A'}$ is a $k'$-algebra isomorphism. Here $A'\otimes_k k'$ is regarded as a $k'$-algebra for the canonical homomorphism from $k'$ onto the right factor $k'$.
			\item We have a natural isomorphism $\phi:(\Res_{k'/k} X')\otimes_k k'\to \prod_{\tau} {}^\tau X'$.
			\item The diagram 
			\[\begin{tikzcd}
				((\Res_{k'/k} X')\otimes_k k')(A')
				\ar[r, "\phi_{A'}"]
				&\prod_\tau {}^\tau X'\ar[d, "s_{\sigma,A'}"]\\
				({}^\sigma((\Res_{k'/k} X')\otimes_k k'))(A')
				\ar[r, "{}^\sigma\phi_{A'}"]\ar[u, equal]
				&\prod_\tau {}^{\sigma\tau}X'(A')
			\end{tikzcd}\]
			commutes.
		\end{enumerate}
	\end{lem}
	\begin{proof}
		Part (1) follows by definitions. Part (2) immediately follows from (1). In fact, we have
		\[\begin{split}
			((\Res_{k'/k} X')\otimes_k k')(A')
			&=X'(A'\otimes_k k')\\
			&\cong X'(\prod_\tau {}^{\tau^{-1}}A')\\
			&\cong \prod_\tau X'({}^{\tau^{-1}}A')\\
			&=\prod_\tau {}^\tau X'(A').
		\end{split}\]
		
		For (3), observe that the diagram of $k'$-algebras
		\[\begin{tikzcd}
			A'\otimes_k k'\ar[r, "\delta_{A'}"]\ar[d, equal]
			&\prod_\tau {}^{\tau^{-1}} A'\ar[d, "s_\sigma"]\\
			{}^{\sigma^{-1}} A'\otimes_k k'
			\ar[r, "\delta_{{}^{\sigma^{-1}} A'}"]
			&\prod {}^{\tau^{-1}\sigma^{-1}} A'
		\end{tikzcd}\]
		commutes. Apply $X'$ to get
		\[\begin{tikzcd}
			X'(A'\otimes_k k')\ar[rr, "X'(\delta_{A'})"]\ar[d, equal]
			&&X'(\prod_\tau {}^{\tau^{-1}} A')\ar[d]\ar[r, equal]
			&\prod_\tau {}^\tau X'(A')\ar[d, "s_{\sigma,A'}"]\\
			X'({}^{\sigma^{-1}} A'\otimes_k k')
			\ar[rr, "X'(\delta_{{}^{\sigma^{-1}} A'})"]
			&&X'(\prod {}^{\tau^{-1}\sigma^{-1}} A')\ar[r, equal]
			&\prod {}^{\sigma\tau}X'(A').
		\end{tikzcd}\]
		The assertion now follows by unwinding the definitions.
	\end{proof}
	By this observation, $\Res_{k'/k} H'$ is a torus which is split over $k'$.
	\begin{prop}\label{prop:galactionofchargrprestorus}
		Identify $X^\ast((\Res_{k'/k} H')\otimes_k k')$ with $\prod_{\tau} X^\ast({}^\tau H')$ by Lemma \ref{lem:galtwistofreschar}. Then the induced Galois action on $\prod_{\tau} X^\ast({}^\tau H')$ is given by
		\[{}^\sigma(\lambda_\tau)=({}^\sigma\lambda_{\sigma^{-1}\tau}),\]
		where $\sigma\in\Gamma$ and $(\lambda_\tau)\in \prod_{\tau} X^\ast({}^\tau H')$.
	\end{prop}
	\begin{proof}
		Let $(\lambda_\tau)\in \prod_{\tau} X^\ast({}^\tau H')$, and $\Lambda\in X^\ast((\Res_{k'/k} H')\otimes_k k')$ be the corresponding element. For $\tau'\in\Gamma$, let $i_{\tau'}:{}^{\tau'} H'\to \prod_{\tau} {}^\tau H'$ be the canonical homomorphism. The assertion follows from the equality ${}^\sigma\Lambda\circ \phi^{-1}\circ i_{\tau'}={}^\sigma (\Lambda\circ\phi^{-1}\circ i_{\sigma^{-1}\tau'})={}^\sigma\lambda_{\sigma^{-1}\tau}$ which holds by Lemma \ref{lem:galtwistofreschar} (3):
		\[\begin{tikzcd}
			H'({}^{(\tau')^{-1}}A')\ar[r, "i_{\tau', A'}"]
			\ar[rd, "{}^\sigma i_{\sigma^{-1}\tau', A'}"']
			&\prod_{\tau} H'({}^{\tau^{-1}}A')
			\ar[d, "s_\sigma"]\ar[r, "\phi^{-1}"]
			&H'(A'\otimes_k k')\\
			&\prod_\tau H'({}^{\tau^{-1}\sigma^{-1}} A')
			\ar[r, "{}^\sigma\phi^{-1}"]
			&H'({}^{\sigma^{-1}}A'\otimes_k k')\ar[u, equal]
			\ar[r, "\Lambda_{{}^{\sigma^{-1}} A'}"]
			&\bG_m(A').
		\end{tikzcd}\]
	\end{proof}
	\subsection{Standard models of classical Lie groups}
	We list standard $\bZ\left[1/2\right]$-forms of classical connected Lie groups and maximal compact subgroups. We follow \cite{kobayashioshima} section 7.1 to avoid using quaternions in Example \ref{ex:modelofSp(p,q)}. Let $n$ be a positive integer, and $p,q$ be nonnegative integers such that $p+q=n$.
	\begin{ex}[$\GL_n(\bR)^+$]
		We have symmetric pairs
		\[\begin{array}{ll}
			(\GL_n,\Oo(n))&(\SL_n,\SO(n))
		\end{array}\]
		over $\bZ\left[1/2\right]$ by definitions. Since $\Oo(n)^\circ=\SO(n)$, $(\GL_n,\SO(n))$ is also a symmetric pair.
	\end{ex}
	\begin{ex}[$\Uu(p,q)$]\label{ex:modelofupq}
		Define an involution $\tau_{p,q}^c$ on $\Res_{\bZ\left[1/2,\sqrt{-1}\right]/\bZ\left[1/2\right]}\GL_{p+q}$ by $g\mapsto I_{p,q}(g^\ast)^{-1}I_{p,q}$. Set
		\[\Uu(p,q)=(\Res_{\bZ\left[1/2,\sqrt{-1}\right]/\bZ\left[1/2\right]}\GL_{p+q})^{\tau_{p,q}^c}.\]
		That is,
		\[\Uu(p,q,A)=\{g\in\GL_{p+q}(A\otimes_{\bZ}\bZ\left[1/2\right]):\ g^\ast I_{p,q}g=I_{p,q}\},\]
		where $A$ runs through all commutative $\bZ\left[1/2\right]$-algebras. Then we have \[\Uu(p,q)\otimes_{\bZ\left[1/2\right]}\bZ\left[1/2,\sqrt{-1}\right]\cong\GL_{p+q};\]
		\[a\otimes 1+b\otimes\sqrt{-1}\mapsto a+\sqrt{-1}b\]
		\[\frac{g+I_{p,q}(g^T)^{-1}I_{p,q}}{2}\otimes 1+\frac{g-I_{p,q}(g^T)^{-1}I_{p,q}}{2\sqrt{-1}}\otimes \sqrt{-1}\leftmapsto g.\]
		In particular, $\Uu(p,q)$ is a reductive group scheme over $\bZ\left[1/2\right]$. Set $\Uu(n)=\Uu(n,0)$. The group scheme $\Uu(p,q)$ is closed under formation of $(-)^\ast$. Moreover, the block diagonal embedding $K=\Uu(p)\times\Uu(q)\hookrightarrow\Uu(p,q)$ identifies $K$ with the symmetric subgroup $\Uu(p,q)\cap\Uu(n)$. 
	\end{ex}
	\begin{ex}[$\GL_n(\bC)$]
		We have a symmetric pair \[(\Res_{\bZ\left[1/2,\sqrt{-1}\right]/\bZ\left[1/2\right]}\GL_n,\Uu(n))\]
		by definitions.
	\end{ex}
	For an affine group scheme $G$ over $\bZ\left[1/2\right]$, define an involution $\bar{\ }$ on
	\[\Res_{\bZ\left[1/2,\sqrt{-1}\right]/\bZ\left[1/2\right]} G\]
	by the conjugation of $\bZ\left[1/2,\sqrt{-1}\right]$.
	\begin{ex}[$\Uu^\ast(2n)$]
		Define a symmetric subgroup \[\Uu^\ast(2n)\subset\Res_{\bZ\left[1/2,\sqrt{-1}\right]/\bZ\left[1/2\right]}\GL_{2n}\]
		by the involution $g\mapsto -J_n\bar{g}J_n$. That is,
		\[\Uu^\ast(2n,A)=\{g\in(\Res_{\bZ\left[1/2,\sqrt{-1}\right]/\bZ\left[1/2\right]}\GL_{2n})(A):\ \bar{g}J_n=J_ng\}\]
		as a copresheaf on $\CAlg_{\bZ\left[1/2\right]}$. We can also define \[\SU^\ast(2n)=\Uu^\ast(2n)\cap\Res_{\bZ\left[1/2,\sqrt{-1}\right]/\bZ\left[1/2\right]}\SL_{2n}.\]
		Notice that we have an isomorphism 
		\[\Uu^\ast(2n)\otimes_{\bZ\left[1/2\right]} \bZ\left[1/2,\sqrt{-1}\right]\cong\GL_{2n};\]
		\[a\otimes 1+b\otimes\sqrt{-1}\mapsto a+\sqrt{-1}b\]
		\[\frac{g-J_ngJ_n}{2}\otimes 1+\frac{g+J_ngJ_n}{2\sqrt{-1}}\otimes\sqrt{-1} \leftmapsto g\]
		over $\bZ\left[1/2,\sqrt{-1}\right]$. It also restricts to $\SU^\ast(2n)\otimes_{\bZ\left[1/2\right]} \bZ\left[1/2,\sqrt{-1}\right]\cong\SL_{2n}$. In particular, $\Uu^\ast(2n)$ (resp.\ $\SU^\ast(2n)$) is a reductive group scheme (resp.\ a semisimple group scheme whose geometric fibers are simply connected). Notice also that $\Uu^\ast(2n)$ and $\SU^\ast(2n)$ are closed under formation of $(-)^\ast$. Let $\Sp(n)=\Uu^\ast(2n)\cap\Uu(2n)$. It is easy to show that
		\[\Sp(n,A)=\{g\in\Uu(2n,A):~g^T J_ng=J_n\}.\]
		In particular, we have
		\[\SU^\ast(2n)\cap \Uu(2n)=\Sp(n)\]
		by $\Sp(n)\subset \Res_{\bZ\left[1/2,\sqrt{-1}\right]/\bZ\left[1/2\right]}
		\Sp_n\subset \Res_{\bZ\left[1/2,\sqrt{-1}\right]/\bZ\left[1/2\right]}
		\SL_{2n}$.
	\end{ex}
	\begin{ex}[$\SO_0(p,q)$]\label{ex:modelofso(p,q)}
		Define an involution $\tau_{p,q}$ on $\SL_n$ by \[\tau_{p,q}(g)=I_{p,q}(g^T)^{-1}I_{p,q}.\]
		Set $\SO(p,q)=(\SL_n)^{\tau_{p,q}}$. That is,
		\[\SO(p,q,A)=\{g\in\SL_{p+q}(A):\ g^T I_{p,q}g=I_{p,q}\}.\]
		Write
		\[g_{p,q}=\left\{\begin{array}{ll}
			\diag(\overbrace{g_2,\cdots,g_2}^{\frac{p}{2}},\overbrace{\sqrt{-1}g_2,\cdots,\sqrt{-1}g_2}^{\frac{q}{2}})
			& (p,q\ \mathrm{even});\\
			\diag(\overbrace{g_2,\cdots,g_2}^{\frac{p}{2}},\overbrace{\sqrt{-1}g_2,\cdots,\sqrt{-1}g_2}^{\frac{q-1}{2}},\sqrt{-1})& (p\ \mathrm{even},\ q\ \mathrm{odd});\\
			\diag(1,\overbrace{g_2,\cdots,g_2}^{\frac{p-1}{2}},\overbrace{\sqrt{-1}g_2,\cdots,\sqrt{-1}g_2}^{\frac{q}{2}})& (p\ \mathrm{odd},\ q\ \mathrm{even});\\
			\diag\left(\overbrace{g_2,\cdots,g_2}^{\frac{p-1}{2}},\left(\begin{array}{cc}
				1& 1 \\ 
				\sqrt{-1} & -\sqrt{-1}
			\end{array} \right),\overbrace{\sqrt{-1}g_2,\cdots,\sqrt{-1}g_2}^{\frac{q-1}{2}}\right)
			& (p,q\ \mathrm{odd}).
		\end{array} \right.\]
		Then $g\mapsto g_{p,q}gg_{p,q}^{-1}$ defines an isomorphism
		\[\SO(p,q)\otimes_{\bZ\left[1/2\right]}
		\bZ\left[1/2,\sqrt{-1}\right]\cong\left\{\begin{array}{ll}
			\SO_n & (n\ \mathrm{even});\\
			\SO_n & (p\ \mathrm{even},\ q\ \mathrm{odd});\\
			\SO_n'& (p\ \mathrm{odd},\ q\ \mathrm{even}).
		\end{array} \right.\]
		Write $\mathrm{S}(\Oo(p)\times\Oo(q))=\SO(p,q)\cap \SO(n)$. Then the block diagonal embedding $K=\SO(p)\times\SO(q)\hookrightarrow \SO(p,q)$ identifies $K$ with $\mathrm{S}(\Oo(p)\times\Oo(q))^\circ$.
	\end{ex}
	\begin{ex}[$\SO^\ast(2n)$]\label{ex:modelofso^ast(2n)}
		Notice that $\SU^\ast(2n)$ is closed under the transpose $(-)^T$. Set
		\[\SO^\ast(2n)=\SU^\ast(2n)\cap\Res_{\bZ\left[1/2,\sqrt{-1}\right]/\bZ\left[1/2\right]}\Oo(2n).\]
		We have an isomorphism 
		\[\SO^\ast(2n)\otimes_{\bZ\left[1/2\right]} \bZ\left[1/2,\sqrt{-1}\right]\cong\SO(2n);\]
		\[a\otimes 1+b\otimes\sqrt{-1}\mapsto a+\sqrt{-1}b\]
		\[\frac{g-J_ngJ_n}{2}\otimes 1+\frac{g+J_ngJ_n}{2\sqrt{-1}}\otimes\sqrt{-1} \leftmapsto g\]
		over $\bZ\left[1/2,\sqrt{-1}\right]$. In particular, $\SO^\ast(2n)$ is a reductive group scheme\footnote{The group scheme $\SO^\ast(2n)$ is semisimple if $n\geq 2$.}. It is clear that $\SO^\ast(2n)$ is closed under formation of $(-)^\ast$. The morphism
		\[a\otimes 1+b\otimes\sqrt{-1}
		\mapsto \left(\begin{array}{cc}
			a & -b \\ 
			b & a
		\end{array} \right)\]
		determines an isomorphism $\Uu(n)\cong\SO^\ast(2n)\cap\Uu(2n)$ over $\bZ\left[1/2\right]$.
	\end{ex}
	\begin{ex}[$\SO(n,\bC)$]
		The involution $\bar{\ }$ determines a symmetric pair \[(\Res_{\bZ\left[1/2,\sqrt{-1}\right]/\bZ\left[1/2\right]}\SO(n),\SO(n)).\]
	\end{ex}
	\begin{ex}[$\Sp_n$]
		The homomorphism over $\bZ\left[1/2\right]$ \[\Uu(n)\hookrightarrow\Sp_n;\]
		\[a\otimes 1+b\otimes\sqrt{-1}\mapsto\left(\begin{array}{cc}
			a & -b \\ 
			b & a
		\end{array} \right)\]
		determines an isomorphism $K=\Uu(n)\cong\Sp_n\cap\Oo(2n)$. Since $\Sp_n$ is closed under formation of the transpose, $(\Sp_n,\Sp_n\cap\Oo(2n))$ is a symmetric pair.
	\end{ex}
	\begin{ex}[$\Sp(p,q)$]\label{ex:modelofSp(p,q)}
		It is clear that $\SU^\ast(2n)$ is closed under formation of the involution $\tau_{p,q,p,q}^c:g\mapsto I_{p,q,p,q}(g^\ast)^{-1}I_{p,q,p,q}$. Define $\Sp(p,q)=\SU^\ast(2n)^{\tau_{p,q,p,q}^c}$. That is
		\[\Sp(p,q,A)=\left\{g=\left(\begin{array}{cc}
			a & -b \\ 
			\bar{b} & \bar{a}
		\end{array} \right)\in\SU^\ast(2n,A):\ g^\ast I_{p,q,p,q}g=I_{p,q,p,q}\right\}.\]
		For example, $\Sp(n)=\Sp(n,0)$. This definition is closely related to quaternions (\cite{kobayashioshima} Proposition 7.9 (ii), \cite{MR1920389} section I.8). Another realization is given by
		\[\Sp(p,q,A)'=\{g\in(\Res_{\bZ\left[1/2,\sqrt{-1}\right]/\bZ\left[1/2\right]}\Sp_n)(A):\ g^\ast I_{p,q,p,q}g=I_{p,q,p,q}\}.\]
		The conjugation by $I_{n+p,q}$ gives an isomorphism $\Sp(p,q)\cong\Sp(p,q)'$. The advantage of $\Sp(p,q)'$ is that the relation with $\Sp_n$ is clear. In fact, we have an isomorphism
		\[\Sp(p,q)'\otimes_{\bZ\left[1/2\right]} \bZ\left[1/2,\sqrt{-1}\right]\cong\Sp_n;\]
		\[a\otimes 1+b\otimes\sqrt{-1}\mapsto a+\sqrt{-1}b\]
		\[\frac{g+I_{p,q,p,q}(g^T)^{-1}I_{p,q,p,q}}{2}\otimes 1+\frac{g-I_{p,q,p,q}(g^T)^{-1}I_{p,q,p,q}}{2\sqrt{-1}}\otimes \sqrt{-1}\leftmapsto g.\]
		Notice also that $\Sp(p,q)$ is closed under formation of $(-)^\ast$. The morphism
		\[\left(\left(\begin{array}{cc}
			a_{11} & a_{12} \\ 
			a_{21} & a_{22}
		\end{array} \right),\left(\begin{array}{cc}
			b_{11} & b_{12} \\ 
			b_{21} & b_{22}
		\end{array}\right)\right)\mapsto \left(\begin{array}{cccc}
			a_{11} & 0 & a_{12} & 0 \\ 
			0 & b_{11} & 0 & b_{12} \\ 
			a_{21} & 0 & a_{22} & 0 \\ 
			0 & b_{21} & 0 & b_{22}
		\end{array} \right)\]
		determines an isomorphism $\Sp(p)\times\Sp(q)\cong\Sp(p,q)\cap\Uu(2n)$.
	\end{ex}
	\begin{ex}[$\Sp_n(\bC)$]
		It is clear that $\Res_{\bZ\left[1/2,\sqrt{-1}\right]/\bZ\left[1/2\right]}\Sp_n$ is closed under formation of $(-)^\ast$. Hence $(\Res_{\bZ\left[1/2,\sqrt{-1}\right]/\bZ\left[1/2\right]}\Sp_n,\Sp(n))$ is a symmetric pair.
	\end{ex}
	\subsection{Standard models of Fundamental Cartan subgroups}
	In the former section, we introduced models $(G,K)$ of symmetric pairs of classical connected Lie groups and their maximal compact subgroups. In this section, we give maximal tori $T$ of $K$ for models of non-complex groups to get fundamental Cartan subgroups of $G$ which we will denote by $H=Z_G(T)$; We will deal with models of complex groups in a more general context in Section 4.2. It easily follows from the arguments on models of compact tori of rank 1 in Section 3.2 that the conjugation action on $X^\ast(T\otimes_{\bZ\left[1/2\right]} \bZ\left[1/2,\sqrt{-1}\right])$ for the tori $T$ in this section is equal to $-1$ (cf.\ Proposition \ref{prop:conj=-1}). We also check that $(G\otimes_{\bZ\left[1/2\right]} \bZ\left[1/2,\sqrt{-1}\right],H\otimes_{\bZ\left[1/2\right]} \bZ\left[1/2,\sqrt{-1}\right])$ is split by constructing isomorphisms with familiar split reductive group schemes $\GL_n$, $\SO_{2n+1}$, $\Sp_n$, $\SO_{2n}$, and $\SO'_{2n}$ with the split maximal tori of diagonal matrices in order to use the techniques developed in Section 2.1, 2.3, and 3.1. To find an isomorphism of $H\otimes_{\bZ\left[1/2\right]} \bZ\left[1/2,\sqrt{-1}\right]$ with a split torus of diagonal matrices will lead to an easy description of $X^\ast(H\otimes_{\bZ\left[1/2\right]} \bZ\left[1/2,\sqrt{-1}\right])$ and $X_\ast(H\otimes_{\bZ\left[1/2\right]} \bZ\left[1/2,\sqrt{-1}\right])$ which appears in Section 4.1. We will denote the split maximal torus of diagonal matrices in each of $\GL_n$, $\SO_{2n+1}$, $\Sp_n$, $\SO_{2n}$, and $\SO'_{2n}$ by $H_s$.
	
	Let $p,q,n$ be nonnegative integers such that $p+q=n$. We will possibly assume more conditions in some cases.
	
	\begin{ex}
		Put $G=\GL_{2n+1}$. Let $T\subset\GL_{2n+1}$ be the subgroup scheme \[\diag(\overbrace{\SO(2),\SO(2),\ldots,\SO(2)}^{n},1).\]
		The conjugation of $\GL_{2n}$ over $\bZ\left[1/2,\sqrt{-1}\right]$ by $g_{2n+1}$ restricts to the isomorphism from $H\otimes_{\bZ\left[1/2\right]} \bZ\left[1/2,\sqrt{-1}\right]$ onto $H_s\subset\GL_{2n+1}$. In fact, notice that a maximal torus containing $T$ is $H$ by definition of fundamental Cartan subgroups. The assertion then follows since $g_{2n+1}Tg_{2n+1}^{-1}$ consists of diagonal matrices.
	\end{ex}
	\begin{ex}
		Put $G=\GL_{2n}$ with $n\geq 1$. Let $T\subset\GL_{2n}$ be the subgroup scheme
		\[\diag(\overbrace{\SO(2),\SO(2),\ldots,\SO(2)}^{n}).\]
		Then the conjugation of $\GL_{2n}$ over $\bZ\left[1/2,\sqrt{-1}\right]$ by $g_{2n}$ restricts to the isomorphism from $H\otimes_{\bZ\left[1/2\right]} \bZ\left[1/2,\sqrt{-1}\right]$ to $H_s\subset\GL_{2n}$.
	\end{ex}
	\begin{ex}
		Put $G=\Uu(p,q)$ with $p+q=n\geq 1$ and $p\geq q$. Let $H= T\subset\Uu(p,q)$ be the subgroup scheme $\Uu(1)^n$ of diagonal matrices in $\Uu(p,q)$. The isomorphism
		\[\Uu(p,q)\otimes_{\bZ\left[1/2\right]}\bZ\left[1/2,\sqrt{-1}\right]\cong\GL_{p+q}\]
		in Example \ref{ex:modelofupq} sends $H\otimes_{\bZ\left[1/2\right]}\bZ\left[1/2,\sqrt{-1}\right]$ onto $H_s\subset\GL_n$.
	\end{ex}
	\begin{ex}
		Put $G=\Uu^\ast(2n)$ with $n\geq 1$. Diagonal matrices form a maximal torus $T$ of $\Sp(n)$. Namely, we have
		\[T=\{\diag(a_1,a_2,\ldots,a_{2n})\in\Uu(2n):
		~a_{i+n}=a_{i}^{-1}\ {\rm for\ } 1\leq i\leq n\}
		\cong\Uu(1)^n.\]
		The fundamental Cartan subgroup of $\Uu^\ast(2n)$ attached to $T\subset\Sp(n)$ consists of the diagonal matrices, i.e.,
		\begin{flalign*}
			&H\\
			&=\{\diag(a_1,a_2,\ldots,a_{2n})\in\Res_{\bZ\left[1/2,\sqrt{-1}\right]/\bZ\left[1/2\right]}\GL_{2n}:
			~a_{i+n}=a_{i}^{-1}\ {\rm for\ } 1\leq i\leq n\}\\
			&\cong 
			(\Res_{\bZ\left[1/2,\sqrt{-1}\right]/\bZ\left[1/2\right]}\bG_m)^n.
		\end{flalign*}
		The canonical isomorphism $\Uu^\ast(2n)\otimes_{\bZ\left[1/2\right]}\bZ\left[1/2,\sqrt{-1}\right]\cong\GL_{2n}$ restricts to $H\otimes_{\bZ\left[1/2\right]}\bZ\left[1/2,\sqrt{-1}\right]\cong H_s\subset\GL_{2n}$;
		\begin{flalign*}
			&\left(\begin{array}{cc}
				\diag(a_i\otimes 1+b_i\otimes\sqrt{-1}) &  \\
				& \diag(a_i\otimes 1-b_i\otimes\sqrt{-1})
			\end{array}
			\right)\\
			&\mapsto\left(\begin{array}{cc}
				\diag(a_i+b_i\sqrt{-1}) &  \\
				& \diag(a_i-b_i\sqrt{-1})
			\end{array}\right)\\
			&\left(\begin{array}{cc}
				\diag(\frac{a_i+a_{i+n}}{2}\otimes 1+\frac{a_i-a_{i+n}}{2\sqrt{-1}}\otimes\sqrt{-1}) &  \\
				& \diag(\frac{a_i+a_{i+n}}{2}\otimes 1-\frac{a_i-a_{i+n}}{2\sqrt{-1}}\otimes\sqrt{-1})
			\end{array}\right)\\
			&\reflectbox{$\mapsto$}\diag(a_i).
		\end{flalign*}
		See also Proposition \ref{prop:galactionofchargrprestorus}.
	\end{ex}
	\begin{ex}
		Put $G=\SO(2p,2q+1)$.
		Let $H= T$ be the subgroup scheme 
		\[\diag(\overbrace{\SO(2),\SO(2),\ldots,\SO(2)}^{p+q},1)
		\subset \SO(2p,2q+1).\]
		The isomorphism
		\[\SO(2p,2q+1)\otimes_{\bZ\left[1/2\right]}\bZ\left[1/2,\sqrt{-1}\right]\cong\SO_{2n+1}\]
		in Example \ref{ex:modelofso(p,q)} restricts to $H\otimes_{\bZ\left[1/2\right]}\bZ\left[1/2,\sqrt{-1}\right] \cong H_s\subset\SO_{2n+1}$.
	\end{ex}
	\begin{ex}
		Put $G=\Sp_n$ with $n\geq 1$.
		A fundamental Cartan subgroup $H=T$ of $\Sp_n$ is given by
		\begin{flalign*}
			&\left\{\left(\begin{array}{cccccc}
				a_1 &  &  & -b_1 &  &  \\
				& \ddots &  &  & \ddots &  \\
				&  & a_n &  &  & -b_n \\
				b_1 &  &  & a_1 &  &  \\
				& \ddots &  &  & \ddots &  \\
				&  & b_n &  &  & a_n
			\end{array}
			\right)\in\GL_{2n}:~a_i^2+b_i^2=1{\rm \ for\ }1\leq i\leq n
			\right\}\\
			&\cong\Uu(1)^n.
		\end{flalign*}
		The conjugation by $g'_{2n}$ determines an automorphism of $\Sp_n$ over $\bZ\left[1/2,\sqrt{-1}\right]$. Moreover, it restricts to the isomorphism from $H\otimes_{\bZ\left[1/2\right]} \bZ\left[1/2,\sqrt{-1}\right]$ onto $H_s\subset \Sp_n$.
	\end{ex}
	\begin{ex}
		Put $G=\Sp(p,q)$ with $p+q=n\geq 1$.
		A fundamental Cartan subgroup $H=T$ of $\Sp(p,q)$ is given by
		\[H=\{\diag(a_1,a_2,\ldots,a_{2n})\in\Uu(2n):
		~a_{i+n}=a_{i}^{-1}\ {\rm for\ } 1\leq i\leq n\}
		\cong\Uu(1)^n.\]
		The composite isomorphism \[\Sp(p,q)\otimes_{\bZ\left[1/2\right]}
		\bZ\left[1/2,\sqrt{-1}\right]\cong
		\Sp(p,q)'\otimes_{\bZ\left[1/2\right]}
		\bZ\left[1/2,\sqrt{-1}\right]
		\cong\Sp_n\]
		in Example \ref{ex:modelofSp(p,q)} restricts to $H\otimes_{\bZ\left[1/2\right]} \bZ\left[1/2,\sqrt{-1}\right]
		\cong H_s\subset\Sp_n$.
	\end{ex}
	\begin{ex}
		Put $G=\SO(2p,2q)$ with $n=p+q\geq 1$.
		Let $H= T$ be the subgroup scheme $\SO(2)^{p+q}$ of block diagonal matrices. The isomorphism
		\[\SO(2p,2q)\otimes_{\bZ\left[1/2\right]}\bZ\left[1/2,\sqrt{-1}\right]\cong\SO_{2n}\]
		in Example \ref{ex:modelofso(p,q)} restricts to
		$H\otimes_{\bZ\left[1/2\right]}\bZ\left[1/2,\sqrt{-1}\right]
		\cong H_s\subset \SO_{2n}$.
	\end{ex}
	\begin{ex}
		Put $G=\SO(2p+1,2q+1)$.
		A maximal torus $T$ of 
		\[K=\diag(\SO(2p),1,1,\SO(2q))\]
		is given by the subgroup scheme \[\diag(\overbrace{\SO(2),\SO(2),\ldots,\SO(2)}^{p},1,1,
		\overbrace{\SO(2),\SO(2),\ldots,\SO(2)}^{q}).\]
		The fundamental Cartan subgroup $H$ attached to $T$ is given by
		\[\diag(\overbrace{\SO(2),\SO(2),\ldots,\SO(2)}^{p},\SO(1,1),
		\overbrace{\SO(2),\SO(2),\ldots,\SO(2)}^{q}).\]
		Notice that the isomorphism
		\[\SO(2p+1,2q+1)\otimes_{\bZ\left[1/2\right]}
		\bZ\left[1/2,\sqrt{-1}\right]\cong \SO_{2n+2}\]
		in Example \ref{ex:modelofso(p,q)} sends $H\otimes_{\bZ\left[1/2\right]}\bZ\left[1/2,\sqrt{-1}\right]$ onto $H_s\subset\SO_{2n+2}$.
	\end{ex}
	\begin{ex}
		Put $G=\SO^\ast(2n)$ with $n\geq 1$.
		A fundamental Cartan subgroup $H=T$ of $\SO^\ast(2n)$ is given by
		\[\begin{split}
			H&=T\\
			&=\left\{\left(\begin{array}{cccccc}
				a_1 &  &  & -b_1 &  &  \\
				& \ddots &  &  & \ddots &  \\
				&  & a_n &  &  & -b_n \\
				b_1 &  &  & a_1 &  &  \\
				& \ddots &  &  & \ddots &  \\
				&  & b_n &  &  & a_n
			\end{array}
			\right)\in\GL_{2n}:~a_i^2+b_i^2=1{\rm \ for\ }1\leq i\leq n
			\right\}.
		\end{split}\]
		The conjugation by $g'_{2n}$ determines an isomorphism
		\[\SO(2n)\otimes_{\bZ\left[1/2\right]}
		\bZ\left[1/2,\sqrt{-1}\right]\cong\SO'_{2n}.\]
		Compose it with the isomorphism
		\[\SO^\ast(2n)\otimes_{\bZ\left[1/2\right]}
		\bZ\left[1/2,\sqrt{-1}\right]\cong \SO(2n)\otimes_{\bZ\left[1/2\right]}
		\bZ\left[1/2,\sqrt{-1}\right]\]
		in Example \ref{ex:modelofso^ast(2n)} to get
		\[\SO^\ast(2n)\otimes_{\bZ\left[1/2\right]}
		\bZ\left[1/2,\sqrt{-1}\right]\cong\SO_{2n}'.\]
		It restricts to the isomorphism from $H\otimes_{\bZ\left[1/2\right]}
		\bZ\left[1/2,\sqrt{-1}\right]$ onto $H_s\subset\SO_{2n}$.
	\end{ex}
	\section{Classification}
	This section is devoted to completing our program in Section 2 and 3.1 for the standard models $(G,T,H)$ introduced in Section 3. I.e., we have the following four aims:
	\begin{enumerate}
		\item Compute $\Dyn G$ and $\type G$ by choosing a simple system $\Pi$.
		\item Determine the character group of the parabolic subgroups over $\bZ\left[1/2,\sqrt{-1}\right]$ attached to a subset of $\Pi$.
		\item Classify equivariant line bundles on partial flag schemes over $\bZ\left[1/2\right]$.
	\end{enumerate}
	\subsection{Non-complex case}
	In this section, we run the above program for the standard models of non-complex classical connected Lie groups. We achieve it in the following way:
	\begin{enumerate}
		\item Take the standard bases $\{\epsilon_i\}$, $\{e_i\}$ of $X_\ast(H_s)$ and $X^\ast(H_s)$ for the split maximal torus $H_s$ in Section 3.4.
		\item Identify $X_\ast(H\otimes_{\bZ\left[1/2\right]} \bZ\left[1/2,\sqrt{-1}\right])$ and $X^\ast(H\otimes_{\bZ\left[1/2\right]} \bZ\left[1/2,\sqrt{-1}\right])$ with $\bZ^n$ for some $n$ by using Stage 1 and the isomorphism
		\[H\otimes_{\bZ\left[1/2\right]} \bZ\left[1/2,\sqrt{-1}\right]\cong H_s\]
		in Section 3.4.
		\item Describe the image of $X_\ast(T\otimes_{\bZ\left[1/2\right]} \bZ\left[1/2,\sqrt{-1}\right])$ in \[X_\ast(H\otimes_{\bZ\left[1/2\right]} \bZ\left[1/2,\sqrt{-1}\right])\]
		under the identification of Stage 2.
		\item Choose a free basis of $X_\ast(T\otimes_{\bZ\left[1/2\right]} \bZ\left[1/2,\sqrt{-1}\right])$ under the expression of Stage 3.
		\item Describe the positive system $\Delta^+$ and the simple system $\Pi$ attached to Stage 4 (cf.\ Construction \ref{cons:simplesys}).
		\item For $\lambda\in X^\ast(H\otimes_{\bZ\left[1/2\right]} \bZ\left[1/2,\sqrt{-1}\right])$ and $\alpha\in\Pi$, compute $\langle\alpha^\vee,\lambda\rangle$.
		\item Find $w\in N_{G(\bZ\left[1/2,\sqrt{-1}\right])}(H\otimes_{\bZ\left[1/2\right]} \bZ\left[1/2,\sqrt{-1}\right])$ such that $w\Pi=-\Pi$ and $(\bar{w}w)^2=e$ (cf.\ Section 2.1, Proposition \ref{prop:conj=-1}, the beginning of Section 3.4, Theorem \ref{thmF}).
		\item Describe the Dynkin scheme by computing the $\ast$-involution (see Example \ref{ex:dynkin}).
		\item For $\lambda\in X^\ast(H\otimes_{\bZ\left[1/2\right]} \bZ\left[1/2,\sqrt{-1}\right])$, see whether $\bar{\lambda}=w\lambda$.
		\item Determine which $\lambda\in X^\ast(H\otimes_{\bZ\left[1/2\right]} \bZ\left[1/2,\sqrt{-1}\right])$ satisfies $\lambda(\bar{w}w)=1$.
	\end{enumerate}
	In fact, we can determine the character group of the parabolic subgroup $P_x'$ attached to a subset $x\subset\Pi$ by Stage 6 (see Section 2.3.2). In view of Theorem \ref{thm:dynkin} (2) or the definition of $\type G$, we can determine when $P'_x$ has a half-integral parabolic type by Stage 8. In particular, we deduce the classification of parabolic partial flag schemes over $\bZ\left[1/2\right]$. Recall that $\bar{\Pi}=-\Pi$ in our setting (see the beginning of Section 3.4). We can determine when $\lambda$ extends to a character of $P_x'$ by Stage 6 (Example \ref{ex:charofparabolics}). For a character $\lambda$ in Stage 10, Stage 7 implies $\lambda(\bar{w}w)\in\{\pm 1\}$. Moreover, we determine the characters $\lambda$ satisfying $\bar{\lambda}=w\lambda$ and $\lambda(\bar{w}w)=1$. The equivariant line bundles with $\bZ\left[1/2\right]$-forms are $\cL_\lambda$ with such $\lambda$ by Remark \ref{rem:associatedbundle} and Corollary \ref{cor:quadcase}. The $\bZ\left[1/2\right]$-form is unique by Theorem \ref{thm:uniqueness} and Example \ref{ex:vanishing}. As a conclusion, these form the complete list of equivariant line bundles on partial flag schemes over $\bZ\left[1/2\right]$ (see Corollary \ref{corG} for the unifying statement).
	
	To avoid the repeated arguments, let us introduce the bases in 1 here before we start the case-by-case discussions for the standard models $G$. We temporally take a positive integer $n$. Define $\epsilon_i\in X_\ast(H_s)$ and $e_i\in X^\ast(H_s)$ for $1\leq i\leq n$ in each of the cases $H_s\subset\GL_n,\SO_{2n+1},\Sp_n,\SO_{2n},\SO'_{2n}$ as follows:
	\begin{description}
		\item[$\GL_n$]
		\[\epsilon_i(a)=\diag(1,1,\ldots,1,\overset{\overset{i}{\vee}}{a},1,1,\ldots,1)\]
		\[e_i(\diag(a_1,a_2,\ldots,a_n))=a_i\]
		\item[$\SO_{2n+1}$] 
		\[\epsilon_i(a)=\diag(1,1,\ldots,1,\overset{\overset{2i-1}{\vee}}{a},a^{-1},1,1,\ldots,1)\]
		\[e_i(\diag(a_1,a_2,\ldots,a_{2n},1))
		=a_{2i-1}\]
		\item[$\Sp_n$] 
		\[\epsilon_i(a)=\diag(1,1,\ldots,1,\overset{\overset{i}{\vee}}{a},1,1,\ldots,1,1,\overset{\overset{i+n}{\vee}}{a^{-1}},1,1,\ldots,1)\]
		\[e_i(\diag(a_1,a_2,\ldots,a_{2n}))
		=a_i\]
		\item[$\SO_{2n}$] 
		\[\epsilon_i(a)=\diag(1,1,\ldots,1,\overset{\overset{2i-1}{\vee}}{a},a^{-1},1,1,\ldots,1)\]
		\[e_i(\diag(a_1,a_2,\ldots,a_{2n}))
		=a_{2i-1}\]
		\item[$\SO'_{2n}$]
		\[\epsilon_i(a)=\diag(1,1,\ldots,1,\overset{\overset{i}{\vee}}{a},1,1,\ldots,1,1,\overset{\overset{i+n}{\vee}}{a^{-1}},1,1,\ldots,1)\]
		\[e_i(\diag(a_1,a_2,\ldots,a_{2n}))=a_i.\]
	\end{description}
	We remark that we have the empty bases for $\SO_1$ since $\SO_1$ is the trivial group.
	
	We now start the case-by-case study. Let $p,q,n$ be nonnegative integers such that $p+q=n$. We will possibly assume more conditions on these integers in some cases. Let $T$ and $H$ be as in Section 3.4.
	\begin{ex}[$\GL_{2n+1}$]
		The image of $X_\ast(T\otimes_{\bZ\left[1/2\right]} \bZ\left[1/2,\sqrt{-1}\right])$ in \[X_\ast(H\otimes_{\bZ\left[1/2\right]} \bZ\left[1/2,\sqrt{-1}\right])\cong\bZ^{2n+1}\]
		is
		\[\{\mu=(\mu_i)\in \bZ^{2n+1}:~\mu_{2i}
		=-\mu_{2i-1}\ \mathrm{for}\ 1\leq i\leq n,\ \mu_{2n+1}=0\}.\]
		The ordered basis $\{\epsilon_1-\epsilon_2,\epsilon_3-\epsilon_4,\ldots,\epsilon_{2n-1}-\epsilon_{2n}\}$ of
		\[X_\ast(T\otimes_{\bZ\left[1/2\right]} \bZ\left[1/2,\sqrt{-1}\right])\]
		attaches the following positive and simple systems:
		\[\begin{split}
			\Delta^+
			&=\{e_i-e_j:~1\leq i<j\leq 2n+1,~i\mathrm{\ is\ odd}\}\\
			&\cup
			\{-e_i+e_j:~1\leq i<j\leq 2n+1,~i\mathrm{\ is\ even}\}
		\end{split}\]
		\[\begin{split}
			\Pi&=\{e_{2i-1}-e_{2i+1}:~1\leq i\leq n\}\\
			&\cup\{-e_{2i}+e_{2i+2}:~1\leq i\leq n-1\}\cup
			\{-e_{2n}+e_{2n+1}\}.
		\end{split}\]
		For $1 \leq i \leq n$ (resp.~$1\leq i\leq n-1$),
		$\lambda= (\lambda_i) \in X^\ast(H_s)$ annihilates $(e_{2i-1} - e_{2i+1})^\vee=\epsilon_{2i-1}-\epsilon_{2i+1}$ (resp.~$(-e_{2i}+e_{2i+2})^\vee=-\epsilon_{2i}+\epsilon_{2i+2}$) if and only if $\lambda_{2i-1} = \lambda_{2i+1}$ (resp.~$\lambda_{2i} = \lambda_{2i+2}$). A character $\lambda$ annihilates $(-e_{2n} + e_{2n+1})^\vee=-\epsilon_{2n}+\epsilon_{2n+1}$ if and only if $\lambda_{2n-2} =\lambda_{2n-1}$.
		The matrix $S_{2n+1}$ represents the longest element of the Weyl group. Observe that for elements $a,b$ of a $\bZ\left[1/2,\sqrt{-1}\right]$-algebra, we have
		\[g_2\diag(a,b)g_2^{-1}=\left(\begin{array}{cc}
			\frac{a+b}{2} & -\frac{a-b}{2\sqrt{-1}} \\
			\frac{a-b}{2\sqrt{-1}} & \frac{a+b}{2}
		\end{array}
		\right)\]
		\[\bar{g}_2\diag(a,b)\bar{g}_2^{-1}=\left(\begin{array}{cc}
			\frac{a+b}{2} & \frac{a-b}{2\sqrt{-1}} \\
			-\frac{a-b}{2\sqrt{-1}} & \frac{a+b}{2}
		\end{array}
		\right)
		=g_2\diag(b,a)g_2^{-1}.\]
		Hence for a character $\lambda=(\lambda_i)\in X^\ast(H\otimes_{\bZ\left[1/2\right]} \bZ\left[1/2,\sqrt{-1}\right])$, we have
		\[\bar{\lambda}=w\lambda=
		(\lambda_2,\lambda_1,\lambda_4,\lambda_3,\ldots,
		\lambda_{2n},\lambda_{2n-1},\lambda_{2n+1}).\]
		In particular, we obtain $\Dyn\GL_{2n+1}\cong\Spec\bZ\left[1/2\right]^{2n}$. For every character $\lambda\in X^\ast(H\otimes_{\bZ\left[1/2\right]} \bZ\left[1/2,\sqrt{-1}\right])$, we have $\lambda(\bar{w}w)=1$ since $\bar{w}w=I_{2n+1}$.
	\end{ex}
	\begin{ex}[$\GL_{2n}$, $n\geq 1$]
		The image of $X_\ast(T\otimes_{\bZ\left[1/2\right]} \bZ\left[1/2,\sqrt{-1}\right])$ in $X_\ast(H\otimes_{\bZ\left[1/2\right]} \bZ\left[1/2,\sqrt{-1}\right])\cong\bZ^n$ is
		\[\{\mu=(\mu_i)\in \bZ^{2n}:~\mu_{2i}
		=-\mu_{2i-1}\ \mathrm{for}\ 1\leq i\leq n\}.\]
		The ordered basis $\{\epsilon_1-\epsilon_2,\epsilon_3-\epsilon_4,\ldots,\epsilon_{2n-1}-\epsilon_{2n}\}$ of $X_\ast(T\otimes_{\bZ\left[1/2\right]} \bZ\left[1/2,\sqrt{-1}\right])$ attaches the following positive and simple systems:
		\[\begin{split}
			\Delta^+
			&=\{e_i-e_j:~1\leq i<j\leq 2n,~i\mathrm{\ is\ odd}\}\\
			&\cup
			\{-e_i+e_j:~1\leq i<j\leq 2n,~i\mathrm{\ is\ even}\}
		\end{split}\]
		\[\Pi=\{e_{2i-1}-e_{2i+1}:~1\leq i\leq n-1\}
		\cup\{-e_{2i}+e_{2i+2}:~1\leq i\leq n-1\}\cup
		\{e_{2n-1}-e_{2n}\}.\]
		For $1 \leq i \leq n-1$,
		$\lambda= (\lambda_i) \in X^\ast(H_s)$ annihilates $(e_{2i-1} - e_{2i+1})^\vee=\epsilon_{2i-1}-\epsilon_{2i+1}$ (resp.~$(-e_{2i}+e_{2i+2})^\vee=-\epsilon_{2i}+\epsilon_{2i+2}$) if and only if $\lambda_{2i-1} = \lambda_{2i+1}$ (resp.~$\lambda_{2i} = \lambda_{2i+2}$). A character $\lambda$ annihilates $(e_{2n-1} - e_{2n})^\vee$ if and only if $\lambda_{2n-1} =\lambda_{2n}$.
		The matrix $S_{2n}$ represents the longest element of the Weyl group. For a character $\lambda=(\lambda_i)\in X^\ast(H\otimes_{\bZ\left[1/2\right]} \bZ\left[1/2,\sqrt{-1}\right])$, we have
		\[\bar{\lambda}=w\lambda=
		(\lambda_2,\lambda_1,\lambda_4,\lambda_3,\ldots,
		\lambda_{2n},\lambda_{2n-1}).\]
		In particular, we obtain $\Dyn\GL_{2n}\cong\Spec\bZ\left[1/2\right]^{2n-1}$. We have $\lambda(\bar{w}w)=1$ for every character $\lambda\in X^\ast(H\otimes_{\bZ\left[1/2\right]} \bZ\left[1/2,\sqrt{-1}\right])$ since $\bar{w}w=I_{2n}$.
	\end{ex}
	\begin{ex}[$\Uu(p,q)$, $p+q=n\geq 1$, $p\geq q$]
		Recall that $H=T$ in this setting. The ordered basis $\{\epsilon_1,\epsilon_2,\ldots,\epsilon_n\}$ gives rise to the positive system of type $A_{n-1}$ and its simple system in \cite{MR1920389} Appendix C. For $1 \leq i \leq n-1$, $\lambda=(\lambda_i)\in X^\ast(H_s)$ annihilates $(e_i-e_{i+1})^\vee=\epsilon_i-\epsilon_{i+1}$ if and only if $\lambda_i=\lambda_{i+1}$. A representative of the longest element is given by
		\[w=\left(\begin{array}{ccc}
			&  & \sqrt{-1}K_q\otimes 1 \\
			& K_{p-q}\otimes 1 &  \\
			\sqrt{-1}K_q \otimes 1&  & 
		\end{array}
		\right)\in\Uu(p,q,\bZ\left[1/2,\sqrt{-1}\right]).\]
		For $\lambda\in X^\ast(H\otimes_{\bZ\left[1/2\right]} \bZ\left[1/2,\sqrt{-1}\right])$, we have
		\[\bar{\lambda}=
		(-\lambda_1,-\lambda_2,\ldots,-\lambda_n)\]
		\[w\lambda=(\lambda_n,\lambda_{n-1},\ldots,\lambda_1).\]
		Hence the $\ast$-involution on $\Pi$ is given by
		\[\overline{w(e_i-e_{i+1})}=e_{2n-i}-e_{2n-i+1}.\]
		In particular, $\overline{w(e_i-e_{i+1})}=e_i-e_{i+1}$ if and only if $n$ is even and $i=\frac{n}{2}$. We thus get
		\[\Dyn\Uu(p,q)=\left\{\begin{array}{ll}
			\Spec\bZ\left[1/2\right]\coprod
			\Spec\bZ\left[1/2,\sqrt{-1}\right]^{\frac{n}{2}-1}
			&(n \mathrm{\ is\ even})\\
			\Spec\bZ\left[1/2,\sqrt{-1}\right]^{\frac{n-1}{2}}
			&(n \mathrm{\ is\ odd})
		\end{array}\right.\]
		For $\lambda\in X^\ast(H)$, $\bar{\lambda}=w\lambda$ if and only if $\lambda_i+\lambda_{n-i}=0$ for $1\leq i\leq n$. Since $\bar{w}w=I_n$, we have $\lambda(\bar{w}w)=1$.
	\end{ex}
	\begin{ex}[$\Uu^\ast(2n)$, $n\geq 1$]
		The image of $X_\ast(T\otimes_{\bZ\left[1/2\right]} \bZ\left[1/2,\sqrt{-1}\right])$ in $X_\ast(H\otimes_{\bZ\left[1/2\right]} \bZ\left[1/2,\sqrt{-1}\right])\cong\bZ^n$ is
		\[\{\mu=(\mu_i)\in \bZ^{2n}:~\mu_{i}
		=-\mu_{i+n}\ \mathrm{for}\ 1\leq i\leq n\}.\]
		The ordered basis 
		$\{\epsilon_i-\epsilon_{i+n}:~1\leq i\leq n\}$ of $X_\ast(T\otimes_{\bZ\left[1/2\right]}\bZ\left[1/2,\sqrt{-1}\right])$ attaches the following positive and simple systems:
		\[\Delta^+(\GL_{2n},H_s)
		=\{e_i-e_j:~1\leq i<j\leq 2n,i\leq n\}
		\cup\{-e_i+e_j:~n<i<j\leq 2n\}\]
		\[\Pi=\{e_i-e_{i+1}:~1\leq i\leq n-1\}
		\cup\{e_n-e_{2n}\}\cup
		\{-e_i+e_{i+1}:~n+1\leq i\leq 2n-1\}.\]
		For $1 \leq i \leq n-1$,
		$\lambda= (\lambda_i) \in X^\ast(H_s)$ annihilates $(e_i - e_{i+1})^\vee=\epsilon_i-\epsilon_{i+1}$ (resp.~$(-e_{i+n}+e_{i+n+1})^\vee=-\epsilon_{i+n}+\epsilon_{i+n+1}$) if and only if $\lambda_i = \lambda_{i+1}$ (resp.~$\lambda_{i+n} = \lambda_{i+n+1}$). A character $\lambda$ annihilates $(e_n - e_{2n})^\vee=\epsilon_n-\epsilon_{2n}$ if and only if $\lambda_n =\lambda_{2n}$. The longest element of the Weyl group is represented by $J_n\in\Sp(n,\bZ\left[1/2\right])$. For $\lambda\in X_\ast(H\otimes_{\bZ\left[1/2\right]}\bZ\left[1/2,\sqrt{-1}\right])$, we have
		\[w\lambda=\bar{\lambda}=
		(\lambda_{n+1},\lambda_{n+2},\ldots,\lambda_{2n},
		\lambda_1,\lambda_2,\ldots,\lambda_n).\]
		This shows $\Dyn\Uu^\ast(2n)\cong\Spec\bZ\left[1/2\right]^{2n-1}$. Since $J_n^2=-I_{2n}$, we have $\lambda(\bar{w}w)=(-1)^{\sum_{i=1}^{2n}\lambda_i}$. In particular, $\lambda(\bar{w}w)=1$ if and only if $\sum_{i=1}^{2n}\lambda_i$ is even.
	\end{ex}
	\begin{ex}[$\SO(2p,2q+1)$]
		Recall $H=T$ in this case. The ordered basis $\{\epsilon_1,\epsilon_2,\ldots,\epsilon_n\}$ gives rise to the positive system of type $B_n$ and its simple system in \cite{MR1920389} Appendix C.
		For $1 \leq i \leq n-1$, $\lambda = (\lambda_i) \in X^\ast(H_s)$ annihilates $(e_i-e_{i+1})^\vee=\epsilon_i-\epsilon_{i+1}$ if and only if $\lambda_i=\lambda_{i+1}$. A character $\lambda$ annihilates $e_n^\vee=2\epsilon_n$ if and only if $\lambda_n=0$.
		A representative of the longest element is given by
		\[w=\diag(1,-1,1,-1,\ldots,1,-1,(-1)^{p+q})\in
		\SO(2p,2q+1,\bZ\left[1/2\right]).\]
		The actions of $w$ and the conjugation on $X_\ast(H\otimes_{\bZ\left[1/2\right]}\bZ\left[1/2,\sqrt{-1}\right])$ are given by $w\lambda=\bar{\lambda}=-\lambda$. This shows $\Dyn\SO(2p,2q+1)\cong\Spec\bZ\left[1/2\right]^n$.We have $\lambda(\bar{w}w)=1$ for every character $\lambda\in X^\ast(H\otimes_{\bZ\left[1/2\right]} \bZ\left[1/2,\sqrt{-1}\right])$ since $\bar{w}w=I_{2n+1}$.
	\end{ex}
	\begin{ex}[$\Sp_n$, $n\geq 1$]
		Recall $H=T$ in this case. The ordered basis $\{\epsilon_1,\epsilon_2,\ldots,\epsilon_n\}$ gives rise to the positive system of type $C_n$ and its simple system in \cite{MR1920389} Appendix C. For $1\leq i \leq n-1$, $\lambda=(\lambda_i)\in X^\ast(H_s)$ annihilates $(e_i-e_{i+1})^\vee=\epsilon_i-\epsilon_{i+1}$ if and
		only if $\lambda_i=\lambda_{i+1}$. A character $\lambda$ annihilates $(2e_n)^\vee=\epsilon_n$ if and only if $\lambda_n=0$. The matrix $w=\sqrt{-1}S_{2n}'\in\Sp_n(\bZ\left[1/2,\sqrt{-1}\right])$ represents the longest element of the Weyl group. For a character $\lambda\in X^\ast(H)$, we have
		\[w\lambda=\bar{\lambda}=-\lambda.\]
		This shows $\Dyn\Sp_n\cong\Spec\bZ\left[1/2\right]^n$. We have $\lambda(\bar{w}w)=1$ for every character $\lambda\in X^\ast(H\otimes_{\bZ\left[1/2\right]} \bZ\left[1/2,\sqrt{-1}\right])$ since $\bar{w}w=I_{2n}$.
	\end{ex}
	\begin{ex}[$\Sp(p,q)$, $p+q=n\geq 1$]
		Recall $H=T$ in this case. The ordered basis $\{\epsilon_1,\epsilon_2,\ldots,\epsilon_n\}$ gives rise to the positive system of type $C_n$ and its simple system in \cite{MR1920389} Appendix C.
		We have the same formulas as the $\Sp_n$ case for the pairing of characters of $H_s$ with simple coroots. The matrix $J_n$ represents the longest element of the Weyl group. For a character $\lambda\in X^\ast(H)$, we have $w\lambda=\bar{\lambda}=-\lambda$.
		This shows $\Dyn\Sp(p,q)\cong\Spec\bZ\left[1/2\right]^n$. Observe that $\lambda(\bar{w}w)=(-1)^{\sum_{i=1}^{2n}\lambda_i}$ for a character $\lambda\in X^\ast(H\otimes_{\bZ\left[1/2\right]} \bZ\left[1/2,\sqrt{-1}\right])$ since $J_n^2=-I_{2n}$. In particular, $\lambda(\bar{w}w)=1$ if and only if $\sum_{i=1}^{2n}\lambda_i$ is even.
	\end{ex}
	\begin{ex}[$\SO(2p,2q)$, $n=p+q\geq 2$ even]\label{ex:soeveneveneven}
		Recall $H=T$ in this case. The ordered basis $\{\epsilon_1,\epsilon_2,\ldots,\epsilon_n\}$ gives rise to the positive system of type $D_n$ and its simple system in \cite{MR1920389} Appendix C. For $1\leq i \leq n-1$, $\lambda=(\lambda_i)\in X^\ast(H_s)$ annihilates $(e_i-e_{i+1})^\vee=\epsilon_i-\epsilon_{i+1}$ if and
		only if $\lambda_i=\lambda_{i+1}$. A character $\lambda$ annihilates $(e_{n-1}+e_n)^\vee$ if and only if $\lambda_n=-\lambda_{n-1}$. A representative of the longest element is given by
		\[w=\diag(1,-1,1,-1,\ldots,1,-1)\in
		\SO(2p,2q,\bZ\left[1/2\right]).\]
		For $\lambda\in X_\ast(H\otimes_{\bZ\left[1/2\right]}\bZ\left[1/2,\sqrt{-1}\right])$, we have $\bar{\lambda}=w\lambda=-\lambda$. In particular, $\Dyn\SO(2p,2q)\cong\Spec\bZ\left[1/2\right]^n$. We have $\lambda(\bar{w}w)=1$ for every character $\lambda\in X^\ast(H\otimes_{\bZ\left[1/2\right]} \bZ\left[1/2,\sqrt{-1}\right])$ since $\bar{w}w=I_{2n}$.
	\end{ex}
	\begin{ex}[$\SO(2p,2q)$, $n=p+q\geq 1$ odd]
		Define positive and simple systems in a similar way to Example \ref{ex:soeveneveneven}. We have the same formulas as Example \ref{ex:soeveneveneven} for the pairing of characters of $H_s$ with simple coroots. A representative of the longest element is given by
		\[w=\diag(1,-1,1,-1,\ldots,1,-1,1,1)\in
		\SO(2p,2q,\bZ\left[1/2\right]).\]
		For $\lambda\in X_\ast(H\otimes_{\bZ\left[1/2\right]}\bZ\left[1/2,\sqrt{-1}\right])$, we have
		\[\bar{\lambda}=-\lambda\]
		\[w\lambda=(-\lambda_1,-\lambda_2,\ldots,-\lambda_{n-1},
		\lambda_n).\]
		Hence the $\ast$-action on the simple system is given by
		\[\begin{array}{ll}
			\overline{w(e_i-e_{i+1})}=e_i-e_{i+1}&(1\leq i\leq n-2)
		\end{array}\]
		\[\overline{w(e_{n-1}-e_n)}=e_{n-1}+e_n\]
		\[\overline{w(e_{n-1}+e_n)}=e_{n-1}-e_n.\]
		This shows
		\[\Dyn\SO(2p,2q)\cong\left\{\begin{array}{ll}
			\Spec\bZ\left[1/2\right]^{n-2}\coprod
			\Spec\bZ\left[1/2,\sqrt{-1}\right]
			&(n\geq 3)\\
			\emptyset
			&(n=1).
		\end{array}\right.\]
		A character $\lambda\in X^\ast(H\otimes_{\bZ\left[1/2\right]} \bZ\left[1/2,\sqrt{-1}\right])$ satisfies $\bar{\lambda}=w\lambda$ if and only if $\lambda_n=0$. In this case, we have $\lambda(\bar{w}w)=1$ since $\bar{w}w=I_{2n}$.
	\end{ex}
	\begin{ex}[$\SO(2p+1,2q+1)$, $n=p+q\in 2\bN$]\label{ex:sooddoddeven}
		When $n=0$ (equivalently, $p=q=0$), we have an isomorphism
		\[\bG_m\cong\SO(1,1);~t\mapsto\left(\begin{array}{cc}
			\frac{t+t^{-1}}{2} & \frac{t-t^{-1}}{2} \\
			\frac{t-t^{-1}}{2} & \frac{t+t^{-1}}{2}
		\end{array}
		\right)\]
		over $\bZ\left[1/2\right]$. Therefore we have no roots, $\Dyn\SO(1,1)=\emptyset$. We may put $w=e$. Then we have $\bar{\lambda}=w\lambda=\lambda$ and $\lambda(\bar{w}w)=1$ for any $\lambda\in X^\ast(\SO(1,1))\cong\bZ$.
		
		In the rest, assume $n\geq 2$. The image of $X_\ast(T\otimes_{\bZ\left[1/2\right]} \bZ\left[1/2,\sqrt{-1}\right])$ in $X_\ast(H\otimes_{\bZ\left[1/2\right]} \bZ\left[1/2,\sqrt{-1}\right])\cong\bZ^{n+1}$ is
		\[\{\mu=(\mu_i)\in \bZ^n:~\mu_{p+1}=0\}.\]
		Under this identification, 
		\[\{\epsilon_1,\epsilon_2,\ldots,\epsilon_p,
		\epsilon_{p+2},\epsilon_{p+3},\ldots,\epsilon_{n+1}\}\]
		is an ordered free basis of $X_\ast(T\otimes_{\bZ\left[1/2\right]} \bZ\left[1/2,\sqrt{-1}\right])$. The attached positive and simple systems are given by
		\[\Delta^+=\{e_i\pm e_j:~1\leq i<j\leq n+1,i\neq p+1\}\cup\{\pm e_{p+1} +e_j:~p+1<j\leq n+1\}\]
		\[\Pi=\left\{\begin{array}{ll}
			\{e_i-e_{i+1}:~1\leq i\leq n,i\neq p,p+1\}\\
			\cup\{e_p-e_{p+2},e_{n+1}\pm e_{p+1}\}&(p\neq 0,n);\\
			\{e_i-e_{i+1}:~2\leq i\leq n\}\cup\{e_{n+1}\pm e_1\}
			&(p=0);\\
			\{e_i-e_{i+1}:~1\leq i\leq n\}
			\cup\{e_n+e_{n+1}\}&(p=n).
		\end{array}\right.\]
		For $1\leq i \leq n$, $\lambda=(\lambda_i)\in X^\ast(H_s)$ annihilates $(e_i-e_{i+1})^\vee=\epsilon_i-\epsilon_{i+1}$ if and
		only if $\lambda_i=\lambda_{i+1}$. For $p\neq 0,n$, $\lambda$ annihilates $(e_p-e_{p+2})^\vee=\epsilon_p-\epsilon_{p+2}$ (resp.~
		$(e_{n+1}-e_{p+1})^\vee=\epsilon_{n+1}-\epsilon_{p+1}$, $(e_{n+1}+e_{p+1})^\vee=\epsilon_{n+1}+\epsilon_{p+1}$) if and only if $\lambda_p=\lambda_{p+2}$ (resp.~$\lambda_{n+1}=\lambda_{p+1}$, $\lambda_{n+1}=-\lambda_{p+1}$). A
		character $\lambda$ annihilates $(e_n+e_{n+1})^\vee$ (resp.\ $(e_{n+1}-e_1)^\vee$, $(e_{n+1}+e_1)^\vee$) if and only if $\lambda_{n+1}=-\lambda_n$ (resp.~$\lambda_{n+1}=\lambda_1$, $\lambda_{n+1}=-\lambda_1$). The longest element of the Weyl group is represented by
		\[w=\diag(\overbrace{1,-1,1,-1,\ldots,1,-1}^{2p},(-1)^p,(-1)^q,\overbrace{1,-1,1,-1,\ldots,1,-1}^{2q}).\]
		For $\lambda\in X_\ast(H\otimes_{\bZ\left[1/2\right]}\bZ\left[1/2,\sqrt{-1}\right])$, we have
		\[w\lambda=\bar{\lambda}=(-\lambda_1,-\lambda_2,\ldots,-\lambda_p,\lambda_{p+1},-\lambda_{p+2},\ldots,-\lambda_{n+1}).\]
		This shows
		\[\Dyn\SO(2p+1,2q+1) \cong \Spec\bZ\left[1/2\right]^{n+1}.\]
		We have $\lambda(\bar{w}w)=1$ for every character $\lambda\in X^\ast(H\otimes_{\bZ\left[1/2\right]} \bZ\left[1/2,\sqrt{-1}\right])$ since $\bar{w}w=I_{2n+2}$.
	\end{ex}
	\begin{ex}[$\SO(2p+1,2q+1)$, $n=p+q\in 2\bN+1$]
		Define $\Delta^+$, $\Pi$, and $w$ in a similar way to the above. We have the same formulas as Example \ref{ex:sooddoddeven} for the pairing of characters of $H_s$ with simple coroots. For $\lambda\in X_\ast(H\otimes_{\bZ\left[1/2\right]}\bZ\left[1/2,\sqrt{-1}\right])$, we have
		\[w\lambda=-\lambda\]
		\[\bar{\lambda}=(-\lambda_1,-\lambda_2,\ldots,-\lambda_p,\lambda_{p+1},-\lambda_{p+2},\ldots,-\lambda_{n+1}).\]
		The $\ast$-involution on $\Pi$ is given as follows:
		\begin{description}
			\item[$pq\neq 0$]
			\[\begin{array}{ll}
				\overline{w(e_i-e_{i+1})}=e_i-e_{i+1}
				&(1\leq i\leq n,i\neq p,p+1)
			\end{array}\]
			\[\overline{w(e_p-e_{p+2})}=e_p-e_{p+2}\]
			\[\overline{w(e_{n+1}-e_{p+1})}=e_{n+1}+e_{p+1}\]
			\[\overline{w(e_{n+1}+e_{p+1})}=e_{n+1}-e_{p+1}.\]
			\item[$p=0$]
			\[\begin{array}{ll}
				\overline{w(e_i-e_{i+1})}=e_i-e_{i+1}
				&(2\leq i\leq q)
			\end{array}\]
			\[\overline{w(e_{q+1}-e_1)}=e_{q+1}+e_1\]
			\[\overline{w(e_{q+1}+e_1)}=e_{q+1}-e_1.\]
			\item[$q=0$]
			\[\begin{array}{ll}
				\overline{w(e_i-e_{i+1})}=e_i-e_{i+1}
				&(1\leq i\leq p-1)
			\end{array}\]
			\[\overline{w(e_p-e_{p+1})}=e_p+e_{p+1}\]
			\[\overline{w(e_p+e_{p+1})}=e_p-e_{p+1}.\]
		\end{description}
		This shows
		\[\Dyn\SO(2p+1,2q+1)\cong\Spec\bZ\left[1/2\right]^{n-1}
		\coprod\Spec\bZ\left[1/2,\sqrt{-1}\right].\]
		A character $\lambda\in X^\ast(H\otimes_{\bZ\left[1/2\right]} \bZ\left[1/2,\sqrt{-1}\right])$ satisfies $\bar{\lambda}=w\lambda$ if and only if $\lambda_{p+1}=0$. In this case, we have $\lambda(\bar{w}w)=1$ since $\bar{w}w=I_{2n+2}$.
	\end{ex}
	\begin{ex}[$\SO^\ast(4n)$, $n\geq 1$]
		Recall that $H=T$ in this case. The ordered basis $\{\epsilon_1,\epsilon_2,\ldots,\epsilon_{2n}\}$ gives rise to the positive system of type $D_{2n}$ and its simple system in \cite{MR1920389} Appendix C. For $1\leq i \leq n-1$, $\lambda=(\lambda_i)\in X^\ast(H_s)$ annihilates $(e_i-e_{i+1})^\vee=\epsilon_i-\epsilon_{i+1}$ if and
		only if $\lambda_i=\lambda_{i+1}$. A character $\lambda$ annihilates $(e_{n-1}+e_n)^\vee=\epsilon_{n-1}+\epsilon_n$ if and only if $\lambda_n=-\lambda_{n-1}$.
		The matrix
		\[w=\sqrt{-1}\otimes\sqrt{-1}\diag(I_{2n},-I_{2n})\in
		\SO^\ast\left(4n,\bZ\left[1/2,\sqrt{-1}\right]\right)\]
		is a representative of the longest element of the Weyl group. For a character $\lambda=(\lambda_i)$, we have
		\[\bar{\lambda}=w\lambda=-\lambda.\]
		This shows that $\Dyn\SO^\ast(4n)\cong\Spec\bZ\left[1/2\right]^{2n}$. A character $\lambda\in X^\ast(H\otimes_{\bZ\left[1/2\right]} \bZ\left[1/2,\sqrt{-1}\right])$ satisfies $\lambda(\bar{w}w)=1$ if and only if $\sum_{i=1}^{2n}\lambda_i$ is even since $\bar{w}w=-w^2=-I_{4n}$.
	\end{ex}
	\begin{ex}[$\SO^\ast(4n+2)$, $n\geq 0$]
		Define positive and simple systems in a similar way to the $\SO^\ast(4n)$ case. We have the same formulas as the $\SO^\ast(4n)$ case for the pairing of characters of $H_s$ with simple coroots. The matrix
		\[w=\diag(\sqrt{-1}\otimes \sqrt{-1}I_{2n},1,-\sqrt{-1}\otimes \sqrt{-1}I_{2n},1)\in \SO^\ast(4n+2,\bZ\left[1/2,\sqrt{-1}\right])\]
		is a representative of the longest element of the Weyl group. For a character $\lambda=(\lambda_i)$, we have
		\[\bar{\lambda}=-\lambda\]
		\[w\lambda=(-\lambda_1,-\lambda_2,\ldots,-\lambda_{2n},
		\lambda_{2n+1}).\]
		In particular, the action of $\bZ/2\bZ$ on the simple system is given by
		\[\overline{w(e_i-e_{i+1})}=\left\{\begin{array}{ll}
			e_i-e_{i+1} & (1\leq i\leq 2n-1)\\
			e_{2n}+e_{2n+1}&(i=2n)
		\end{array}\right.\]
		\[\overline{w(e_{2n}+e_{2n+1})}=e_{2n}-e_{2n+1}.\]
		This implies
		\[\Dyn\SO^\ast(4n+2)\cong\left\{\begin{array}{ll}
			\Spec\bZ\left[1/2\right]^{2n-1}\coprod
			\Spec\bZ\left[1/2,\sqrt{-1}\right]
			&(n\geq 1)\\
			\emptyset
			&(n=0).
		\end{array}\right.\]
		A character $\lambda\in X^\ast(H\otimes_{\bZ\left[1/2\right]} \bZ\left[1/2,\sqrt{-1}\right])$ satisfies $\bar{\lambda}=w\lambda$ if and only if $\lambda_{2n+1}=0$. We have $\lambda(\bar{w}w)=1$ if and only if $\sum_{i=1}^{2n}\lambda_i$ is even since $\bar{w}w=\diag(-I_{2n},1,-I_{2n},1)$.
	\end{ex}
	\subsection{Restriction case}
	In this section, we run the program for group schemes obtained by the Weil restriction as a generalization of the models of complex groups.
	
	Fix a Galois extension $k\to k'$ of Galois group $\Gamma$ such that $\Spec k'$ is connected. Let $(G',H')$ be a split reductive group scheme over $k'$. We plan to compute $\Dyn\Res_{k'/k} G'$, study the character group of parabolic subgroups of $\Res_{k'/k} G'$ (whose types are defined over $k$), and determine when $\cL_\lambda$ admits a descent datum by a minor modification of the former section:
	\begin{enumerate}
		\item Fix s simple system $\vec{\Pi}$ of $((\Res_{k'/k} G')\otimes_k k',(\Res_{k'/k} H')\otimes_k k')$.
		\item Find $\vec{w}_\sigma$ such that ${}^\sigma\vec{\Pi}=\vec{w}_\sigma\vec{\Pi}$ for each $\sigma\in\Gamma$.
		\item Compute $\Dyn \Res_{k'/k} G'$.
		\item For $x\in(\type \Res_{k'/k} G')(k)$, find a parabolic subgroup $(\Res_{k'/k} G')\otimes_k k'$ whose type is $x|_{\Spec k'}$.
		\item Compute $\langle\alpha^\vee,\lambda\rangle$ for simple roots $\alpha$ and characters $\lambda$ to determine the character group of parabolic subgroups of Stage 5.
		\item Study when the equality ${}^\sigma\lambda=\vec{w}_\sigma\lambda$ holds.
		\item Prove that if the equality of 6 holds, $\beta_\lambda$ is trivial.
	\end{enumerate}
	For the complete classification, see Theorem \ref{thm:uniqueness} and Remark \ref{rem:associatedbundle}.
	
	Without loss of generality, we may choose a reductive group scheme $G$ and a maximal torus $H$ such that there are compatible isomorphisms
	\[G'\cong G\otimes_k k'\]
	\[H'\cong H\otimes_k k'.\]
	Notice that
	\[((\Res_{k'/k} (G\otimes_k k'))\otimes_k k',(\Res_{k'/k} (H\otimes_k k'))\otimes_k k')\]
	is split. Under the identification
	\[(\Res_{k'/k} (H\otimes_k k'))\otimes_k k'\cong \prod_{\tau\in\Gamma} {}^\tau(H\otimes_k k')=\prod_{\tau\in\Gamma} H\otimes_k k',\]
	the set of roots is given by
	\begin{flalign*}
		&\Delta((\Res_{k'/k} (G\otimes_k k'))\otimes_k k',
		(\Res_{k'/k} (H\otimes_k k'))\otimes_k k')\\
		&=\{\alpha\sigma\in \prod_{\tau\in\Gamma} 
		X^\ast(H\otimes_k k'):~\sigma\in\Gamma,
		\alpha\in\Delta(G\otimes_k k',H\otimes_k k')\}.
	\end{flalign*}
	It is clear that $(\alpha\sigma)^\vee=\alpha^\vee\sigma$. Choose a simple system $\Pi$ of $(G\otimes_k k',H\otimes_k k')$. Define a simple system of $(\Res_{k'/k} G\otimes_k k')\otimes_k k'$ by
	\[\vec{\Pi}=\coprod_\tau \Pi\tau=
	\{\alpha\tau\in\prod_{\tau\in\Gamma} X^\ast(H\otimes_k k'):~\alpha\in\Pi,\tau\in\Gamma\}.\]
	
	For $\sigma\in\Gamma$, choose $w_\sigma\in N_{G(k')}(H\otimes_k k')$ such that ${}^\sigma \Pi=w_\sigma\Pi$. Define $\vec{w}_\sigma\in \prod_{\tau\in\Gamma} G(k')$ by $\vec{w}_\sigma=\sigma((w_\tau))(w_\tau)^{-1}=(\sigma(w_{\sigma^{-1}\tau})^{-1}w_\tau)$. We remark that we have $\sigma(w_{\sigma^{-1}\tau})^{-1}w_\tau=w_\sigma$ in $W(G,H)(k')$ by the cocycle condition. For $\alpha\in\Pi$ and $\sigma,\tau\in\Gamma$, we have
	\[{}^{\sigma^{-1}} (\vec{w}_\sigma(\alpha\tau))
	={}^{\sigma^{-1}}(w_\sigma\alpha\tau)
	={}^{\sigma^{-1}}(w_\sigma\alpha)\sigma^{-1}\tau,\]
	which implies the following results:
	\begin{thm}\label{thm:computedyngrescase}
		\begin{enumerate}
			\renewcommand{\labelenumi}{(\arabic{enumi})}
			\item We have ${}^\sigma \vec{\Pi}=\vec{w}_\sigma\vec{\Pi}$ and $\Dyn\Res_{k'/k} G\cong \coprod_{\alpha\in\Pi} \Spec k'$.
			\item The set $(\type \Res_{k'/k} G)(k)$ is bijective to the power set of \[\{\{w_\tau^{-1}{}^\tau\alpha\tau\in \prod_{\tau\in\Gamma} X^\ast(H\otimes_k k')
			:~\tau\in\Gamma\}\in\Set:~\alpha\in\Pi\}\cong\Pi.\]
			\item For each subset $\Pi'\subset \Pi$, let $P'_{\Pi'}\subset G'$ be the parabolic subgroup corresponding to $\Pi'$. Then the set of roots of $(w_{\tau})^{-1}(\prod_{\tau\in\Gamma} {}^\tau P'_{\Pi'}) (w_\tau)$ contain $\vec{\Pi}$. Moreover, we have $t((w_{\tau})^{-1}(\prod_{\tau\in\Gamma} {}^\tau P'_{\Pi'}) (w_\tau))=\Pi'$ in $(\type \Res_{k'/k} G')(k)$ under the identification of (2).
		\end{enumerate}
	\end{thm}
	\begin{proof}
		We only prove (3). Recall that we have $(\Res_{k'/k} P'_{\Pi'})\otimes_k k'\cong\prod_{\tau\in\Gamma} {}^\tau P'_{\Pi'}$. If we write $\Phi$ for the set of roots of $P'_{\Pi'}$ relative to $H\otimes_k k'$, then that of
		\[(w_{\tau})^{-1}(\prod_{\tau\in\Gamma} {}^\tau P'_{\Pi'}) (w_\tau)\]
		relative to $\prod_{\tau\in\Gamma} H\otimes_k k'$
		is $\coprod_{\tau\in\Gamma} w^{-1}_\tau {}^\tau \Phi\tau$. In particular, $(w_{\tau})^{-1}(\prod_{\tau\in\Gamma} {}^\tau P'_{\Pi'}) (w_\tau)$ contains all simple roots of $\prod_{\tau\in\Gamma} G\otimes_k k'$. Moreover, we have
		\[t((w_{\tau})^{-1}(\prod_{\tau\in\Gamma} {}^\tau P'_{\Pi'}) (w_\tau))\\
		=\coprod_{\tau\in\Gamma} w_\tau^{-1} {}^\tau \Pi'\tau.\]
		This completes the proof.
	\end{proof}
	Let $P'_{\Pi'}$ be as in Theorem \ref{thm:computedyngrescase} (3). Then we can determine the character group of $(w_{\tau})^{-1}(\prod_{\tau\in\Gamma} {}^\tau P'_{\Pi'}) (w_\tau)$ by the following obvious equality:
	\begin{prop}\label{prop:extcondition}
		For $(\lambda_\tau)\in \prod_{\tau\in\Gamma} X^\ast(H\otimes_k k')$, $\alpha\in \Pi$, and $\sigma\in\Gamma$, we have
		\[\langle(w_\sigma^{-1}{}^\sigma\alpha\sigma)^\vee,
		(\lambda_\tau)\rangle
		=\langle{}^\sigma\alpha^\vee,
		w_\sigma\lambda_\sigma\rangle.\]
	\end{prop}
	We next see when $(\lambda)_\tau$ is invariant under the $\ast$-action:
	\begin{prop}\label{prop:conjcondition}
		Let $(\lambda_\tau)\in \prod_{\tau\in\Gamma} X^\ast(H\otimes_k k')$. Then ${}^\sigma(\lambda_\tau)=\vec{w}_\sigma(\lambda_\tau)$ for all $\sigma\in\Gamma$ if and only if $\lambda_\sigma=w_\sigma^{-1} {}^\sigma\lambda_e$.
	\end{prop}
	\begin{proof}
		Recall
		\[{}^\sigma(\lambda_\tau)=
		({}^\sigma\lambda_{\sigma^{-1}\tau})\]
		\[\vec{w}_\sigma(\lambda_\tau)=(w_\sigma \lambda_\tau).\]
		The ``only if'' direction follows by putting $\tau=\sigma$. The converse follows from the cocycle condition of $w_\sigma$:
		\[{}^\sigma\lambda_{\sigma^{-1}\tau}
		={}^\sigma (w^{-1}_{\sigma^{-1}\tau} {}^{\sigma^{-1}\tau}\lambda_e)
		=\sigma(w_{\sigma^{-1}\tau})^{-1} {}^\tau \lambda_e
		=w_\sigma w_\tau^{-1} {}^\tau \lambda_e
		=w_\sigma \lambda_\tau.\]
	\end{proof}
	\begin{rem}
		Under the equivalent conditions of Proposition \ref{prop:conjcondition}, we have
		$\langle(w_\sigma^{-1}{}^\sigma\alpha\sigma)^\vee,
		(\lambda_\tau)\rangle
		=\langle\alpha^\vee,\lambda_e\rangle$ (see Proposition \ref{prop:extcondition}).
	\end{rem}
	The existence problem of descent data finishes by the following observation:
	\begin{prop}
		Let $(\lambda_\tau)\in \prod_{\tau\in\Gamma} X^\ast(H\otimes_k k')$ be a character satisfying the equivalent conditions of Proposition \ref{prop:conjcondition}. Then $\beta_{(\lambda_\tau)}$ is trivial.
	\end{prop}
	\begin{proof}
		The assertion follows from
		\[\begin{split}
			\vec{w}_{\sigma\sigma'}^{-1}\sigma(\vec{w}_{\sigma'})\vec{w}_{\sigma}
			&=(w_\tau) \sigma\sigma'((w_\tau))^{-1}
			\sigma(\sigma'((w_\tau))(w_\tau)^{-1})
			\sigma((w_\tau))(w_\tau)^{-1}\\
			&=(w_\tau) \sigma\sigma'((w_\tau))^{-1}
			\sigma(\sigma'((w_\tau)))\sigma((w_\tau))^{-1}
			\sigma((w_\tau))(w_\tau)^{-1}\\
			&=(e)
		\end{split}\]
		for $\sigma,\sigma'\in\Gamma$. In fact, we have
		\[\beta_{(\lambda_\tau)}(\sigma,\sigma')
		=(\lambda_\tau)(\vec{w}_{\sigma\sigma'}^{-1}\sigma(\vec{w}_{\sigma'})\vec{w}_{\sigma})=(\lambda_\tau)((e))=1.\]
	\end{proof}
	
	To end this paper, we explain the compatibility of our choice of maximal tori and simple roots with Section 3.1.
	\begin{prop}
		Let $X$ be a $k$-scheme. Let $u:X\to \Res_{k'/k}(X\otimes_k k')$ be the unit. Then the composite map
		\[X\otimes_k k'\overset{u\otimes_k k'}{\to}
		(\Res_{k'/k}(X\otimes_k k'))\otimes_k k'
		\overset{\phi}{\cong} \prod_{\tau\in\Gamma} X\otimes_k k'\]
		coincides with the diagonal map.
	\end{prop}
	\begin{proof}
		Let $A'$ be a $k'$-algebra. Write $u:A'\to A'\otimes_k k'$ for the unit. Then the composite $k$-algebra homomorphism $\delta_{A'}\circ u:A'\to \prod_{\tau\in\Gamma} A'$ is clearly equal to the diagonal map. Apply $X$ to deduce the assertion.
	\end{proof}
	\begin{cor}
		Let $T$ be a commutative group scheme over $k$. Then the composite map
		\[\begin{split}
			X_\ast(T\otimes_k k')
			&\overset{u\otimes_k k'}{\to}
			X_\ast((\Res_{k'/k}(X\otimes_k k'))\otimes_k k')\\
			&\overset{\phi}{\cong} X_\ast(\prod_{\tau\in\Gamma} T\otimes_k k')\\
			&\cong \prod_{\tau\in\Gamma} X_\ast(T\otimes_k k')
		\end{split}\]
		is the diagonal map.
	\end{cor}
	Assume the following conditions:
	\begin{enumerate}
		\renewcommand{\labelenumi}{(\roman{enumi})}
		\item $\Gamma=\bZ/2\bZ$.
		\item $2$ is a unit of $k$.
	\end{enumerate}
	Regard $G$ as a symmetric subgroup of $\Res_{k'/k} (G\otimes_k k')$ by the pointwise conjugation. Then $\Res_{k'/k} (H\otimes_k k')=Z_{\Res_{k'/k} G}(H)$ is a fundamental Cartan subgroup since $H$ is a maximal torus of $G$. Choose an ordered basis $\{\epsilon_i\}$ of $X_\ast(H\otimes_k k')$. Let $\Pi$ be the associated simple system of $(G\otimes_k k',H\otimes_k k')$.
	\begin{cor}
		The simple system of $(\Res_{k'/k} G\otimes_k k')\otimes_k k'\cong (G\otimes_k k')^2$ attached to the ordered basis $\{\epsilon_i\}$ via Construction \ref{cons:simplesys} is $\Pi\times\{0\}\coprod\{0\}\times\Pi$.
	\end{cor}
	\begin{rem}
		Another possible approach to this setting is to set $\coprod_{\tau\in\Gamma} {}^\tau \Pi \tau$ for a simple system. Then Example \ref{ex:quasisplit} and Example \ref{ex:defoverk} are available. We adopted the above choice because of the future applications to $A_\fq(\lambda)$-modules.
	\end{rem}

	\section*{Acknowledgments}
	The author is grateful to Fabian Januszewski for his generosity in sharing his idea on Galois descent in representation theory of real reductive Lie groups.
	
	Thanks Hiraku Atobe for helpful comments.
	
	This work was supported by JSPS KAKENHI Grant Number 21J00023.

\end{document}